\documentclass[11pt,reqno]{amsart}
\usepackage{amsmath,amssymb,latexsym,esint,cite,mathrsfs}
\usepackage{verbatim,wasysym}
\usepackage[left=1.8cm,right=1.8cm,top=2.4cm,bottom=2.4cm]{geometry}
\usepackage{tikz,enumitem,graphicx, subfig, microtype, color}
\usepackage{epic,eepic}

\usepackage[colorlinks=true,urlcolor=blue, citecolor=red,linkcolor=blue,
linktocpage,pdfpagelabels, bookmarksnumbered,bookmarksopen]{hyperref}
\usepackage[hyperpageref]{backref}
\usepackage[english]{babel}
\usepackage{appendix}
\usepackage{bm}
\usepackage{amsmath}
\DeclareMathOperator{\diag}{diag}%
\usepackage{float}
\usepackage{amsmath,amssymb}
\usepackage{array,booktabs,tabularx}
\usepackage{placeins}
\numberwithin{equation}{section}

\newtheorem{thm}{Theorem}[section]
\newtheorem{lem}[thm]{Lemma}
\newtheorem{cor}[thm]{Corollary}
\newtheorem{Prop}[thm]{Proposition}
\newtheorem{Def}[thm]{Definition}
\newtheorem{Rem}[thm]{Remark}

\begin{document}
	\title[Spectral analysis of multi-bubble solutions to the Brezis-Nirenberg problem]
	{Spectral analysis of multi-bubble solutions to the Brezis-Nirenberg problem}

	\author[M. Yang]{Minbo Yang}
	\address{\noindent Minbo Yang  \newline
		School of Mathematical Sciences, Zhejiang Normal University,\newline
		Jinhua 321004, Zhejiang, People's Republic of China.}\email{mbyang@zjnu.edu.cn}
	
	\author[S. Zhao]{Shunneng Zhao}
	\address{\noindent Shunneng Zhao  \newline
	School of Mathematical Sciences, Zhejiang Normal University,\newline
		Jinhua 321004, Zhejiang, People's Republic of China.}
	\email{snzhao@zjnu.edu.cn}


	
	\allowdisplaybreaks
	
	\begin{abstract}
		{\small
We investigate weighted spectral limits for separated multi-bubble configurations and refined spectral asymptotics
for the Brezis--Nirenberg problem
\[
    -\Delta u=u^p+\varepsilon u^q, \quad u>0, \quad\text{in }\Omega,
    \qquad
    u=0 \quad \text{on }\partial\Omega,
\]
where $\Omega$ is a smooth bounded domain in $\mathbb R^N$ with $N\geq3$,
$p=(N+2)/(N-2)$, $1\leq q<p$ and $\varepsilon$ is a positive parameter.
We first determine the limits of the ordered weighted eigenvalues
associated with families of positive functions admitting an
$m$-bubble decomposition at distinct interior points with a
remainder vanishing in $H_0^1(\Omega)$.
The limiting eigenvalues are those of the single-bubble problem,
with all multiplicities multiplied by $m$.
We then refine this analysis for positive multi-bubble solutions
of this problem with comparable concentration scales, deriving
precise expansions of the first $(N+2)m$ eigenvalues and
characterizing their eigenfunctions when
$1<q<3$ if $N=4$ or $1\leq q<p$ if $N\geq5$.
The leading corrections are determined by explicit symmetric
matrices involving $q$, the concentration parameters, and the
Green and Robin functions.
In particular, we prove that the dilation matrix is positive definite.
Together, these spectral results yield Morse index bounds for
sufficiently small $\varepsilon$, with an exact formula when the
matrix governing the translation modes is invertible.
}
	\end{abstract}

	\vspace{3mm}
	
	\maketitle
	\section{Motivation and main results}
	In this paper, we are concerned with
the Brezis-Nirenberg problem
\begin{equation}\label{ele-1.1}
\left\lbrace
\begin{aligned}
    &-\Delta u=u^{p}+\epsilon u^{q} \quad \quad\mbox{in}\quad \Omega,\\
    &u>0\quad \quad \quad \quad\quad\quad\quad\hspace{1mm}  \mbox{in}\quad\Omega,\\
    &u=0\quad\quad\quad \quad\quad \quad\quad \hspace{0.5mm}  \mbox{on}\hspace{2mm}\partial\Omega,
   \end{aligned}
\right.
\end{equation}
where $\Omega$ is a smooth bounded domain in $\mathbb{R}^N$ with $N\geq3$, $p=(N+2)/(N-2)$, $1\leq q<p$ and $\epsilon>0$ is a small parameter. We are interested in positive multi-bubble solutions $\epsilon\rightarrow0$, and in the asymptotic behavior of their linearized eigenvalues and eigenfunctions.

We first recall some known results for $q=1$.
Brezis and Nirenberg \cite{Brezis-1983}
showed that if $N\geq4$, problem \eqref{ele-1.1} has a
	nontrivial solution for $\epsilon\in(0,\lambda_1)$,  where $\lambda_1$ is the first eigenvalue of $-\Delta$ with Dirichlet boundary
	condition;
If $\epsilon=0$, problem~\eqref{ele-1.1} becomes again much more delicate. Pohozaev discovered in \cite{Pohozaev-1965} that \eqref{ele-1.1} does not have a solution if $\Omega$ is a star-shaped domain.  Bahri and Coron \cite{Bahri-1988} gave an existence result of a positive solution to problem \eqref{ele-1.1} when $\Omega$ has a
	nontrivial topology, while Ding \cite{Ding1989} constructed a solution to  \eqref{ele-1.1} in the case where $\Omega$ is contractible.

Once existence is known, a natural question is how these solutions behave as $\epsilon\rightarrow0$.
For positive solutions whose Sobolev quotient tends to the best
Sobolev constant, Han \cite{HANZCHAO} and Rey~\cite{Rey-1989} showed that
concentration occurs at a single interior critical point of the
Robin function, and determined the concentration rate and the
limiting profile. Conversely, Rey~\cite{Rey-1990} constructed
single-bubble solutions near nondegenerate critical points of
the Robin function for $N\geq5$.
In dimension three, del Pino, Dolbeault and Musso~\cite{DDM}
constructed positive single-bubble solutions as
$\lambda\downarrow\lambda_0>0$, under suitable assumptions
at a zero critical point of the Robin function of
$-\Delta-\lambda_0$.
The four-dimensional construction
was considered by Pistoia and Rocci~\cite{PR}. For the blow-up analysis of bubble solutions, we also refer to \cite{HLWW,LVWW}.
The spectral properties of single-bubble solutions were studied
by Takahashi~\cite{T-1}. He obtained asymptotic estimates
for the eigenvalues and the corresponding eigenfunctions of
the linearized problem. In particular, for $N\geq6$, the first
corrections to the translation eigenvalues are determined by
the Hessian of the Robin function at the concentration point.

For multi-bubble solutions, Musso and Pistoia~\cite{Musso-Pistoia-2002}
proved existence for $N\geq5$ under suitable stable critical
point conditions on a reduced function involving the Green
and Robin functions. In $N=3$, Musso and
Salazar~\cite{MS-2018} constructed multi-bubble solutions as the
linear parameter approaches a suitable positive value.
More recently, Pistoia, Rago and Vaira~\cite{P-R-V} established
a four-dimensional existence result under a positive stable
critical set condition on the smallest eigenvalue of the
Green-Robin interaction matrix. In this case, the bubble
widths are exponentially small in $1/\varepsilon$. For $N\geq6$, Cao, Luo and Peng~\cite{CLP} established local uniqueness and
counting results under suitable assumptions on the reduced
data. Further information on multi-bubble solutions has been obtained
through refined concentration analysis.
Considered the critical elliptic equation
\begin{equation}\label{ele-1.1-1}
\left\lbrace
\begin{aligned}
    &-\Delta u+\epsilon V(x)u^{q}=u^{p}\quad \quad\mbox{in}\quad \Omega,\\
    &u>0\quad \quad \quad \quad\quad\quad\quad\quad\quad\hspace{0.5mm}  \mbox{in}\quad\Omega,\\
    &u=0\quad\quad\quad \quad\quad \quad\quad\quad\quad  \mbox{on}\hspace{2mm}\partial\Omega.
   \end{aligned}
\right.
\end{equation}
K\"{o}nig and Laurain \cite{K-L-2024} in the case $q=1$ in \eqref{ele-1.1-1} with $V\in C^1(\bar{\Omega})$ and $V<0$ on $\bar{\Omega}$ proved that \eqref{ele-1.1-1} concentration occurs at finitely many distinct
interior points and obtained precise information on the
concentration rates and limiting configurations in terms of
the Green and Robin functions and the potential $V$. In dimension $N=3$,
the authors in \cite{K-L-2025} gives a complete picture of blow-up phenomena in the Brezis-Nirenberg framework. The choice $V\equiv-1$ in \cite{K-L-2024} corresponds, after a constant normalization of the solution, to the case $q=1$ of our problem.

We next discuss nonlinear perturbations $1<q<p$.
Brezis and Nirenberg in a celebrated paper \cite{Brezis-1983} proved that if $N\geq4$ and $1<q<p$ problem \eqref{ele-1.1} has a solution for every $\epsilon>0$.
Molle and Pistoia~\cite{MPistoia} constructed single-bubble solutions
by a finite-dimensional reduction, with concentration points
determined by the Robin function. More recently, Gao and
Ma~\cite{GM-1} studied asymptotic behavior of a positive least energy solution $u_{\epsilon}$ for \eqref{ele-1.1}.  The spectral properties of single-bubble solutions were studied by Yang and Zhao~\cite{YZ}.
For multi-bubble solutions, Micheletti and
Pistoia in \cite{Micheletti} studied \eqref{ele-1.1-1} and developed a finite-dimensional reduction and constructed sign-changing solutions on domains joined by thin handles, namely that
have a solution of the following form
\begin{equation}\label{PW}
u_{\epsilon}=\sum_{i=1}^ma_{i}PW_{\lambda_{i},\xi_{i\epsilon}/\epsilon^{\alpha_0}}+\omega_{\epsilon},
\end{equation}
where $\alpha_0=2/(N-6+q(N-2))$, $a_i\in\{1,-1\}$, $\lambda_1,\cdots,\lambda_m$ are positive parameters and $\xi_{1},\cdots,\xi_m$ are different
points in $\Omega$.
Gao and Ma~\cite{GM} obtained refined bubble decompositions,
concentration rates and equations for the limiting parameters
of positive multi-peak solutions under suitable assumptions.
For the studies of the related topic, Damascelli, Grossi and Pacella~\cite{dgp} used the maximum principle to study symmetry and nodal properties of solutions
of the linearized equation in symmetric domains, with applications to uniqueness and nondegeneracy of positive solutions. For the problem with a negative linear perturbation,$-\Delta u=u^p-\varepsilon u$, Musso and Pistoia~\cite{Musso-Pistoia-2003}
constructed positive two-bubble solutions in domains with
a sufficiently small hole when $N\geq5$.

The preceding results lead naturally to the spectral analysis of positive
multi-bubble solutions of \eqref{ele-1.1}. We first describe the limits of all
fixed-index eigenvalues for positive families admitting a separated
multi-bubble decomposition. This identifies the limiting spectral clusters
for the positive multi-bubble solutions of \eqref{ele-1.1}. We then use the problem \eqref{ele-1.1} to determine the refined
asymptotic behavior of the first $(N+2)m$ eigenvalues and their
eigenfunctions, and relate these expansions to the Morse index of the solution.
We follow the approach of Choi--Kim--Lee~\cite{CKL-1}, with
modifications required by the subcritical perturbation. Remarkably, differently from single-bubble solutions in dimension $N\geq3$, the multi-bubble analysis requires careful control of the interactions between distinct bubbles
when deriving the eigenvalue asymptotics. Moreover, the term
$\varepsilon u_\epsilon^q$ contributes at leading order to
the local scaling identities and must be retained in the
spectral calculation. These effects lead to finite-dimensional
matrices depending on both the concentration configuration
and the exponent $q$. We state our main results below and
then explain the main difficulties and the ideas used to
overcome them.

\subsection{Main results}
We define the Aubin-Talenti bubble function
$$W(x)=[N(N-2)]^{(N-2)/4}(1+|x|^{2})^{-\frac{N-2}{2}}:=C_N(1+|x|^{2})^{-\frac{N-2}{2}}$$
and, for every the concentration rate $\lambda>0$ and the center $\xi=(\xi_1,\cdots,\xi_N)\in\mathbb{R}^N$, its rescaled and translated versions
\begin{equation}\label{WMIU}
W_{\lambda,\xi}(x)=C_N\left(\frac{\lambda}{\lambda^2+|x-\xi|^2}\right)^{\frac{N-2}{2}}\hspace{2mm}\mbox{for}\hspace{2mm}x\in\mathbb{R}^N.
\end{equation}
We notice that $W_{\lambda,\xi}$ are all the positive solutions of the equation
\begin{equation}\label{bec}
-\Delta W=W^{p}\quad \mbox{in}\ \ \mathbb{R}^N,\quad W>0\quad\mbox{in}\ \ \mathbb{R}^N.
\end{equation}
The function $PW$ in \eqref{PW} is the projection of $W$ onto $H_{0}^{1}(\Omega)$, namely,
\begin{equation*}
\Delta PW
=\Delta W,
\quad\text{in}\quad\Omega\quad
PW=0\quad\text{on}\quad\partial\Omega.
\end{equation*}
Furthermore, $G$ denotes the
Green's function of the Laplacian $-\Delta$ in $\Omega$ with the Dirichlet boundary condition, namely that
\begin{equation}\label{Robin}
	H(x,y)=\frac{1}{(N-2)\omega_N|x-y|^{N-2}}-G(x,y),
\end{equation}
where the function $H$ is nothing but the regular part of the Green function and
$H(x,y)$ satisfies
\begin{equation*}
\Delta H(x,y) =0\quad\mbox{in}\quad\Omega,\quad\quad
	H(x,y)=\frac{1}{(N-2)\omega_N|x-y|^{N-2}}\quad\mbox{on}\quad\partial\Omega
\end{equation*}
with $\omega_N$ is the area of the unit sphere in $\mathbb{R}^N$.
For every $x\in\Omega$ the leading term of the regular part of the Green's function $\phi(x):=H(x,x)$ is called the Robin function of $\Omega$ at the pint $x$.

Let us introduce the space
\[
\mathcal O_m(\Omega)
=(0,\infty)^m\times
\{(\xi_1,\ldots,\xi_m)\in\Omega^m:\xi_i\ne\xi_j\ (i\ne j)\},
\]
and the function
 $\Psi_{\Omega,q}(\lambda,\xi):(\mathbb{R}^+)^m\times\Omega^m\rightarrow\mathbb{R}$ defined by
\begin{equation}\label{pusaik-0}
\Psi_{\Omega,q}(\lambda,\xi)=\frac{\mathcal{A}_0^2}{2}\bigg(\sum_{i=1}^k\phi(\xi_i)\lambda_{i}^{N-2}-\sum^k _{\substack{i,j=1,\\
i\neq j}} G(\xi_i,\xi_j)\lambda_{i}^{\frac{N-2}{2}}\lambda_{j}^{\frac{N-2}{2}}\bigg)-\mathcal{B}_0\sum_{i=1}^k\lambda_{i}^{\frac{(N-2)(p-q)}{2}}.
\end{equation}
The constants $\mathcal{A}_0$ and $\mathcal{B}_0$ are given by
\begin{equation}\label{pusaik-1-1}
\mathcal{A}_0=\int_{\mathbb{R}^N}W^{p}(z)dz\quad\mbox{and}\quad \mathcal{B}_0=\frac{1}{q+1}\int_{\mathbb{R}^N}W^{q+1}(z)dz.
\end{equation}
Before turning to the spectral analysis, we first state an
existence result for positive multi-bubble solutions. It provides solutions with distinct interior concentration points and a prescribed asymptotic decomposition, which will serve as the starting point for our analysis.
\begin{thm}\label{existence-0}
Assume that $p=(N+2)/(N-2)$, $1<q<3$ if $N=4$, and $1\leq q<p$ if $N\geq5$, and fix an integer $m\geq2$. Then
\begin{enumerate}[label=\textup{(\roman*)},leftmargin=2.2em,itemsep=5pt]
\item\label{it:given}
 If $(\lambda^{\ast},\xi^{\ast})=(\lambda_{1},\cdots,\lambda_{m},x_{1},\cdots, x_{m})\in \mathcal O_m(\Omega)$ is an isolated stable critical point of $\Psi_{\Omega,q}$. In particular, $(\lambda^{\ast},\xi^{\ast})$ is \(C^1\)-stable in the sense of Li \cite{LI-1997}. Then for $\epsilon$ small enough, there exists a family of positive solution of \eqref{ele-1.1} having the form
\begin{equation}\label{pusaik-1}
u_{\epsilon}=\sum_{i=1}^mPW_{\lambda_{i\epsilon}\epsilon^{\alpha_0},\xi_{i\epsilon}}+\omega_{\epsilon},\quad\mbox{with}\hspace{2mm} \alpha_0=\frac{2}{N-6+q(N-2)},
\end{equation}
which blow-up at $\xi_{1},\cdots, \xi_{m}$ with the rate of the concentration $\lambda_{1}\cdots, \lambda_{m}$. Moreover, $\lambda_{i\epsilon}\rightarrow\lambda_{i}$, $\xi_{i\epsilon}\rightarrow \xi_{i}$ and $\omega_\epsilon\rightarrow0$ in $H_{0}^1(\Omega)$ as $\epsilon\rightarrow0$.

\item\label{it:handles}
There are $m$ pairwise disjoint balls
$B_1,\ldots,B_m$ and a family of smooth bounded connected domains
$\Omega_\delta$, obtained by joining the balls with tubes of width $\delta$,
such that, for all sufficiently small $\delta>0$,
$\Phi_{\Omega,q}$ has an isolated nondegenerate local minimum
$z_\delta\in\mathcal O_m(\Omega_\delta)$.  Consequently, after fixing any such
$\delta$, the positive $m$-bubble solutions in \eqref{pusaik-1}
exist on $\Omega_\delta$ for every sufficiently small $\epsilon>0$.  Their
concentration points lie in mutually distinct ball chambers.
\end{enumerate}
\end{thm}

We begin the spectral analysis with a result that depends only on the
separated multi-bubble decomposition. For this result, the background
$u_\epsilon$ need not solve the Brezis-Nirenberg equation, and its
bubble widths need not have a prescribed concentration rate.
For this general result, we allow $N\geq3$, set $p=(N+2)/(N-2)$,
and assume $1\leq q<p$.
For a positive function $u_\epsilon\in H_0^1(\Omega)$, consider
\begin{equation}\label{ele-2}
\left\lbrace
\begin{aligned}
    &-\Delta v=\mu\big(pu_{\epsilon}^{p-1}+\epsilon qu_{\epsilon}^{q-1}\big)v \quad \hspace{1mm}\mbox{in}\quad \Omega,\\
    &v=0\quad \quad \quad \quad \quad \quad \quad \quad \quad \quad \quad \hspace{1.5mm}\mbox{on}\hspace{2mm}\partial\Omega.
   \end{aligned}
\right.
\end{equation}
Here and in Theorem~\ref{THm1-2}, this problem is
understood weakly in $H_0^1(\Omega)$: for every $\varphi\in H_0^1(\Omega)$,
\[
 \int_\Omega\nabla v\cdot\nabla\varphi
 =\mu\int_\Omega
 \bigl(pu_\varepsilon^{p-1}+\varepsilon q u_\varepsilon^{q-1}\bigr)v\varphi.
\]
For each $\epsilon>0$, let
\[
    0<\mu_{1\epsilon}\leq\mu_{2\epsilon}
    \leq\cdots\leq\mu_{l\epsilon}\leq\cdots
\]
be the eigenvalues of the linearized problem, arranged in
nondecreasing order and counted with multiplicity.
Related spectral decoupling results can be found in
\cite[Proposition~4.7 and Remark~4.8]{FigalliGlaudo2020}.
In particular, Remark~4.8 concerns finite-dimensional spectral
subspaces for bubbles with fixed scales and mutually diverging centers
in $\mathbb R^N$, where $N\geq5$.
The following theorem describes the limiting spectrum for positive
families admitting a separated multi-bubble decomposition.
 \begin{thm}\label{THm1-2}
 Let $N\geq3$, let $\Omega\subset\mathbb R^N$ be a smooth bounded domain,
and set $p=(N+2)/(N-2)$. Fix $m\geq1$ and $1\leq q<p$.
Suppose that, as $\epsilon\to0$,
$u_\epsilon\in H_0^1(\Omega)$ is positive almost everywhere and
\[
 u_\epsilon
 =\sum_{i=1}^{m}PW_{\delta_{i\epsilon},\xi_{i\epsilon}}
 +\omega_\epsilon,
 \qquad\|\omega_\epsilon\|_{H_0^1(\Omega)}\longrightarrow0,
\]
where $\delta_{i\epsilon}>0$, $\delta_{i\epsilon}\to0$, and
$\xi_{i\epsilon}\to\xi_i\in\Omega$ with $\xi_i\ne\xi_j$ for $i\ne j$.
Let $\mu_{l\epsilon}$ be the eigenvalues of
\eqref{ele-2}, counted with multiplicity.
For $s=0,1,2,\ldots$, define
\[
 \mu_s^\star
 =\frac{N(N-2)+4s(s+N-1)}{N(N+2)},\qquad
 h_s=\binom{N+s}{s}-\binom{N+s-2}{s-2},
\]
where the second binomial coefficient is zero for $s=0,1$.
Let $(\mu_l^0)_{l\geq1}$ be the nondecreasing sequence obtained by
repeating each $\mu_s^\star$ exactly $mh_s$ times. Then, for every
fixed $l\in\mathbb N$,
\[
 \mu_{l\varepsilon}\longrightarrow\mu_l^0
            \qquad\text{as }\epsilon\to0^+.
\]
In particular, setting $K=(N+2)m$, we have
\[
 \mu_{l\epsilon}\longrightarrow
 \begin{cases}
  1/p,&1\leq l\leq m,\\[2pt]
  1,&m+1\leq l\leq K,\\[2pt]
  (N+4)/N,&K+1\leq l\leq m(N+1)(N+4)/2.
 \end{cases}
\]
\end{thm}
The proof of Theorem~\ref{THm1-2} is given in Section~\ref{sec3}.
No nonlinear equation is imposed on $u_\epsilon$, and
the concentration scales $\delta_{i\epsilon}$ need not
be comparable or have a prescribed dependence on $\epsilon$.
The stated multiplicities refer to the limiting spectral levels;
eigenvalues within the same cluster need not coincide
for $\epsilon>0$.

We now apply Theorem~\ref{THm1-2} to the problem \eqref{ele-1.1}.
For the remaining results in the introduction, we return to the ranges
$1<q<3$ if $N=4$ and $1\leq q<p$ if $N\geq5$, with $m\geq2$.
Let $\{u_\epsilon\}$ be a family of positive solutions of \eqref{ele-1.1}
satisfying \eqref{pusaik-1}, with
$\lambda_{i\epsilon}\to\lambda_i>0$ and
$\xi_{i\epsilon}\to\xi_i\in\Omega$ at distinct points.
This includes the solutions constructed in Theorem~\ref{existence-0}.
Taking $\delta_{i\epsilon}
 =\lambda_{i\epsilon}\epsilon^{\alpha_0}$
verifies the assumptions of
Theorem~\ref{THm1-2}.
For this solution family, \eqref{ele-2} is the weighted eigenvalue problem
associated with the linearization of \eqref{ele-1.1}.

Let $v_{l\epsilon}$ be a corresponding eigenfunction,
normalized by
$\|v_{l\epsilon}\|_{L^\infty(\Omega)}=1$.
Given $i=1,\cdots,m$,
 let $V_{li\epsilon}$ be a dilation of the eigenfunction $v_{l\epsilon}$ corresponding to $\lambda_{i\epsilon}$ defined as
\begin{equation}\label{vie}
V_{li\epsilon}(x)=v_{l\epsilon}\left(\lambda_{i\epsilon}\epsilon^{\alpha_0}x+\xi_{i\epsilon}\right)\hspace{2mm}\mbox{for}\hspace{2mm} x\in\Omega_{i\epsilon}:=\lambda_{i\epsilon}^{-1}\epsilon^{-\alpha_0}(\Omega-\xi_{i\epsilon}),\hspace{2mm}\mbox{where}\hspace{2mm}\alpha_0=\frac{2}{N-6+q(N-2)}.
\end{equation}

Since $\mu_{(K+1)\epsilon}\to (N+4)/N>1$, only the first
$K$ eigenvalues can contribute to the Morse index or the kernel
of the linearized operator for sufficiently small $\epsilon$.
To determine the position of the eigenvalues converging to $1$
relative to this threshold, we use equation~\eqref{ele-1.1} to derive
their leading-order corrections.
We organize this analysis into the three ranges
\[
1\leq l\leq m,\qquad
m+1\leq l\leq(N+1)m,\qquad
(N+1)m+1\leq l\leq(N+2)m,
\]
which will be associated with the amplitude, translation, and dilation modes, respectively. The eigenvalue expansions hold as \(\varepsilon\to0\), while the limiting coefficient vectors are obtained along common subsequences.

We first consider $1\leq l\leq m$. Let $M=M(\xi)=(m_{ij})_{1\leq i,j\leq m}$ be the symmetric matrix
defined by
$$
m_{ii}(x)=\phi(\xi_i),\hspace{3mm} m_{ij}(\xi)=-G(\xi_i,\xi_j)\hspace{2mm}\mbox{if}\hspace{2mm}i\neq j.
$$
Define the diagonal matrices
\begin{equation}\label{DDT}
    D_t
    :=
    \operatorname{diag}\left(
        \lambda_1^{\frac{N-2}{2}},\ldots,
        \lambda_m^{\frac{N-2}{2}}
    \right),
    \qquad
    D_{q}
    :=
    \operatorname{diag}\left(
        \lambda_1^{\frac{(N-2)(p-q)}{2}},\ldots,
        \lambda_m^{\frac{(N-2)(p-q)}{2}}
    \right).
\end{equation}
We then introduce the symmetric matrix
\begin{equation}\label{pusaik-2}
    A
    :=
    D_t M D_t
    +
    \frac{(p-q)(q+1-p)\mathcal B_0}{p\mathcal A_0^2}\,D_{q},
\end{equation}
and denote its eigenvalues by
$\nu_1^1\leq\cdots\leq\nu_m^1$, counted with multiplicity.
The following theorem gives
the asymptotic expansions of these eigenvalues, which converge
to $1/p$, and shows that the local limits of the corresponding
rescaled eigenfunctions are multiples of the Aubin-Talenti bubble $W$.
\begin{thm}\label{Figalli-1}
Assume that $p=(N+2)/(N-2)$, $1<q<3$ if $N=4$, and $1\leq q<p$ if $N\geq5$. Let $\{u_{\epsilon}\}$ be a
family of solutions of \eqref{ele-1.1} of the form \eqref{pusaik-1} and $\mu_{l\epsilon}$ the $l$-th eigenvalue of problem \eqref{ele-2} for $1\leq l\leq m$. Then there exists a constant $\alpha_N$ defined by
$\alpha_N=\frac{\mathcal{A}_0^2}{p\int_{\mathbb{R}^N}W^{p+1}(x)dx}$
such that
	\begin{equation}\label{vie-1}	
\mu_{l\epsilon}=\frac{1}{p}+\alpha_N\nu_{l}^{1}\epsilon^{(N-2)\alpha_0}+o\big(\epsilon^{(N-2)\alpha_0}\big)
\end{equation}
as $\epsilon\rightarrow0$. Moreover, there exists a nonzero column vector
\begin{equation}\label{eltita}
\tilde{\bm{e}}_l=\Big(\lambda_{1}^{\frac{N-2}{2}}e_{l1},\cdots,\lambda_{m}^{\frac{N-2}{2}}e_{lm}\Big)^T\in\mathbb{R}^m
\end{equation}
such that for each $i\in\{1,\cdots,m\}$ the function $V_{li\epsilon}\rightharpoonup e_{l\epsilon}W$ weakly in $\mathcal{D}^{1,2}(\mathbb{R}^N)$ as $\epsilon\rightarrow0$. Also the vector
$\tilde{\bm{e}}_l$ is an eigenvector corresponding to $\nu_{l}^1$, $1\leq l\leq m$ and these vectors are nonzero and mutually orthogonal.
\end{thm}

We next consider the range $m+1\leq l\leq(N+1)m$. The first task is to identify their local eigenfunction limits and exclude a dilation component. The
following theorem shows that only translation modes remain.
\begin{thm}\label{existence-1}
		For $m+1\leq l\leq(N+1)m$, it holds that
	\begin{equation}\label{vlie}
V_{li\epsilon}\rightarrow-\sum_{k=1}^{N}a_{l,i,k}\partial_kW\quad\mbox{as}\hspace{2mm}\epsilon\rightarrow0\hspace{2mm}\mbox{in}\hspace{2mm}C_{loc}^{1}(\mathbb{R}^N),\hspace{2mm}i=1,\cdots,m
\end{equation}
for all $\epsilon>0$ small.
\end{thm}

We continue to work with $m+1\leq l\leq(N+1)m$. The preceding theorem describes the eigenfunctions near the concentration points, but determining the first eigenvalue corrections also requires an expansion away from these points. To describe this exterior behavior, let us introduce a symmetric $m\times m$ matrix $M(\xi)=(m_{ij})_{1\leq i,j\leq m}$ defined by
$$
m_{ii}(x)=\phi(\xi_i),\hspace{3mm} m_{ij}(\xi)=-G(\xi_i,\xi_j)\hspace{2mm}\mbox{if}\hspace{2mm}i\neq j.
$$
Introducing the matrixs
\begin{equation}\label{Nq}
N_{q}:=M+
\frac{\mathcal{B}_0(p-q)(q+1-p)}{\mathcal{A}_0^2}\diag(\lambda_i^{\frac{(N-2)(p-q-2)}{2}}).
\end{equation}
In virtue of Lemma \ref{positivede}, we know $N_{q}$ is positive definite in particular invertible. We denote its inverse $N_q^{-1}=(\eta^{ij})_{1\leq i,j\leq m}$.

For $m+1\leq l\leq(N+1)m$, the next theorem gives the exterior expansion in terms of the Green function
and its first derivatives.
\begin{thm}\label{vepsilon}
For $m+1\leq l\leq(N+1)m$, it holds that
\begin{equation}\label{116}
\begin{split}
\frac{v_{l\epsilon}}{\epsilon^{(N-1)\alpha_0}}\rightarrow& \mathcal{A}_0\Bigg[\sum_{i=1}^{m}\sum_{j=1}^{m}\sum_{k=1}^{N}\eta^{ij}\Bigg(-\frac{1}{2}
\lambda_{j}^{N-1}a_{l,j,k}\frac{\partial \phi}{\partial x_k}(\xi_{j})+\sum_{l\neq j}\lambda_{l}^{N-1}a_{l,l,k}\frac{\partial G}{\partial y_k}(\xi_{j},\xi_{l})\Bigg)G(x,\xi_{i})
\\&+\sum_{i=1}^{m}\sum_{k=1}^{N}\lambda_{i}^{N-1}a_{l,i,k}\frac{\partial G}{\partial y_k}(x,\xi_{i})
\Bigg]
\quad\mbox{in}\quad C_{loc}^{1}\big(\Omega\setminus\{\xi_{1},\cdots,\xi_{k}\}\big)
\end{split}
\end{equation}	
as $\epsilon\rightarrow0$, where $\mathcal{A}_0$ is a dimensional constant in \eqref{pusaik-1-1}.
\end{thm}

Let $\bm{\tilde{d}}_l$ is a nonzero vector defined by
\begin{equation}\label{pusaik-3}
\begin{split}
\bm{\tilde{d}}_l=\Bigg(\lambda_{1}^{\frac{N-2}{2}}&a_{l,1,1},\cdots,\lambda_{1}^{\frac{N-2}{2}}a_{l,1,N},\lambda_{2}^{\frac{N-2}{2}}a_{l,2,1},\cdots,
\lambda_{m-1}^{\frac{N-2}{2}}a_{l,m-1,N}, \lambda_{m}^{\frac{N-2}{2}}a_{l,m,1},\cdots,\lambda_{m}^{\frac{N-2}{2}}a_{l,m,N}\Bigg)^T\in\mathbb{R}^{m\times N}.
\end{split}
\end{equation}
Introducing a $m\times mN$ matrix $P=(P_{i,(j,k)})_{1\leq i,j\leq m, 1\leq k\leq N}$ and a symmetry $mN\times mN$ matrix $Q=(Q_{(i,k),(j,t)})_{1\leq i,j\leq m,1\leq k,t\leq N}$ as follow.
\begin{equation}\label{pusaik-4}
P_{i,(j,k)}=
\left\lbrace
\begin{aligned}
    &\lambda_{j}^{\frac{N}{2}}\frac{\partial G}{\partial y_k}(\xi_{i},\xi_{j})\quad\quad\quad\quad\hspace{8.5mm}\mbox{if}\hspace{3mm} i\neq j,
    \\&-\frac{1}{2}\lambda_{i}^{\frac{N}{2}}\frac{\partial \phi}{\partial x_k}(\xi_{i})\quad \quad\quad\quad\quad\hspace{3mm}\mbox{if}\hspace{3mm}i=j,
 \end{aligned}
\right.
\end{equation}
and
\begin{equation}\label{pusaik-5}
Q_{(i,k),(j,t)}=
\left\lbrace
\begin{aligned}
    &\lambda_{i}^{\frac{N}{2}}\lambda_{j}^{\frac{N}{2}}\frac{\partial G}{\partial x_k\partial y_t}(\xi_{i},\xi_{j})\quad\quad\quad\quad\quad\quad\quad\quad\quad\quad\quad\quad\quad\quad\quad\hspace{3mm} \mbox{if}\hspace{3mm} i\neq j,
    \\&-\frac{\lambda_{i}^{N}}{2}\frac{\partial^2 \phi}{\partial x_k\partial x_t}(\xi_{i})+\lambda_{i}^{\frac{N+2}{2}}\sum_{l\neq i}\lambda_{l}^{\frac{N-2}{2}}\frac{\partial^2 G}{\partial x_k\partial x_t}(\xi_{i},\xi_{l})\quad \hspace{1mm}\quad\quad\mbox{if}\hspace{3mm}i=j.
 \end{aligned}
\right.
\end{equation}
For the same range $m+1\leq l\leq (N+1)m$, our result can be stated as follows.
	\begin{thm}\label{remainder terms}
Assume that $B_q$ is a $mN\times mN$ symmetric matric given by
$$B_q=P^TN_q^{-1}P+Q.$$
Let $\nu_{m+1}^2\geq \cdots\nu_{(N+1)m}^2$ be the eigenvalues of $B$.
Then there exist some constant
$$\beta_N:=\mathcal{A}_0^2\big(p\int_{\mathbb R^N}
W^{p-1}(\partial_1W)^2\,dz\big)^{-1}>0$$ such that for $\epsilon$ small, it holds that
\begin{equation}\label{eq:translationresult}
\mu_{l\epsilon}
=1-\beta_N\nu_{l}^2\epsilon^{\alpha_0N}+o(\epsilon^{\alpha_0N}),
\quad\mbox{for any}\hspace{2mm}m+1\leq l\leq(N+1)m.
\end{equation}
Moreover,
the vector $\tilde{\bm{d}}_l$ is an eigenvector corresponding to $\nu_{l}^2$, $m+1\leq l\leq (N+1)m$ and these vectors are nonzero and mutually orthogonal, and hence form a basis of $\mathbb{R}^{Nm}$.
\end{thm}

\begin{Rem}
No nondegeneracy assumption on $B$ is imposed.
If $\nu_l^2=0$, the expansion \eqref{eq:translationresult} gives only
\[
    \mu_{l\epsilon}-1=o(\epsilon^{\alpha_0N}).
\]
Thus, this expansion alone does not determine the sign of
$\mu_{l\epsilon}-1$.
\end{Rem}

Let us denote by $\nu_{l}^3$ the $(l-k(N+1))$ eigenvalues of the matrix $H_q$ by
\begin{equation}\label{pusaik-2-2}
 H_q:=\mathcal{A}_0^2(\frac{N-2}{2})^2D_tM D_t
      +\mathcal{B}_0(\frac{N-2}{2})^2(p-q)(q+1-p)D_{q},
\end{equation}

The next theorem describes both the local profiles and the eigenvalue corrections for $(N+1)m+1\leq l\leq(N+2)m$. In particular, the positive definiteness of $H_q$ places these eigenvalues above 1 for small $\epsilon$.
	\begin{thm}\label{thmprtb}
let $0<\nu_{(N+1)m+1}^{3}\le\cdots\le\nu_{(N+2)m}^{3}$ be the eigenvalues of $H_q$, counted with multiplicity. Set
$$
Z_0=-z\cdot\nabla W-aW,\quad\mathcal{C}_{p,q}:=(p\int_{\mathbb R^N}
    W^{p-1}Z_0^2\,dz)^{-1}.
$$
Then there exists a nonzero vector
$$\tilde{\bm{f}}_l=\Big(\lambda_{1}^{\frac{N-2}{2}}a_{l,1},\cdots,\lambda_{m}^{\frac{N-2}{2}}a_{l,m}\Big)^T\in\mathbb{R}^m$$
such that for each $(N+1)m+1\leq l\leq(N+2)m$,
it holds that
		\begin{equation}\label{rtbdy46}	
\mu_{l\epsilon}
=1+\mathcal{C}_{p,q}\nu_l^{3}\epsilon^{\alpha_0(N-2)}+o(\epsilon^{\alpha_0(N-2)}),\quad V_{l,i,\epsilon}\rightharpoonup a_{l,i} Z_0\quad\mbox{weakly in}\quad\mathcal{D}^{1,2}(\mathbb{R}^N) \quad\mbox{as}\quad\epsilon\rightarrow0,
\end{equation}
The nonzero vectors $\tilde{\bm{f}}_{l}$ in \eqref{eq:dilationlocal} are corresponding eigenvectors of $H_q$ and form an orthogonal basis of $\mathbb{R}^m$ along every simultaneous convergent subsequence.
\end{thm}

The preceding spectral estimates yield information on the Morse
index of $u_\varepsilon$. Recall that its Morse index and
augmented Morse index are defined by
\[
    \operatorname{ind}(u_\epsilon)
    :=\#\{l:~\mu_{l\epsilon}<1\},
    \qquad
    \operatorname{ind}_0(u_\epsilon)
    :=\#\{l:~\mu_{l\epsilon}\leq1\},
\]
respectively, where eigenvalues are counted with multiplicity.
For a real symmetric matrix $C$, we similarly denote by
$\operatorname{ind}(A)$ and $\operatorname{ind}_0(A)$ the numbers
of its negative and nonpositive eigenvalues, counted with
multiplicity. As a direct consequence of Theorem \ref{Figalli-1}, Theorem \ref{remainder terms} and Theorem \ref{thmprtb}, we can prove the following corollary.
\begin{cor}\label{emm-1}
Let the assumptions of Theorem \ref{Figalli-1} be satisfied. Then
$$m\leq m+ind(-B_q)\leq ind(u_{\epsilon})\leq ind_0(u_{\epsilon})\leq m+ind_0(-B_q)\leq (N+1)m$$
for sufficiently small $\epsilon>0$. Therefore if the matrix $B_q$ nondegenerate, then we have
$$ind(u_{\epsilon})=m+ind(-B_q)\in[m,(N+1)m].$$
\end{cor}	

Before starting the proof of our main theorems, we would like to mention about related results obtained for the slightly subcritical problem
 \begin{equation}\label{SC}
 -\Delta u
=u^{p-\epsilon}
 ,~~~u>0,~~~\text{in}~~~\Omega,~~~
 		u=0,~~~\text{on}~~~\partial\Omega,
 \end{equation}
 The studies on the asymptotic behavior of radial solutions of \eqref{SC} were initiated by Atkinson and Peletier in \cite{ATKINSON-1986} by using ODE technique in the unit ball of $\mathbb{R}^3$. Later, Brezis and Peletier \cite{BP} used the method of PDE to obtain the same results as that in \cite{ATKINSON-1986} for the spherical domains.
 Finally, the same kind of results hold for nonspherical domain, which was settled by Han in \cite{HANZCHAO} (independently by Rey in \cite{Rey-1989}). Later that Wei in \cite{JW0} further locate the blow-up point $x_0$ and to give a precise asymptotic expansion of the least energy solutions to problem \eqref{SC}. Li, Wei and Zou~\cite{LiWZ} further established local uniqueness and nondegeneracy of concentrating solutions to the slightly
subcritical problem under suitable assumptions, and obtained
uniqueness and counting results on certain classes of domains.
 For multi-bubble solutions, Bahri-Li-Rey in \cite{B-L-R}
studied the concentration configurations and their relation
to the Morse index through a finite-dimensional function
involving the Green and Robin functions. Musso and
Pistoia~\cite{Musso-Pistoia-2002} constructed multi-bubble solutions
under suitable stable critical point assumptions on the
reduced function. Furthermore, Grossi and Pacella in \cite{GP05} deduced the asymptotic behavior of the eigenvalues $\lambda$ for the linearized problem
 \begin{equation}\label{SC-1}
 -\Delta v
=(p-\epsilon)\lambda u_{\epsilon}^{p-1-\epsilon}v
 ,~~~\text{in}~~\Omega,~~~
 		v=0,~~~\text{on}~~\partial\Omega,
 \end{equation}
and the Morse index of $u_{\epsilon}$, as well as their result in \cite{CKL-1} was extended later to the linearized problem \eqref{SC-1} at $u_{\epsilon}$ if asymptotic forms of the solutions are written as
\begin{equation*}
u_{\epsilon}=\sum_{i=1}^mPW_{\lambda_{i\epsilon}\epsilon^{1/(N-2)},\xi_{i\epsilon}}+\omega_{\epsilon},
\end{equation*}
 where $\alpha_{i\epsilon}\rightarrow1$, $\lambda_{i\epsilon}\rightarrow\lambda_{i}$ and $\xi_{i\epsilon}\rightarrow \xi_{i}$ as $\epsilon\rightarrow0$.
 In \cite{GG09}, Gladiali and Grossi studied the asymptotic behavior of the eigenvalues and Morse index of one-bubble solution for the linearized Gelfand problem.
Later that, the corresponding problem of multi-bubble solutions was proved in \cite{GGO-1}, and further the qualitative properties of the first $m$ eigenfunctions was established by Gladiali et al. in \cite{GGO}.
For qualitative analysis to an eigenvalue problem, we also refer to \cite{LTX}.

Theorem \ref{THm1-2} determines the limiting spectral clusters from the
separated multi-bubble decomposition. To obtain the refined
eigenvalue and eigenfunction asymptotics, we also use
problem~\eqref{ele-1.1}, together with quantitative estimates for the
solutions and their eigenfunctions. In particular, terms that
vanish in the limiting weighted form may still contribute to
the leading-order eigenvalue corrections.

We conclude this introduction by describing the main difficulties in
extending the approach of Choi--Kim--Lee~\cite{CKL-1} to problem \eqref{ele-1.1}, together with the new ingredients needed to overcome them. These
difficulties arise at the several key steps of the Choi--Kim--Lee's argument, namely the analysis corresponding to Proposition~4.4 and the derivation of the
sharp exterior expansion~(1.16).

\medskip
\noindent \textbf{(1) The nonhomogeneous scaling equation.}
The counterpart of Proposition~4.4 in Choi--Kim--Lee~\cite{CKL-1}
requires a substantial modification. The main obstruction is that the
scaling mode used in \cite{CKL-1} solves the homogeneous linearized equation,
whereas in the present problem
\[
    L_\varepsilon \omega_{i\varepsilon}
    =
    \frac{2(p-q)}{p-1}\,
    \varepsilon u_\varepsilon^q,\quad \mbox{where}\quad L_\varepsilon=-\Delta -pu_{\epsilon}^{p-1}-\epsilon qu_{\epsilon}^{q-1},
\]
See also Lemma \ref{nabla}. Consequently, the corresponding local identity contains an additional
defect volume term. Despite the small factor $\epsilon$, this term
contributes at the same leading order as the boundary terms and hence
cannot be absorbed into the error, so the argument of \cite{CKL-1} does not
apply directly. To handle this defect term, we establish the refined estimates and identities in Lemmas~\ref{m00}-\ref{limti}. In particular, the resulting leading-order
coefficients depend explicitly on $q$, which leads to a new matrix
$H_q$ in place of the matrix arising in \cite{CKL-1}. We then prove that $H_q$
is strictly positive definite in Lemma \ref{ckl}, which provides the
coercivity needed to complete the analogue of Proposition~4.4.

\medskip
\noindent \textbf{(2) The sharp exterior expansion.}
 A further difficulty arises in deriving the sharp exterior expansion
corresponding to (1.16) in \cite{CKL-1}. The argument in \cite{CKL-1} relies essentially on the decay estimate of Lemma~2.8 to control the Taylor remainder
in the Green representation and on the cancellation mechanism of
Lemma~5.5 to determine the relevant zeroth-order mass. Both
mechanisms are affected by the additional lower-order perturbation
in the present problem. Indeed, the Green representation contains the additional defect term $\epsilon q u_\epsilon^{q-1}v_{l\varepsilon}$,
whose decay, in the low-dimensional borderline regime, is insufficient
for the direct Choi-Kim-Lee's argument to produce the required
$o(\epsilon^{\alpha_0(N-1)})$ remainder. Instead of estimating this term
separately, we retain its dependence on the undetermined the
coefficients through
\[
    \mathfrak h_\epsilon=\epsilon^{\alpha_0}
             +\frac{\max_i|k_\epsilon|}{\epsilon^{\alpha_0(N-2)}}.
\]
Corollary~\ref{lem:q-potential} and Lemma~\ref{lem:green} then yield a preliminary exterior expansion
with remainder $o(\epsilon^{\alpha_0(N-2)}\mathfrak h_\epsilon)$,
without assuming $\mathfrak h_\epsilon=O(\epsilon^{\alpha_0})$.
To conclude this estimate, one has to overcome a second obstruction:
the Brezis-Nirenberg perturbation leaves an additional signed volume term in the
bilinear Pohozaev identity, which cannot be eliminated by the
cancellation used in \cite{CKL-1}. The signed relation established
in Corollary~\ref{lem:moment}, combined with the bilinear identity (5.8) in \cite{CKL-1} and the identity~\eqref{deference}, yields
\[N_q\kappa_1
    =
    \mathcal R_\varepsilon(\kappa_2)
    +
    o(\epsilon^{\alpha_0(N-2)} \mathfrak h_\epsilon).
    \tag{5.18}
\]
Since $N_q>0$ by~Lemma \ref{positivede}, we obtain
\[
    \kappa_1=O(\epsilon^{\alpha_0(N-1)}),
    \qquad
    \mathfrak h_\epsilon=O(\epsilon^{\alpha_0}).
\]
Substituting this estimate back into the preliminary expansion improves
the remainder to $o(\epsilon^{\alpha_0(N-1)})$, and hence gives the sharp
exterior expansion~\eqref{exterior}.

\medskip
\noindent \textbf{(3) Sharp eigenfunction decay.}
We also refine the eigenfunction decay estimate corresponding to
Lemma~2.8 in~\cite{CKL-1}. The additional lower-order
term in the present Brezis-Nirenberg problem requires a modification of their comparison argument. More precisely, the potential in the rescaled eigenvalue
equation satisfies
\begin{equation}\label{sigamaq}
    c_\varepsilon(x)
    \leq C d_\varepsilon(x)^{-\sigma_q},
    \qquad
    \sigma_q:=\min\{4,(N-2)q\}>2,
\end{equation}
outside fixed neighborhoods of the rescaled concentration points,
where $d_\varepsilon(x)$ denotes the distance to these points.
In particular, when $N=4$ and $1<q<3/2$, we have $\sigma_q=2q<3$,
so this estimate does not provide the $d_\epsilon^{-3}$ bound
used in the argument of \cite{CKL-1}.
To accommodate this slower decay, we adapt the comparison iteration
to the above decay estimate, choosing increments in the decay exponent
smaller than $\sigma_q-2$ and starting directly from the $L^\infty$
normalization. This first yields decay with exponents below $N-2$.
Since $|x|^{2-N}$ is harmonic away from the origin, the pure-power
iteration does not directly yield the endpoint estimate. We therefore
introduce a corrected barrier to complete the comparison and obtain \eqref{V5}.
Here $C_l$ is independent of $\epsilon$ for each fixed $l$.
The exponent $N-2$ is optimal in general, as shown by the first
eigenfunction.

The preceding analysis provides the key ingredients for the fine
spectral asymptotics in Theorem \ref{remainder terms} and Theorem \ref{thmprtb}. The translation
eigenvalue asymptotics follow from the local Green identities
combined with the sharp exterior expansion. For the dilation
eigenvalues, Lemma~\ref{lem:dilationcluster} identifies the limiting dilation modes, and retaining the additional volume term in the scaling identity
leads to the same matrix $H_q$ already obtained in the
scale-mode analysis. The relation between the two limiting matrices
is described in Remark~\ref{r9-3} in terms of the scale Hessian and its
Schur complement.

Theorem \ref{THm1-2} also applies to the borderline case $N=4$ and $q=1$. The refined asymptotic expansions in this case, however, requires changes to both the concentration scale and several key estimates used in the spectral analysis.
The main differences from the parameter range considered here
are summarized in Table~\ref{tab:parameter-regimes}.

\subsection{Structure of the paper.}
	The paper is organized as follows.

In section \ref{section2}, we analyze the existence of multi-bubble solutions, and obtain some decay estimates on rescaled functions of $u_{\epsilon}$.

In section \ref{section-3}, we study the behavior of solutions $u_{\epsilon}$ outside the blow-up point $\{\xi_1,\cdots,\xi_m\}$, and we obtain a sharp decay estimate on the eigenvalue function $v_{l\epsilon}$ which will be an indispensable tool throughout the paper.

In Section \ref{sec3}, we establish the limiting spectrum under a general
separated multi-bubble decomposition, prove strong energy decompositions of the eigenfunctions,
and derive coercivity estimates under suitable orthogonality conditions.

In section \ref{asymptotic} we are devoted to considering the eigenvalue $\mu_{l\epsilon}$ for $1\leq l\leq m$, and concluding the proof of Theorem \ref{Figalli-1}.

In section \ref{section4}, we first give the estimates for the eigenvalues $\mu_{l\epsilon}$ for $m+1\leq l\leq(N+1)m$ and asymptotic behaviour of the eigenfunctions. In additional, we the refined estimates and identities in Lemmas~\ref{m00}-\ref{limti}.
Based on these results and combined with the positive definiteness of the matrix $H_q$, we complete the proof of Theorem \ref{existence-1}.

In section \ref{section-6}, We introduce the $\ast$-norm and the critical mass $\mathfrak h_\epsilon$ , and establish several important results to overcome the difficulties arising from the perturbation term,
so as to prove Theorem \ref{vepsilon}.

In section \ref{section-7}, we apply the Green's identity to prove Theorem \ref{remainder terms}.

In section \ref{section-8}, we are devoted to proving Theorem \ref{thmprtb}.

In apprienx \ref{sharp}, we prove the decay estimate of the eigenfunctions and illustrate the optimality of the common decay exponent.

\begin{table}[!htbp]
    \begingroup
    \small
    \centering
    \caption{Comparison of the two parameter regimes.}
    \label{tab:parameter-regimes}
    \setlength{\tabcolsep}{4pt}
    \renewcommand{\arraystretch}{1.20}
    \renewcommand{\tabularxcolumn}[1]{m{#1}}
    \begin{tabularx}{\linewidth}{|
        >{\raggedright\arraybackslash}m{0.16\linewidth}|
        >{\centering\arraybackslash}X|
        >{\centering\arraybackslash}X|}
        \hline
        Cases
        & $1<q<3$ if $N=4$, $1\leq q<p$ if $N\geq5$,
        & $N=4$, $q=1$ \\
        \hline

         Concentration rate
        & \vspace{2pt}
          $\begin{gathered}
            N-2-\tau_0(p-q)>0,\\
            \epsilon^{\alpha_0};\hspace{2mm}
        \epsilon\epsilon^{\tau_0(p-q)}=\epsilon^{N-2}
          \end{gathered}$
          \par\smallskip
          (See \eqref{miu}-\eqref{beta1})
        & \vspace{2pt}
          $\begin{gathered}
            N-2=\tau_0(p-q)=2,\\
            \epsilon\log(\delta_{1,\epsilon}^{-1})
            \longrightarrow c_*>0
          \end{gathered}$
          \par\smallskip
          (See \cite{P-R-V})\par\vspace{3pt}
        \\
        \hline

        Decay estimates
        & \vspace{2pt}
          $\begin{gathered}
            \sigma_q=\min\{4,(N-2)q\}>2\hspace{2mm}\mbox{in}\hspace{2mm}\eqref{sigamaq},\\
            0<\theta<\sigma_q-2
          \end{gathered}$
          \par\smallskip
          (See Appendix \ref{sharp})
        & \vspace{2pt}
          $\begin{gathered}
            \sigma_1=2,\\
            (0,\sigma_1-2)=\varnothing
          \end{gathered}$
          \par\smallskip
          No positive gain from this bound.
          \par\vspace{2pt}
        \\
        \hline

        Integral estimates
        & \vspace{2pt}
          $\begin{gathered}
            \mathcal{B}_0<\infty,\hspace{2mm}
            \displaystyle\int_{\mathbb R^N}W^qZ_0=\tau_0(p-q) \mathcal{B}_0
          \end{gathered}$
          \par\smallskip
          (See Lemma \eqref{UqZ0})
        & \vspace{2pt}
          $\begin{gathered}
            \displaystyle\int_{B_R}W^2\approx\log R;
            \displaystyle\int_{B_R}WZ_0\approx\log R,
          \end{gathered}$
          \par\smallskip
          $\mathcal{B}_0=+\infty$\hspace{2mm}\mbox{as}\hspace{2mm}$R\to\infty$
          \par\smallskip
          (See Lemma \ref{limti})
          \par\vspace{2pt}
        \\
        \hline

        Positive definiteness of the matrixs
        & \vspace{2pt}
          $\begin{gathered}
            H_q\succ0,~~
            N_q\succ0
          \end{gathered}$
          \par\smallskip
          (See Lemma \ref{ckl},~Lemma \ref{positivede})
        & \vspace{2pt}
          $\begin{gathered}
    \mathcal K_0:=D_r(M-\Lambda_1I)D_r,\\
    \mathcal K_0\succeq0,\\
    \ker\mathcal K_0=\operatorname{span}\{\mathbf1\}.
\end{gathered}$
          \par\vspace{2pt}
        \\
        \hline
    \end{tabularx}

    \smallskip
\begin{minipage}{\linewidth}
    \footnotesize
    Here $c_*>0$ is the constant in the exponential concentration
    rate. In the last row, $M=M(\xi)$ is evaluated at the limiting
    configuration, and $\Lambda_1$ is its simple smallest eigenvalue.
    The corresponding positive eigenvector
    $r=(r_1,\ldots,r_m)^T$ is normalized by $r_1=1$.
    For the solutions constructed in~\cite{P-R-V}, the parametrization
    in Section~2.2 and the limiting system in Section~3.1 imply
    $\delta_{i,\varepsilon}/\delta_{1,\varepsilon}\to r_i$
    along a family whose concentration points converge to $\xi$.
    We write $D_r=\operatorname{diag}(r_1,\ldots,r_m)$ and
    $\mathbf1=(1,\ldots,1)^T$.
    The identity $\ker\mathcal K_0=\operatorname{span}\{\mathbf1\}$
follows from $Mr=\Lambda_1r$ and the simplicity of $\Lambda_1$
\end{minipage}
    \endgroup
\end{table}

\FloatBarrier

\subsection{Notations.}
We list some notational conventions which will be used throughout the paper.

\medskip
\noindent
- $\omega_N$ is the area of the unit sphere in $\mathbb{R}^N$, and $a\approx b$ means that $a\lesssim b$ and $\gtrsim b$.

\medskip
\noindent
- We write $\|\cdot\|_{\infty}$ instead of the norm of $L^{\infty}(\Omega)$.

\medskip
\noindent
- We denote by $\mathbf{1}_E$ the indicator function of a measurable set $E$.

\medskip
\noindent
- For $x\in\Omega$, we denote the distance from $x$ to $\partial\Omega$ by $\operatorname{dist}(x,\partial\Omega)$.

\medskip
\noindent
- $B(\xi_i,r)=\{x\in\mathbb{R}^{N}:~|x-\xi_i|<r\}$ for each $\xi_i\in\mathbb{R}^N$ and $r>0$, for $i=1,\cdots m$.

\medskip
\noindent
- The surface measure is denoted as $dS_{x}$ where $x$ is the variable of the integrand.

\medskip
\noindent
- $\mathcal{D}^{1,2}(\mathbb{R}^N)$ denotes the closure of $C^\infty_c(\mathbb{R}^N)$ with respect to the norm $\|u\|_{\mathcal{D}^{1,2}(\mathbb{R}^N)}=\|\nabla u\|_{L^2}$.

\medskip
\noindent
- For a real symmetric matrix $A$, the notation $A\succ0$
(respectively, $A\succeq0$) means that $A$ is positive definite
(respectively, positive semidefinite).

\medskip
\noindent
- We will use big $O$ and small $o$ notations to describe the limit behavior of a certain quantity as $\epsilon\rightarrow0$.

\medskip
\noindent
- $C$ and $c$ are indiscriminately used to denote various absolutely positive constants.

\section{Preliminaries}\label{section2}
In this section, we are devoted to establishing some results necessary for our analysis in the sequel.

 \subsection{Existence of multi-bubble solutions}
In this section, in spirit of \cite{Micheletti}, we prove the existence of multi-bubble solutions for \eqref{ele-1.1},
\begin{proof}[Proof of Theorem \ref{existence-0}]
In the ranges $1<q<3$ if $N=4$, and $1\leq q<p$ if $N\geq5$, we have $\mathcal{B}_0<\infty$ and
$0<p-q<2$, so the scale $\lambda_{i\epsilon}\epsilon^{\alpha_0}$ is well defined.
For \textup{(i)}, choose $w\equiv1$ and $a_i=1$ for every $i$
in the reduction of \cite[Section~2]{Micheletti}. Its reduced functional
\cite[(3.6)]{Micheletti} is precisely \eqref{pusaik-0}. By
\cite[Propositions~2.7 and~3.3, Theorem~3.4]{Micheletti}, the odd extension
of \eqref{ele-1.1} has a solution with the decomposition in
\eqref{pusaik-1}. The parameters lie near $z_*=(\lambda^{\ast},\xi^{\ast})$, and every
accumulation point is a critical point of $\Phi_{\Omega,q}$.
Since $z_*$ is isolated, $(\lambda_\epsilon,\xi_\epsilon)\to z_*$.

We verify that this solution is positive. Write
$u_\epsilon=S_\epsilon+\omega_\varepsilon$, where
$S_\epsilon=\sum_i P
U_{\lambda_{i\epsilon}\epsilon^{\alpha_0},\xi_{i\epsilon}}>0$,
and set $v_\epsilon=u_\epsilon^-$. Then
$0\le v_\epsilon\le |\omega_\epsilon|$, so
$\|v_\epsilon\|_{2^*}=o(1)$. Testing the equation against
$-v_\epsilon$ gives
\[
\|\nabla v_\epsilon\|_2^2
\le C_\Omega\bigl(
 \|v_\epsilon\|_{2^*}^{p-1}
 +\epsilon\|v_\epsilon\|_{2^*}^{q-1}
 \bigr)\|\nabla v_\epsilon\|_2^2.
\]
The factor in parentheses tends to zero, including when $q=1$.
Thus $v_\epsilon=0$ for small $\epsilon$, and the strong
maximum principle yields $u_\epsilon>0$.

For \textup{(ii)}, let $\Omega_0=\bigcup_{i=1}^m B_i$, and denote
the center and radius of $B_i$ by $c_i$ and $R_i$. When
$\xi_i\in B_i$, the interactions between different balls vanish.
Moreover,
\[
\phi_{B_i}(x)
=\frac{1}{(N-2)|S^{N-1}|}\left(\frac{R_i}{R_i^2-|x-c_i|^2}\right)^{N-2},
\qquad
(\lambda_i^0)^{N-2-\frac{(N-2)(p-q)}{2}}
=\frac{\mathcal{B}_0(p-q)}{\mathcal{A}_0^2\,\phi_{B_i}(c_i)}.
\]
Since $0<p-q<2$, these formulas show that
$z_0=(\lambda^0,c_1,\ldots,c_m)$ is a nondegenerate local minimum
of $\Phi_{\Omega_0,q}$.
Connect the balls by thin tubes as in
\cite[proof of Theorem~1.4]{Micheletti}. The Green function convergence in
\cite[Lemma~3.5]{Micheletti} improves to local $C^2$ convergence: the
nonnegative difference $G_{\Omega_\delta}-G_{\Omega_0}$ is harmonic
on each $B_i\times B_j$ after cancellation of the diagonal
singularity, so Harnack's inequality and interior estimates apply.
Consequently, $\Phi_{\Omega_\eta,q}\to\Phi_{\Omega_0,q}$ in $C^2$
near $z_0$. For small $\delta$, a nondegenerate local minimum
$z_\delta\to z_0$ therefore persists. Fix such an $\delta$ and apply
\textup{(i)} on $\Omega_\delta$ gives positive $m$-bubble solutions whose
parameters converge to $z_\delta$ as $\varepsilon\to0$.
\end{proof}

\subsection{Local profiles and pointwise estimates}
For any $i\in\{1,\cdots,m\}$,
we set
\begin{equation}\label{miu}
\alpha_0=\frac{2}{N-6+q(N-2)}\quad \mbox{and}\quad\tau_{0}=\frac{N-2}{2},
\end{equation}
which implies that
\begin{equation}\label{beta1}
\epsilon\cdot(\epsilon^{\alpha_0\tau_{0}(p-q)})=(\epsilon^{2\alpha_0\tau_0}),\quad 2\tau_0-\tau_{0}(p-q)>0, \quad 2\tau_0(q+1)>N.
\end{equation}
Moreover, we have the following decomposition hold
\begin{equation}\label{lmiu}
u_{\epsilon}=W_{\lambda_{i\epsilon}\epsilon^{\alpha_0},\xi_{i\epsilon}}+\big(PW_{\lambda_{i\epsilon}\epsilon^{\alpha_0},\xi_{i\epsilon}}-W_{\lambda_{i\epsilon}\epsilon^{\alpha_0},\xi_{i\epsilon}}\big)+\sum_{i\neq j}PW_{\lambda_{i\epsilon}\epsilon^{\alpha_0},\xi_{i\epsilon}}+\omega_{\epsilon}.
\end{equation}
Then we rescale the solution $u_{\epsilon}$ by
\begin{equation}\label{miu-1}
U_{i\epsilon}(x)=(\lambda_{i\epsilon}\epsilon^{\alpha_0})^{\tau_0}u_{\epsilon}\left(\lambda_{i\epsilon}\epsilon^{\alpha_0}x+\xi_{i\epsilon}\right)\hspace{2mm}\mbox{for}\hspace{2mm} x\in\Omega_{i\epsilon}:=\lambda_{i\epsilon}^{-1}\epsilon^{-\alpha_0}(\Omega-\xi_{i\epsilon}).
\end{equation}	
We first state the following lemma which is useful in our analysis.
\begin{lem}\label{Lem6.1}
There exist a positive constant $C$ and a small constant $\sigma_0>0$ such that for $\epsilon>0$ sufficiently small,  there holds that
\begin{equation*}
\big\|U_{i\epsilon}\big\|_{L^{\infty}\big(B^{N}(0,\epsilon^{-\alpha_0}\sigma_0)\big)}\leq C,\quad
U_{i\epsilon}\rightharpoonup W\hspace{2mm}\mbox{weakly in}\hspace{2mm} \mathcal{D}^{1,2}(\mathbb{R}^N)\hspace{2mm}\mbox{as}\hspace{2mm}\epsilon\rightarrow0.
\end{equation*}
Moreover,
\begin{equation}\label{tidu}
\big|\nabla^{j}u_{\epsilon}\big|\leq C\sum_{i=1}^{m}\frac{\epsilon^{\frac{(N-2)\alpha_0}{2}}}{(\epsilon^{\alpha_0}+|x-\xi_{i\epsilon}|)^{N-2+j}},\hspace{4mm}j=0,1.
\end{equation}
The constants may depend on the fixed exponent $q$.
\end{lem}
\begin{proof}
Writing $\kappa_\epsilon=\epsilon(\lambda_{i\epsilon}\epsilon^{\alpha_0})^{\tau_0(p-q)}$, we have
\begin{equation}\label{eq:rescaling}
-\Delta U_{i\epsilon}=U_{i\epsilon}^p+\kappa_{\epsilon} U_{i\epsilon}^q.
\end{equation}
The rescaled bubble decomposition, as in the proof of
\cite[Lemma~2.2]{CKL-1}, gives
$U_{i\varepsilon}\rightharpoonup W$ in
$\mathcal D^{1,2}(\mathbb R^N)$.

The global bubble bound follows from
\cite[Theorem~1.1(2)(e)]{GM} when $q>1$, with their exponent
replaced by $q+1$, and from
\cite[Proposition~2.1(vi)]{K-L-2024} when $q=1$.
After matching the bubble parameters and using scale comparability,
these information give the estimate $\nabla^{0}u$.
The asserted $L^\infty$ bound for $U_{i,\varepsilon}$ follows
immediately from this estimate and the separation of the centers.

For $q=1$, the gradient estimate follows from
\cite[Proposition~B.2, (B.4)]{K-L-2024}, together with the core and
exterior $C^1$ estimates in \cite[Proposition~2.1(iv)--(v)]{K-L-2024}.
For $q>1$, we define
\begin{equation}\label{UU}
    \Omega_\varepsilon=\alpha^{-\alpha_0}\Omega,\qquad
    \widehat U_\epsilon(x)
    =\epsilon^{\alpha\tau_0} u_\varepsilon(\epsilon^{\alpha_0}x),
    \qquad
    X_{i\epsilon}=\epsilon^{-\alpha_0}\xi_{i\varepsilon}.
\end{equation}
Then
\[
    -\Delta\widehat U_\epsilon
    =\widehat U_\epsilon^p
     +\epsilon^{\alpha_{0}(N-2)}\widehat U_\epsilon^q.
\]
Since
$1+|x-X_{i\varepsilon}|\leq C/\epsilon^{\alpha_0}$ on $\Omega_\varepsilon$,
the preceding pointwise bound implies
\[
\begin{aligned}
    \widehat U_\epsilon^p
    +\epsilon^{\alpha_0(N-2)}\widehat U_\epsilon^q
    &\leq C\sum_i
    \left((1+|x-X_{i\varepsilon}|)^{-N-2}
          +\epsilon^{\alpha_0(N-2)}(1+|x-X_{i\varepsilon}|)^{-q(N-2)}\right)\\
    &\leq C\sum_i(1+|x-X_{i\varepsilon}|)^{-N-\gamma}.
\end{aligned}
\]
where $\gamma=\min\{2,(N-2)q-2\}>0$.
Using the Dirichlet Green representation and the standard estimate
$|\nabla_xG_{\Omega_\varepsilon}(x,y)|
\leq C|x-y|^{1-N}$, we obtain
\[
\begin{aligned}
    |\nabla_x\widehat U_\epsilon(x)|
    &\leq C\sum_i\int_{\mathbb R^N}
    \frac{(1+|y-x_{i\epsilon}|)^{-N-\gamma}}
         {|x-y|^{N-1}}\,dy\\
    &\leq C\sum_i
    (1+|x-X_{i\epsilon}|)^{1-N}.
\end{aligned}
\]
The last inequality follows by splitting the integral into
neighborhoods of $x$ and $X_{i\varepsilon}$ and their complement.
Returning to the original variables yields estimate of $\nabla^{1}u$,
which completes the proof.
\end{proof}

\subsection{A moving sphere argument}\label{moving}
We close the section by stating and proving the following local bubble bound by applying the moving sphere argument (refer to \cite{Li-Zhang,CKL-1}), which
is very useful in the sequel.
\begin{Prop}\label{prop:main}
Let $N\ge4$, let $\Omega\subset\mathbb{R}^N$ be a smooth bounded domain, and fix $1\le q<p$. Let $\lambda_{i\epsilon}\epsilon^{\alpha_0}\rightarrow0$
and
$\overline{B(\xi_{i\epsilon},\sigma_0)}\subset\Omega$ for some fixed $\sigma_0>0$. Assume that
\begin{equation}\label{eq:inputs}
\begin{gathered}
\sup_{i,\epsilon}\|U_{i\epsilon}\|_{L^\infty(B_{r_0/(\lambda_{i\epsilon}\epsilon^{\alpha_0})})}\le M,
\qquad U_{i\epsilon}\to W\quad\text{in }C^2_{\mathrm{loc}}(\mathbb{R}^N)
\end{gathered}
\end{equation}
for some $M>0$. Then there exist $C>0$ and $\varrho_0\in(0,\sigma_0)$, independent of sufficiently small $\epsilon$, such that
\begin{equation}\label{eq:bound}
U_{i\epsilon}(z)\le CW(z)\quad\mbox{for}\hspace{3mm}|z|\le \varrho_0/(\lambda_{i\epsilon}\epsilon^{\alpha_0}),
\qquad i=1,\ldots,m.
\end{equation}
Equivalently, $u_\epsilon\le CW_{\lambda_{i\epsilon}\epsilon^{\alpha_0},\xi_{i\epsilon}}$ in $B(\xi_{i\epsilon},\varrho_0)$.
\end{Prop}
\begin{proof}
Taking $R_\epsilon=r_0/(\lambda_{i\epsilon}\epsilon^{\alpha_0})$, then $U_{i\epsilon}$ solve \eqref{eq:rescaling} in $B_{R_\epsilon}$.
We note that $\omega_N A=[N(N-2)]^{(N-2)/4}$,
and since the boundary values of $U_{i\epsilon}$ are positive and $\kappa_\epsilon U_{i\epsilon}^q\ge0$. Thus,
as the argument of (A.10) in \cite{CKL-1}, for every $\eta>0$, we obtain
\begin{equation}\label{eq:lower}
U_{i\epsilon}(x)\ge(1-\eta)[N(N-2)]^{(N-2)/4}|x|^{2-N}
\quad \mbox{for}\hspace{2mm}T\le|x|\le cR_\epsilon,
\end{equation}
with fixed $T>1$ and $c\in(0,1)$ depending on $\eta$.

Now the Kelvin transform $U_{i\epsilon}^{\lambda}$ of $U_{i\epsilon}$ defined by
$U_{i\epsilon}^\lambda(x)=(\lambda/|x|)^{N-2}
 U_{i\epsilon(\lambda^2x/|x|^2}),$
satisfies
\begin{equation}
-\Delta U_{i\epsilon}^\lambda=(U_{i\epsilon}^\lambda)^p+
\kappa_\epsilon\left(\frac{\lambda}{|x|}\right)^{2\tau_0(p-q)}(U_{i\epsilon}^\lambda)^q.
\end{equation}
Thus, for $|x|\ge\lambda$, $\omega_{i\epsilon}^\lambda:=U_{i\epsilon}-U_{i\epsilon}^\lambda$ satisfies
$
-\Delta\omega_{i\epsilon}^\lambda\ge c_{\epsilon,\lambda}(x)\omega_{i\epsilon}^\lambda,
$
where $c_{\epsilon,\lambda}\geq0$ is the difference quotient of $t^p+\kappa_\epsilon t^q$.
This is the inequality used in \cite[Lemma~A.2]{CKL-1}. Together with \eqref{eq:lower} and local $C^1$ convergence, it gives
\begin{equation}\label{eq:minimum}
\min_{|x|=r}w_{i\epsilon}(x)\le(1+\zeta)W(r)
\quad \mbox{for}\quad 0<r\le c_\zeta R_\epsilon
\end{equation}
for every fixed $\zeta>0$ and some $c_\zeta\in(0,1)$.
Indeed, the same contradiction argument starts the sphere at $\lambda_-<1$; a failure of \eqref{eq:minimum} keeps the outer boundary comparison strict, so the continuation argument in that lemma moves it to $\lambda_+>1$. The limit would satisfy $W\ge W^{\lambda_+}$ outside $B_{\lambda_+}$, contradicting the explicit formula for $W$.

We next pass from the spherical minimum to a pointwise bound, following the Green-function step of \cite[Lemma~2.4]{Li-Zhang}. We record the estimates needed for a uniform Harnack inequality. Fix a small $\delta_0>0$, set $L_\epsilon=cR_\epsilon$ with $c\le c_h$, and choose a minimum point $x_\epsilon$ on $\partial B_{L_\epsilon}$. By \eqref{eq:minimum},
\[
U_{i\epsilon}(x_\epsilon)\le(1+\delta_0)[N(N-2)]^{(N-2)/4}L_\epsilon^{2-N}.
\]
After taking $\theta,c>0$ sufficiently small, we have
\[
G_{R_\epsilon}(x_\epsilon,y)\ge(1-\delta_0)\omega_NL_\epsilon^{2-N}
\quad\mbox{for}\quad |y|\le\theta L_\epsilon.
\]
Consequently,
\begin{equation}\label{eq:mass}
\int_{B_{\theta S_\epsilon}}U_{i\epsilon}^p\le\frac{1+\delta_0}{1-\delta_0}\mathcal{A}_0.
\end{equation}
For a sufficiently large fixed $T_1$, local convergence gives
$\int_{B_{T_1}}U_{i\epsilon}^p\ge(1-2\delta_0)\mathcal{A}_0$ for all small $\epsilon$.
Subtracting this from \eqref{eq:mass}, and using $U_{i\epsilon}^{p+1}\le M U_{i\epsilon}^p$, shows that for every $\rho>0$, we can choose fixed $T_1>1$ and $d=\theta c>0$ such that
\begin{equation}\label{eq:tail}
\int_{B_{dR_\epsilon}\setminus B_{T_1}}U_\epsilon^{p+1}\le\rho.
\end{equation}

For $4T_1\le r\le dR_\epsilon/4$, let $V(z)=r^{\frac{N-2}{2}} U_{i\epsilon}(rz)$ on
$\mathcal{D}_0=\{1/4<|z|<4\}$. Then \eqref{eq:tail} gives
$\int_{\mathcal D_0}V^{p+1}\le\omega$.
The rescaled BN equation is
\begin{equation}\label{eq:annulus}
-\Delta V=V^p+b_{\epsilon,r}V^q,\qquad
0\le b_{\epsilon,r}=\kappa_\epsilon r^{\tau_0(p-q)}
=\epsilon(\lambda_{i\epsilon}\epsilon^{\alpha_0} r)^{\tau_0(p-q)}\le\epsilon r_0^{\tau_0(p-q)}.
\end{equation}
For $0<\epsilon\le1$, taking $K_2=\sigma_0^{\tau_0(p-q)}$. Since $1\le q<p$, its linear potential satisfies
\begin{equation}\label{eq:potential}
0\le c_V:=V^{p-1}+b_{\epsilon,r}V^{q-1}
\le K_2+(1+K_2)V^{p-1}.
\end{equation}
To apply Harnack's inequality, first take a cutoff $\chi$ in $\mathcal D_0$, equal to one on $\mathcal D_1=\{1/3<|z|<3\}$, and test the equation with $\chi^2V^{p}$. Sobolev and H\"older inequalities give
\[
\|\chi V^{\frac{p+1}{2}}\|_{L^{p+1}}^2
\le C\omega+C(1+K)\omega^{\frac{2}{N}}
\|\chi V^{\frac{(p+1)}{2}}\|_{L^{p+1}}^2.
\]
Choosing $\rho$ small and absorbing the last term bounds $V$ in $L^s(\mathcal D_1)$, where
\[
s=\frac{2N^2}{(N-2)^2},\qquad
\frac{s}{p-1}=\frac{N^2}{2(N-2)}>\frac N2.
\]
Thus \eqref{eq:potential} bounds $c_V$ in $L^\nu$ for some $\nu>N/2$. Local boundedness for $-\Delta V=c_VV$ gives a uniform bound for $V$, and hence for $c_V$, on $\{1/2<|z|<2\}$. Harnack's inequality along the unit sphere now yields
\[
\max_{|x|=r}U_{i\epsilon}(x)\le C\min_{|x|=r}U_{i\epsilon}(x)\le CW(r),
\]
where the last inequality follows from \eqref{eq:minimum}.
Local convergence handles $|x|\le4T_1$.
Hence \eqref{eq:bound} holds with $\varrho_0=\sigma_0d/4$ for the fixed index. Taking the minimum of these radii and the maximum of the constants over the finitely many indices completes the proof.
\end{proof}

\section{Limiting spectrum and energy decomposition}\label{sec3}
In this section, we assume the hypotheses of
Theorem~\ref{THm1-2}. We establish the spectral limits and a strong energy decomposition, and then derive coercivity estimates under suitable orthogonality
conditions. We next recall some well-know results.
Set
\[
 W_{i\epsilon}:=W_{\delta_{i\epsilon},\xi_{i\epsilon}},
 \qquad f_\epsilon(t)=t^p+\epsilon t^q.
\]
For a positive $g\in L^{N/2}(\Omega)$, the embedding
$H_0^1(\Omega)\hookrightarrow L^2(\Omega,g\,dx)$ is compact by
\cite[Proposition~A.1]{FigalliGlaudo2020}, applied to zero extensions
of the functions and a positive $L^{N/2}$ extension of $g$.
Consequently, the operator defined by
$(K_gv,z)_{H_0^1}=\int_\Omega gvz$ is compact, positive, self-adjoint
and injective. The compact spectral theorem
\cite[Theorem~6.11]{Brezis2011} and the min--max principle,
applied with $g=f_\epsilon'(u_\epsilon)$, give
\begin{equation}\label{bn:eq-weighted-minmax}
 \mu_{l\epsilon}
 =\inf_{\substack{S\subset H_0^1(\Omega)\\\dim S=l}}
   \sup_{0\ne v\in S}
   \frac{\int_\Omega|\nabla v|^2}
        {\int_\Omega f_\epsilon'(u_\epsilon)v^2}.
\end{equation}
Choose a complete energy-orthonormal eigenbasis
$\{\widehat v_{l\epsilon}\}_{l\ge1}$. Then
\begin{equation}\label{bn:eq-eigenbasis-normalization}
 (\widehat v_{l\epsilon},\widehat v_{k\epsilon})_{H_0^1}
 =\delta_{lk},\qquad
 \int_\Omega f_\epsilon'(u_\epsilon)
       \widehat v_{l\epsilon}\widehat v_{k\epsilon}
 =\frac{\delta_{lk}}{\mu_{l\epsilon}}.
\end{equation}
For the solution families considered later, this is the eigenbasis
used in the introduction, with energy rather than $L^\infty$ normalization.

Recall that the weighted eigenvalue problem
$-\Delta\phi=\mu pW^{p-1}\phi$ in $\mathcal{D}^{1,2}(\mathbb R^N)$
has eigenvalues $\mu_s^\star$ with multiplicities $h_s$.
For the eigenvalue formula, with the normalization
$pW^{p-1}=N(N+2)(1+|y|^2)^{-2}$, see
\cite[proof of Corollary~6 and Appendix~B]{DolbeaultJankowiak2014}.
The multiplicities and completeness follow from the corresponding
results for spherical harmonics; see
\cite[Proposition~5.8 and Theorem~5.12]{AxlerBourdonRamey2001}.
The first two eigenspaces are
\begin{equation}\label{bn:eq-first-model-spaces}
 \operatorname{span}\{W\},\qquad
 \operatorname{span}\{\partial_1W,\ldots,\partial_NW,Z_0\},
 \qquad Z_0=-y\cdot\nabla W-\frac{N-2}{2}W,
\end{equation}
at $1/p$ and $1$, respectively; see
\cite[Lemma~A1]{BE91}.
On $\mathscr D=(D^{1,2}(\mathbb R^N))^m$, the energy and weight forms
are
\[
 \sum_{i=1}^m\int_{\mathbb R^N}\nabla\phi_i\cdot\nabla\psi_i,
 \qquad
 p\sum_{i=1}^m\int_{\mathbb R^N}W^{p-1}\phi_i\psi_i.
\]
The ordered spectrum of this direct sum is $(\mu_l^0)_{l\ge1}$,
with multiplicity $mh_s$ at $\mu_s^\star$.

For $\phi\in D^{1,2}(\mathbb R^N)$, define the energy-preserving map
\[
 (T_{i\epsilon}\phi)(x)
 =\delta_{i\epsilon}^{-\tau_0}
   \phi\left(\frac{x-\xi_{i\epsilon}}{\delta_{i\epsilon}}\right).
\]
Choose $\chi_i\in C_c^\infty(\Omega)$, equal to one near $\xi_i$,
with pairwise disjoint supports.
By the Sobolev and H\"older inequalities,
\begin{equation}\label{bn:eq-critical-form-bound}
 \left|\int gvz\right|
 \le C\|g\|_{N/2}\|\nabla v\|_2\|\nabla z\|_2,
 \qquad g\in L^{N/2}.
\end{equation}
This estimate holds for $v,z\in H_0^1(\Omega)$ and for
$v,z\in \mathcal{D}^{1,2}(\mathbb R^N)$, with the integrals and norms
taken over the corresponding domain.
We first reduce the potential to a sum of single-bubble weights
and establish the corresponding convergence of the weighted forms.
\begin{lem}
\label{bn:lem-potential-decomposition}
Under the assumptions of Theorem~\ref{THm1-2},
\begin{equation}\label{bn:eq-potential-error}
 \mathcal R_\epsilon
 :=f_\epsilon'(u_\epsilon)-p\sum_{i=1}^mW_{i\epsilon}^{p-1}
 \longrightarrow0\quad\text{in }L^{N/2}(\Omega).
\end{equation}
In particular, $\|f_\epsilon'(u_\epsilon)\|_{N/2}$ is uniformly
bounded and $\mu_{l\epsilon}\ge c_*>0$, uniformly in $l$ and small
$\epsilon$. Moreover, if $v_\epsilon,z_\epsilon$ are bounded
in $H_0^1(\Omega)$ and, along a common subsequence,
\[
 T_{i\epsilon}^{-1}v_\epsilon\rightharpoonup v_i,\qquad
 T_{i\epsilon}^{-1}z_\epsilon\rightharpoonup z_i
 \quad\text{in }\mathcal{D}^{1,2}(\mathbb R^N),
\]
then
\begin{equation}\label{bn:eq-weighted-capture}
 \int_\Omega f_\epsilon'(u_\epsilon)v_\epsilon z_\epsilon
 \longrightarrow
 p\sum_{i=1}^m\int_{\mathbb R^N}W^{p-1}v_i z_i.
\end{equation}
\end{lem}

\begin{proof}
Since the centers remain away from $\partial\Omega$, the maximum
principle for $W_{i\epsilon}-PW_{i\epsilon}$ gives
$0\le W_{i\epsilon}-PW_{i\epsilon}\le C\delta_{i\epsilon}^{\tau_0}$.
The assumed decomposition and Sobolev's inequality imply
\[
 \sup_\epsilon\|u_\epsilon\|_{2^*}<\infty,\qquad
 \left\|u_\epsilon-\sum_iW_{i\epsilon}\right\|_{p+1}\to0.
\]
The power map $a\mapsto a^{p-1}$ is continuous from the nonnegative
cone of $L^{p+1}$ to $L^{N/2}$, uniformly on bounded sets. Indeed, for
$a,b\ge0$,
\[
 \|a^{p-1}-b^{p-1}\|_{N/2}
 \le
 \begin{cases}
 \|a-b\|_{p+1}^{p-1},&p-1\le1,\\
 C(\|a\|_{p+1}+\|b\|_{p+1})^{p-2}\|a-b\|_{p+1},&p-1>1.
 \end{cases}
\]
Fix $R>1$ and truncate $W_{i\epsilon}$ to
$B(\xi_{i\epsilon},R\delta_{i\epsilon})$. For small
$\epsilon$ these supports are disjoint, so the power of their sum
is the sum of their powers. Each discarded $L^{2^*}$ tail is bounded
by $\|W\mathbf1_{\{|y|>R\}}\|_{p+1}$. The preceding continuity
estimate, first with $\epsilon\to0$ and then $R\to\infty$, yields
\[
 \left\|u_\epsilon^{p-1}-\sum_iW_{i\epsilon}^{p-1}\right\|_{N/2}
 \to0.
\]
For $q>1$,
\[
 \|\epsilon q u_\epsilon^{q-1}\|_{N/2}
 \le\epsilon q|\Omega|^{2/N-(q-1)/(p+1)}
       \|u_\epsilon\|_{p+1}^{q-1}\to0;
\]
for $q=1$ the norm is $\epsilon|\Omega|^{2/N}$. This proves
\eqref{bn:eq-potential-error}. By scaling \eqref{bn:eq-weighted-minmax} and \eqref{bn:eq-critical-form-bound} give the uniform bounds.

For the last assertion, apply the compact embedding
$\mathcal{D}^{1,2}(\mathbb R^N)\hookrightarrow L^2(\mathbb R^N,pW^{p-1}\,dy)$
from \cite[Proposition~A.1]{FigalliGlaudo2020} to the rescaled
sequences. Their weak energy convergence becomes strong convergence
in this weighted space. Changing variables in each bubble term and
using \eqref{bn:eq-potential-error}, Sobolev and H\"{o}lder's inequalities prove \eqref{bn:eq-weighted-capture}.
\end{proof}

We next consider eigenfunctions whose eigenvalues remain bounded
as $\epsilon\to0$. The convergence of the weighted forms
in Lemma~\ref{bn:lem-potential-decomposition} allows us to identify their limiting profiles
and establish a strong energy decomposition.
\begin{Prop}
\label{bn:prop-strong-profile}
Let $J$ be fixed and suppose that
$\mu_{l\epsilon}$ is bounded for $1\le l\le J$. From any sequence
$\epsilon\to0$ one can extract a common subsequence such that
$\mu_{l\epsilon}\to\nu_l>0$ and
\[
 \widehat V_{li\epsilon}
 :=T_{i\epsilon}^{-1}\widehat v_{l\epsilon}
 \rightharpoonup\widehat V_{li}
 \quad\text{in }D^{1,2}(\mathbb R^N).
\]
For $1\le l,k\le J$, these profiles satisfy
\begin{align}
 -\Delta\widehat V_{li}
 &=\nu_l pW^{p-1}\widehat V_{li}
       \quad\text{in }\mathbb R^N,
       \label{bn:eq-limit-profile-eqn}\\
 p\sum_i\int_{\mathbb R^N}W^{p-1}\widehat V_{li}\widehat V_{ki}
 &=\frac{\delta_{lk}}{\nu_l},
       \label{bn:eq-profile-weighted-orthogonality}\\
 \sum_i\int_{\mathbb R^N}\nabla\widehat V_{li}\cdot\nabla\widehat V_{ki}
 &=\delta_{lk},
       \label{bn:eq-profile-energy-orthogonality}
\end{align}
and
\begin{equation}\label{bn:eq-strong-energy-decomposition}
 \left\|\widehat v_{l\epsilon}
       -\sum_i\chi_iT_{i\epsilon}\widehat V_{li}\right\|_{H_0^1(\Omega)}
 \longrightarrow0\qquad(1\le l\le J).
\end{equation}
\end{Prop}

\begin{proof}
The eigenvalue lower bound in Lemma~\ref{bn:lem-potential-decomposition}
and the energy normalization give a common subsequence with the
asserted limits. To identify the profile equation, take
$\phi\in C_c^\infty(B_R)$ and test the eigenvalue equation with
$T_{i\epsilon}\phi\in H_0^1(\Omega)$. For $j\ne i$, the distinct
limiting centers give a fixed $d>0$ such that
\begin{equation}\label{bn:eq-cross-potential-tail}
 \left\|pW_{j\epsilon}^{p-1}
       \mathbf1_{B(\xi_{i\epsilon},R\delta_{i\epsilon})}\right\|_{N/2}
 \le
 \left\|pW^{p-1}\mathbf1_{\{|y|>d/\delta_{j\epsilon}\}}\right\|_{N/2}
 \longrightarrow0.
\end{equation}
Thus all other-center terms vanish by \eqref{bn:eq-critical-form-bound}; the remainder vanishes by
\eqref{bn:eq-potential-error}. The same-center term passes to the
limit by weak energy convergence and
\eqref{bn:eq-critical-form-bound}. Density gives
\eqref{bn:eq-limit-profile-eqn} for all finite-energy test functions.
Applying \eqref{bn:eq-weighted-capture} to
\eqref{bn:eq-eigenbasis-normalization} yields
\eqref{bn:eq-profile-weighted-orthogonality}. Testing the profile
equation against $\widehat V_{ki}$ and summing in $i$ then gives
\eqref{bn:eq-profile-energy-orthogonality}.

For each fixed $\phi\in \mathcal{D}^{1,2}(\mathbb R^N)$,
\begin{equation}\label{bn:eq-cutoff-energy-approx}
 \|\chi_iT_{i\epsilon}\phi-T_{i\epsilon}\phi\|_{\mathcal{D}^{1,2}(\mathbb R^N)}
 \longrightarrow0.
\end{equation}
The gradient term is bounded by the tail of $\nabla\phi$; the cutoff
term satisfies
\[
 \|\nabla\chi_i\,T_{i\epsilon}\phi\|_2
 \le\|\nabla\chi_i\|_N
      \|T_{i\epsilon}\phi\|_{L^{p+1}(\operatorname{supp}\nabla\chi_i)}
 \longrightarrow0,
\]
since $\operatorname{supp}\nabla\chi_i$ stays away from
$\xi_{i\epsilon}$.
Set $E_{l\epsilon}=\sum_i\chi_iT_{i\epsilon}\widehat V_{li}$.
By disjoint supports, \eqref{bn:eq-cutoff-energy-approx} and
\eqref{bn:eq-profile-energy-orthogonality},
$\|E_{l\epsilon}\|_{H_0^1}^2\to1$. Rescaled weak convergence
also gives $(\widehat v_{l\epsilon},E_{l\epsilon})_{H_0^1}\to1$.
Expanding the squared norm of the difference proves
\eqref{bn:eq-strong-energy-decomposition}.
\end{proof}

The strong convergence in \eqref{bn:eq-strong-energy-decomposition}
is a decomposition in the original energy space. An individual
rescaling need not converge strongly on all of $\mathbb R^N$, since
the other bubbles may carry energy to infinity in those coordinates.

We now prove Theorem~\ref{THm1-2} using the min--max principle
and the compactness established above.

\begin{proof}[Proof of Theorem~\ref{THm1-2}]
Fix $l\in\mathbb N$.
For the upper bound, approximate the first $l$ energy-orthonormal
vectors of the direct-sum model by
$\psi_1,\ldots,\psi_l\in(C_c^\infty(\mathbb R^N))^m$.
Their energy and weight Gram matrices can be made arbitrarily close
to $I_l$ and
$D_l=\operatorname{diag}((\mu_1^0)^{-1},\ldots,(\mu_l^0)^{-1})$.
Transplant them by
\[
 J_\epsilon\psi_a=\sum_{i=1}^mT_{i\epsilon}(\psi_a)_i,
 \qquad 1\le a\le l.
\]
For small $\epsilon$, the component supports are disjoint and
contained in $\Omega$, so their energy Gram matrix is unchanged.
By \eqref{bn:eq-potential-error} and
\eqref{bn:eq-cross-potential-tail}, their weighted Gram matrix
converges to that of the compactly supported model vectors.
Hence, given small $\eta>0$, the approximation and then
$\epsilon$ can be chosen so that these matrices $A_\varepsilon$
and $C_\varepsilon$ satisfy
$\|A_\epsilon-I_l\|_{\rm op}<\eta$ and
$\|C_\epsilon-D_l\|_{\rm op}<\eta$.
The trial space has dimension $l$, and its largest Rayleigh quotient
is bounded by
\[
 \sup_{c\ne0}\frac{c^{\mathsf T}A_\varepsilon c}
                         {c^{\mathsf T}C_\varepsilon c}
 \le\frac{1+\eta}{(\mu_l^0)^{-1}-\eta}.
\]
The min--max principle and $\eta\downarrow0$ give
\begin{equation}\label{bn:eq-spectral-upper}
 \limsup_{\epsilon\to0}\mu_{l\epsilon}\le\mu_l^0.
\end{equation}

For the lower bound, choose a sequence on which $\mu_{l\varepsilon}$
tends to its lower limit. By \eqref{bn:eq-spectral-upper}, the first
$l$ eigenvalues are bounded. Proposition~\ref{bn:prop-strong-profile}
provides a further subsequence with
\[
 0<\nu_1\le\cdots\le\nu_l
       =\liminf_{\epsilon\to0}\mu_{l\epsilon}.
\]
The corresponding direct-sum profiles are energy-orthonormal by
\eqref{bn:eq-profile-energy-orthogonality}, and each is a model
eigenvector at $\nu_k$. Thus the model spectral subspace below or at
$\nu_l$ has dimension at least $l$. Its ordered spectrum is
$(\mu_k^0)$, so $\mu_l^0\le\nu_l$. Together with
\eqref{bn:eq-spectral-upper}, this proves
$\mu_{l\epsilon}\to\mu_l^0$ for every fixed $l$.
The stated cluster sizes follow from the model multiplicities $mh_s$.
\end{proof}

In particular, Proposition~\ref{bn:prop-strong-profile} applies to
every fixed finite set of indices with $\nu_l=\mu_l^0$.
The eigenvalue limits hold for the full family, whereas the profiles
and the strong decomposition are obtained along common subsequences.

As a consequence, we obtain coercivity on the complements
of the low eigenspace and the explicit bubble modes.
Let $K=(N+2)m$ and define
\[
 B_\epsilon(v,z)
 =\int_\Omega\nabla v\cdot\nabla z
       -\int_\Omega f_\epsilon'(u_\epsilon)vz,
 \qquad
 F_\epsilon=\operatorname{span}\{\widehat v_{1\epsilon},\ldots,
                                  \widehat v_{K\epsilon}\}.
\]
Set
\[
 E_{\le1}=\operatorname{span}\{W,\partial_1W,\ldots,\partial_NW,Z_0\},
 \qquad
 \mathcal Z_\epsilon
 =\operatorname{span}\{\chi_iT_{i\epsilon}\phi:
                 1\le i\le m,\ \phi\in E_{\le1}\}.
\]
By \eqref{bn:eq-cutoff-energy-approx},
$\dim\mathcal Z_\epsilon=K$ for small $\epsilon$.
For a finite-dimensional subspace $E\subset H_0^1(\Omega)$, write
\[
 c_\epsilon(E)
 =\inf_{\substack{0\ne w\in H_0^1(\Omega)\\w\perp E}}
     \frac{B_\epsilon(w,w)}{\|\nabla w\|_2^2},
\]
where orthogonality is in the energy inner product.

Coercivity away from the bubble modes is classical
\cite[Proposition~3.10]{FigalliGlaudo2020}; see also
\cite[proof of Lemma~6.2]{DengSunWei2025} for its use with small
energy perturbations. The following consequence records the exact
limiting constant for both complements in the present setting.

\begin{cor}\label{bn:cor-coercivity}
As $\epsilon\to0$,
\begin{equation}\label{bn:eq-complement-coercivity}
 c_\epsilon(F_\epsilon)
 =1-\frac1{\mu_{(K+1)\epsilon}}\longrightarrow\frac4{N+4},
 \qquad
 c_\epsilon(\mathcal Z_\epsilon)\longrightarrow\frac4{N+4}.
\end{equation}
Consequently, for every $0<c<4/(N+4)$ and sufficiently small
$\epsilon$,
$B_\epsilon(w,w)\ge c\|\nabla w\|_2^2$ on either complement.
\end{cor}

\begin{proof}
For $w=\sum_{l>K}c_l\widehat v_{l\epsilon}$, the spectral
expansion gives
\[
 B_\varepsilon(w,w)=\sum_{l>K}(1-\mu_{l\epsilon}^{-1})c_l^2.
\]
Its normalized minimum is attained at
$\widehat v_{K+1,\epsilon}$. Since
$\mu_{K+1,\epsilon}\to(N+4)/N$, the first assertion follows.

Let $P_\epsilon$ and $\Pi_\epsilon$ be the energy-orthogonal
projections onto $F_\epsilon$ and $\mathcal Z_\epsilon$.
We claim that
\begin{equation}\label{bn:eq-projection-gap}
 \eta_\epsilon
 :=\|P_\epsilon-\Pi_\epsilon\|_{H_0^1\to H_0^1}\longrightarrow0.
\end{equation}
From any sequence $\epsilon\to0$, apply
Proposition~\ref{bn:prop-strong-profile} with $J=K$.
All its profiles lie in $E_{\le1}$, because their eigenvalues are
$1/p$ or $1$. Hence their approximants $E_{l\varepsilon}$ belong
to $\mathcal Z_\epsilon$ and satisfy
\[
 \max_{l\le K}\|E_{l\epsilon}-\widehat v_{l\epsilon}\|_{H_0^1}\to0.
\]
Their Gram matrix tends to $I_K$. Multiplication by its inverse
square root produces an orthonormal basis of
$\mathcal Z_\epsilon$ still $o(1)$-close to
$\{\widehat v_{l\epsilon}\}_{l\le K}$. The finite-rank
projection formulas yield \eqref{bn:eq-projection-gap} on this
subsequence; the subsequence criterion gives it for the full family.

Using \eqref{bn:eq-critical-form-bound} and Lemma~\ref{bn:lem-potential-decomposition}, we have
$|B_\epsilon(v,z)|\le C\|v\|_{H_0^1}\|z\|_{H_0^1}$ uniformly.
If $\|w\|_{H_0^1}=1$ and $w\perp\mathcal Z_\epsilon$, then
$(I-P_\epsilon)w\in F_\epsilon^\perp$ differs from $w$ by
at most $\eta_\epsilon$. After normalization, their quadratic
forms differ by at most $C\eta_\epsilon$. Interchanging the
two spaces gives
\[
 |c_\epsilon(\mathcal Z_\epsilon)
       -c_\epsilon(F_\epsilon)|\le C\eta_\epsilon\to0,
\]
which proves the second assertion.
\end{proof}

\section{A sharp decay estimate of the eigenfunctions}\label{section-3}
We study the behavior of solutions $u_{\epsilon}$ of \eqref{ele-1.1} outside the blow-up points $\{\xi_1,\cdots,\xi_m\}$. Introducing the set in what follows
$$\Sigma_r:=\Omega\setminus\bigcup_{i=1}^{m}B(\xi_i,r)\quad\mbox{for any}\quad r>0.$$

The following lemma will play a key role in our analysis.
\begin{lem}\label{outside-1}
Assume that $\{u_{\epsilon}\}_\epsilon$ is a family of solutions of \eqref{ele-1.1} having the form \eqref{pusaik-1}. Then for any $r>0$ we have
$u_{\epsilon}\rightarrow0$ as $\epsilon\rightarrow0$ uniformly for $x\in\Sigma_r$.
\end{lem}
\begin{proof}
In virtue of Lemma 6-7 from \cite{HANZCHAO} and \eqref{pusaik-1}, we have
\begin{equation*}
\begin{split}
\int_{\Sigma_{r/4}}|u_{\epsilon}^{p-1}+\epsilon u_{\epsilon}^{q-1}|^{\frac{N}{2}}dx&\leq C\sum_{i=1}^m\Big(\int_{\Sigma_{r/4}}PW^{\frac{2N}{N-2}}_{\lambda_{i\epsilon}\epsilon^{\alpha_0},\xi_{i\epsilon}}dx
+\epsilon\int_{\Sigma_{r/4}}PW^{\frac{(q-1)N}{2}}_{\lambda_{i\epsilon}\epsilon^{\alpha_0},\xi_{i\epsilon}}dx\Big)\\&
+\int_{\Sigma_{r/4}}|R_{\epsilon}|^{\frac{2N}{N-2}}dx+\epsilon\int_{\Sigma_{r/4}}|R_{\epsilon}|^{\frac{(q-1)N}{2}}dx\\&
=O(\epsilon^{\frac{(N-2)(q-1)N}{N-6+q(N-2)}})+O(\epsilon^{\frac{2N}{N-6+q(N-2)}+1})+o(1)=o(1).
\end{split}
\end{equation*}
Choosing $\delta^{\prime}=\frac{4}{N}$, we deduce that
$
\int_{\Sigma_{r/2}}\big|u_{\epsilon}^{p-1}+\epsilon u_{\epsilon}^{q-1}\big|^{\frac{N}{2-\delta^{\prime}}}dx=o(1).
$
Applying Lemma 7 of \cite{HANZCHAO}  gives $\|u_{\epsilon}\|_{L^{\infty}(\Sigma_{r})}=o(1)$ and concluding the proof.
\end{proof}

\begin{lem}\label{outside-2}
We have the following estimates, uniformly for any $x\in\Sigma_{r}$ with any fixed $r>0$ small
\begin{equation*}
u_{\epsilon}(x)=O(\epsilon^{\frac{N-2}{N-6+q(N-2)}}).
\end{equation*}
\end{lem}
\begin{proof}
We proceed similarly to \cite{CKL-1}. Choosing $r\in(0,\varrho_0)$ and for any fixed $i\in\{1,\cdots,m\}$, by \eqref{eq:bound} we have
$$u_{\epsilon}(x)=O(\epsilon^{\frac{N-2}{N-6+q(N-2)}})\quad \mbox{for all}\quad\Sigma_{\frac{r}{2}}\setminus\Sigma_{r}.$$
Moreover, from the integral representation of $u_{\epsilon}$, we get that
\begin{equation}\label{green1}
u_{\epsilon}(x)=
\int_{\Sigma_{r/2}}\big(u_{\epsilon}^{p}(y)+\epsilon u_{\epsilon}^{q}(y)\big)G(x,y)dy+\sum_{i=1}^m\int_{ B(\xi_{i},r/2)}\big(u_{\epsilon}^{p}(y)+\epsilon u_{\epsilon}^{q}(y)\big)G(x,y)dy
\end{equation}
We now may write $\mathcal{A}_{\epsilon}=\max\{u_{\epsilon}(x):~x\in\Sigma_{r}\}$ and $\mathcal{B}(r)=\max\{G(x,y):~x,y\in\Sigma_{r},~|x-y|\geq r/2\}$, then by exploiting \eqref{eq:bound}, we infer
\begin{equation}\label{outside-3}
\begin{split}
\int_{\Sigma_{r/2}}\big(u_{\epsilon}^{p}(y)+\epsilon u_{\epsilon}^{q}(y)\big)G(x,y)dy&
\leq C\Big(\mathcal{A}^p_{\epsilon}+\epsilon\mathcal{A}^q_{\epsilon}+C \epsilon^{\frac{(N-2)p}{N-6+q(N-2)}}+C\epsilon^{\frac{(N-2)q}{N-6+q(N-2)}+1}\Big)
\end{split}
\end{equation}
for all $x\in\Sigma_{r/2}$, and
\begin{equation}\label{outside-4}
\begin{split}
\int_{ B(\xi_{i},r/2)}\big(u_{\epsilon}^{p}(y)+\epsilon u_{\epsilon}^{q}(y)\big)G(x,y)dy
\leq C\mathcal{B}(r)\epsilon^{\frac{N-2}{N-6+q(N-2)}}.
\end{split}
\end{equation}
Indeed, the inequality holds since
\begin{equation*}
\begin{split}
\int_{B(0,(\lambda_{i\epsilon}^{-1}\epsilon^{-\alpha_0})r/2)}W^{q}dy\leq
\left\lbrace
\begin{aligned}
&
C(\epsilon^{\alpha_0})^{\frac{(N-2)q}{2}}, \quad\quad\quad\quad\quad\quad\quad\hspace{2mm}\mbox{if}\quad q<N/(N-2),\\
&C(\epsilon^{\alpha_0})^{\frac{2N-(N-2)q}{2}}|\log (\epsilon)|, \quad\quad\quad\mbox{if}\quad q=N/(N-2),\\
&C(\epsilon^{\alpha_0})^{\frac{2N-(N-2)q}{2}},\quad\quad \quad\quad\quad\quad\hspace{1mm}\mbox{if}\quad q>N/(N-2),
\end{aligned}
\right.
\end{split}
\end{equation*}
As a consequence, we obtain
$$
\mathcal{A}_{\epsilon}\leq C\Big(\mathcal{A}^p_{\epsilon}+\epsilon\mathcal{A}^q_{\epsilon}+C \epsilon^{\frac{N-2}{N-6+q(N-2)}}\Big).
$$
Notice that $\mathcal{A}_{\epsilon}\rightarrow0$ as $\epsilon\rightarrow0$ by Lemma \ref{outside-1}, Then we eventually deduce that $\mathcal{A}_{\epsilon}\leq C\epsilon^{\frac{N-2}{N-6+q(N-2)}}$, which is the desired result.
\end{proof}

\begin{lem}\label{2-AS-eigenvalues}
Assume that $p=(N+2)/(N-2)$, $1<q<3$ if $N=4$, and $1\leq q<p$ if $N\geq5$. Then
\begin{equation}\label{AS-eigenvalues-1}
\frac{u_{\epsilon}(x)}{\epsilon^{(N-2)\alpha_0/2}}=\mathcal{A}_{0}\sum_{i=1}^mG(x,\xi_i)\lambda_{i\epsilon}^{\frac{N-2}{2}}+o(1),
\end{equation}
where the convergence is in $C^{2}(\Omega\setminus\{\xi_1,\cdots,\xi_m\})$ with $\{\xi_1,\cdots,\xi_m\}$ are the blow-up points of $u_\epsilon$ and $\mathcal{A}_0$ is defined in \eqref{pusaik-1-1}.
\end{lem}
\begin{proof}
Recalling \eqref{green1}, for $i\in\{1,\cdots,m\}$, we have that
\begin{equation*}
\begin{split}
\epsilon^{-\alpha_0\tau_0}u_{\epsilon}(x)=&
\epsilon^{-\alpha_0\tau_0}\int_{\Sigma_{r/2}}\big(u_{\epsilon}^{p}(y)+\epsilon u_{\epsilon}^{q}(y)\big)G(x,y)dy+\epsilon^{-\alpha_0\tau_0}\sum_{i=1}^mG(x,\xi_i)\int_{ B(\xi_{i},r/2)}\big(u_{\epsilon}^{p}(y)+\epsilon u_{\epsilon}^{q}(y)\big)dy\\&
+\epsilon^{-\alpha_0\tau_0}\sum_{i=1}^m\int_{ B(\xi_{i},r/2)}\big(u_{\epsilon}^{p}(y)+\epsilon u_{\epsilon}^{q}(y)\big)\big(G(x,y)-G(x,\xi_i)\big)dy:=I_1+I_2+I_3.
\end{split}
\end{equation*}
By Lemma \ref{outside-2}, we obtain
\begin{equation*}
|I_1|\leq C\big(\epsilon^{\frac{(N-2)(p-1)}{N-6+q(N-2)}}+\epsilon^{\frac{(N-2)(q-1)}{N-6+q(N-2)}+1}\big)\int_{\Omega}G(x,y)dy=o(1)
\end{equation*}
as $\epsilon\rightarrow0$. Using Lemma \ref{Lem6.1} by dominated convergence, we get
\begin{equation*}
\begin{split}
I_2&=\sum_{i=1}^mG(x,\xi_i)\int_{ B(0,\lambda_{i\epsilon}^{-1}\epsilon^{-\alpha_0}r/2)}\Big[\lambda_{i\epsilon}^{\frac{N-2}{2}}U_{i\epsilon}^{p}(z)+\epsilon^{\alpha_0(N-(q+1)\tau_0)+1}\lambda_{i\epsilon}^{N-q\tau_0} U_{i\epsilon}^{q}(z)\Big]dz
=\lambda_{i\epsilon}^{\frac{N-2}{2}}\sum_{i=1}^mG(\xi,\xi_i)\mathcal{A}_{0}+o(1)
\end{split}
\end{equation*}
as $\epsilon\rightarrow0$ and where $N>(q+1)\tau_0$.
Here have used estimate the total mass of the perturbation. Since $q\ge1$
and $m$ is fixed, the preceding bound, followed by the change of
variables $y=\xi_{i\epsilon}+\epsilon^{\alpha_0}z$, yields
\begin{align*}
\epsilon\epsilon^{-\alpha_0\tau_0}\int_\Omega u_\epsilon^q\,dy
&\le C\epsilon\epsilon^{\alpha_0(N-\tau_0(q+1))}
\int_{B(0,(1+diam(\Omega))/(\epsilon^{\alpha_0}))}(1+|z|)^{-(N-2)q}\,dz\\
&=C\epsilon^{(N-2)\alpha_0}
\int_{B(0,(1+diam(\Omega))/(\epsilon^{\alpha_0}))}(1+|z|)^{-(N-2)q}\,dz,
\end{align*}
where $D>0$ is fixed and sufficiently large. For $R\ge2$, polar
coordinates give
\[
\int_{B(0,R)}(1+|z|)^{-(N-2)q}\,dz
\le C
\begin{cases}
R^{N-(N-2)q},&q<\dfrac{N}{N-2},\\[3pt]
\log R,&q=\dfrac{N}{N-2},\\[3pt]
1,&q>\dfrac{N}{N-2}.
\end{cases}
\]
Consequently,
\[
\epsilon\epsilon^{-\alpha_0\tau_0}\int_\Omega u_\epsilon^q\,dy
\le C
\begin{cases}
\epsilon^{((N-2)q-2)\alpha_0},&q<\dfrac{N}{N-2},\\[3pt]
\epsilon^{(N-2)\alpha_0}|\log\epsilon^{\alpha_0}|,&q=\dfrac{N}{N-2},\\[3pt]
\epsilon^{(N-2)\alpha_0},&q>\dfrac{N}{N-2},
\end{cases}
\quad\longrightarrow0.
\]
For $I_3$, combining the above estimates, there holds that
\begin{equation*}
\begin{split}
I_3&=\epsilon^{-\alpha_0\tau_0}\sum_{i=1}^m\int_{ B(\xi_{i\epsilon},r/2)}\big[\langle\nabla G(x,\xi_i),y-\xi_i\rangle+O(|y-\xi_i|^2)\big]\big(u_{\epsilon}^{p}(y)+\epsilon u_{\epsilon}^{q}(y)\big)\leq Cr.
\end{split}
\end{equation*}
Combining these estimates, we conclude the validity of \eqref{AS-eigenvalues-1} by $r>0$ is arbitrary and the elliptic regularity.
\end{proof}

\begin{lem}\label{1-AS-eigenvalues-1}
Assume that $\{v_{l\epsilon}\}_\epsilon$ is a family of eigenfunctions of \eqref{ele-2} corresponding to $\lambda_{l,\epsilon}$ and $\|v_{l,\epsilon}\|_{\infty}=1$. Then for any $r>0$ we have
$v_{l\epsilon}\rightarrow0$ as $\epsilon\rightarrow0$ uniformly for $x\in\Sigma_r$.
\end{lem}
\begin{proof}
For $x\in\Sigma_{r}$, using the Green's representation formula we can write
\begin{equation}\label{Vl}
\begin{split}
\frac{v_{l\epsilon}(x)}{\mu_{l\epsilon}}&=
\int_{\Sigma_{r/2}}\big(pu_{\epsilon}^{p-1}(y)+\epsilon q u_{\epsilon}^{q-1}(y)\big)v_{l,\epsilon}(y)G(x,y)dy\\&
\hspace{3mm}+\sum_{i=1}^k\int_{ B(\xi_{i},r/2)}\big(pu_{\epsilon}^{p-1}(y)+\epsilon q u_{\epsilon}^{q-1}(y)\big)v_{l\epsilon}(y)G(x,y)dy.
\end{split}
\end{equation}
Owing to Lemma \ref{outside-2},
\begin{equation}\label{V2}
\Big|\int_{\Sigma_{r/2}}\big(pu_{\epsilon}^{p-1}(y)+\epsilon q u_{\epsilon}^{q-1}(y)\big)v_{l,\epsilon}(y)G(x,y)dy\Big|=O(\epsilon^{\frac{4}{N-6+q(N-2)}}+\epsilon^{\frac{(N-2)(q-1)}{N-6+q(N-2)}+1}).
\end{equation}
Note that, we have used the hypothesis $\|v_{l\epsilon}\|_{\infty}=1$.
Moreover, using \eqref{eq:bound}, we derive
\begin{equation}\label{V3}
\begin{split}
&\Big|\int_{ B(\xi_{i},r/2)}\big(pu_{\epsilon}^{p-1}(y)+\epsilon q u_{\epsilon}^{q-1}(y)\big)v_{l,\epsilon}(y)G(x,y)dy\Big|
\\ \leq &C\mathcal{B}(r)\int_{B(0,r)}\Big[W^{\frac{4}{N-2}}_{0,\lambda_{i\epsilon}\epsilon^{\alpha_0}}
+\epsilon W^{q-1}_{0,\lambda_{i\epsilon}\epsilon^{\alpha_0}}\Big]dy=o(1)
\end{split}
\end{equation}
for any $i\in\{1,\cdots,m\}$ as $\epsilon\rightarrow0$, and
where $\mathcal{B}(r)$ is defined in the proof of Lemma \ref{outside-2}. Thus, combining \eqref{Vl} and \eqref{V2}-\eqref{V3}, we conclude that $V_{l,\epsilon}=o(1)$ as $\epsilon\rightarrow0$ uniformly for $x\in\Sigma_r$.
\end{proof}

Now, we recall $\Omega_{\epsilon}=\epsilon^{-\alpha_0}\Omega,\hspace{2mm}X_{i\epsilon}=\epsilon^{-\alpha_0}\xi_{i\epsilon}$ in \eqref{UU} and define the rescaled solutions $V_{l\epsilon}$ and $d_{\epsilon}$ as follows.
\begin{equation}\label{XX}
V_{l\epsilon}(x)=v_{l\epsilon}(\epsilon^{\alpha_0}x)\hspace{2mm}\mbox{for}\hspace{2mm} x\in\Omega_{\epsilon},\hspace{2mm}d_{\epsilon}(x)=\min\limits_{i}|x-X_{i\epsilon}|,
\end{equation}
Recalling Proposition \ref{existence-0}, it gives a fixed number $m$ of concentration points $\xi_{i\epsilon}\rightarrow\xi_i\in\Omega$ and $\lambda_{i\epsilon}\rightarrow\lambda_i$ as $\epsilon\rightarrow0$. Then there exists $c_0>0$ such that for $\epsilon$ small enough, there holds that
\begin{equation}\label{en}
dist(\xi_{i\epsilon},\partial\Omega)\geq c_0,\quad|\xi_{i\epsilon}-\xi_{j\epsilon}|\geq c_0\hspace{2mm}\mbox{for}\hspace{2mm}i\neq j,\quad c_0\leq\lambda_{i\epsilon}\leq c_0^{-1}.
\end{equation}
We have the following a sharp decay estimate of the eigenfunctions.
\begin{lem}\label{3-AS-eigenvalues-1}
Denote $\langle d_{\epsilon}(x)\rangle=(1+|d_{\epsilon}(x)|^2)^{1/2}$ with $d_{\epsilon}(x)=\min\limits_{i}|x-X_{i\epsilon}|$.
There exist a positive constant $C_l>0$, independent of $\epsilon$, such that
\begin{equation}\label{V4}
|V_{l\epsilon}(x)|\leq C_l \sum_{i=1}^{m}\langle x-X_{i\epsilon}\rangle^{2-N}\leq C_l\langle d_{\epsilon}(x)\rangle^{2-N}\quad\mbox{in}\quad \Omega_{\epsilon}.
\end{equation}
By rescaling, it can be shown that it is equivalent to
\begin{equation}\label{V5}
|v_{l\epsilon}(x)|\leq
C_l\sum_{i}^{m}\left(\frac{\epsilon^{2\alpha_0}}{\epsilon^{2\alpha_0}+|x-\xi_{i\epsilon}|^2}\right)^{(N-2)/2},
\end{equation}
for some $r>0$ and where the constant $C_l$ may depend on the fixed $q$, $l$ and geometry constants.
In particular, for every fixed $r>0$,
\begin{equation}\label{V8}
\sup_{\substack{x\in\Omega\\\min\limits_i|x-\xi_{i\epsilon}|\ge r}}
|v_{\ell\epsilon}(x)|\le C_{l,r}\epsilon^{\alpha_0(N-2)}.
\end{equation}
\end{lem}
\begin{proof}
Its proof is deferred to Section \ref{sharp}.
\end{proof}

As a consequence of the previous Lemma, we conclude this section with the following useful result.
\begin{cor}\label{AS-eigenvalues-2}
It holds that $v_{l\epsilon}=O(\epsilon^{(N-2)\alpha_0})$ in $\Sigma_r$ for some $r>0$ small.
\end{cor}

\section{The proof of Theorem \ref{Figalli-1}}\label{asymptotic}
In this section, we are devoted to proving Theorem \ref{Figalli-1}.
The following result will be used to complete the proof of Theorem \ref{Figalli-1}.
\begin{lem}\label{cobvergence-1}
Assume that $\{v_{l\epsilon}\}_\epsilon$ is a family of eigenfunctions of \eqref{ele-2} corresponding to $\mu_{l\epsilon}$ and $\|v_{l\epsilon}\|_{\infty}=1$. Then there exists at least one $i_0\in\{1,\cdots,k\}$ and fixed $l$, such that
$V_{li_0\epsilon}\rightharpoonup V_0$ weakly in $\mathcal{D}^{1,2}(\mathbb{R}^N)$ as $\epsilon\rightarrow0$, where the
nonzero function
$V_0$ solves
\begin{equation}\label{WP}
-\Delta V_i=\mu_l pW^{p-1}V_i\hspace{2mm}\mbox{in}\hspace{2mm}\mathbb{R}^N.
\end{equation}
\end{lem}
\begin{proof}
The derivation of the conclusion can be achieved in Lemmas \ref{1-AS-eigenvalues-1}-\eqref{3-AS-eigenvalues-1} and elliptic regularity.
\end{proof}

The following aims to establish the estimates of the eigenvalues $\mu_{l\epsilon}$.	
For $i=1,\dots,k$ and let us fix a small number $r>0$ such that for any $1\leq i\neq j\leq k$ and any $\epsilon>0$ small, there holds
$$
B(x_i,4r)\subset\Omega\quad\mbox{and}\quad B(x_i,4r)\cap B(x_j,4r)=\emptyset.
$$
We first define
\begin{equation}\label{pusai}
\psi_i(x)=\tilde{\psi}(x-x_{i}),
\end{equation}
where $\tilde{\psi}$ is a function in $C_{0}^{\infty}(B(0,3r))$ such that $\tilde{\psi}\equiv1$ in $B(0,2r)$, $0\leq\tilde{\psi}\leq1$ in $B(0,3r)$, and
\begin{equation}\label{varepsilon}
\bar{u}_{i,\epsilon}=\psi_iu_{\epsilon}\hspace{2mm}\mbox{and}\hspace{2mm}
\eta_{m,i,\epsilon}(x)=\psi_i(x)\frac{\partial u_{\epsilon}}{\partial x_m}(x),\quad m=1,\cdots,N,
\end{equation}
\begin{equation}\label{varep}
\eta_{N+1,i,\epsilon}(x)=\psi_i(x)\left((x-\xi_{i\epsilon})\cdot\nabla u_{\epsilon}+\frac{2}{p-1}u_{\epsilon}\right):=\psi_i(x)\omega_{i\epsilon}(x).
\end{equation}
We need some preliminary computations.
\begin{lem}\label{nabla}
It holds that
\begin{equation}\label{womiga}
-\Delta \omega_{i\epsilon}=\big(pu_{\epsilon}^{p-1}+\epsilon qu_{\epsilon}^{q-1}\big)\omega_{i\epsilon}+\frac{2(p-q)}{p-1}\epsilon u_{\epsilon}^q.
\end{equation}
\end{lem}
\begin{proof}
See Lemma 5.1 in \cite{YZ}.
\end{proof}	
We have the following the bilinear identities for $u_{\epsilon}$ and $v_{l\epsilon}$.
\begin{lem}\label{conseq}
It holds that
\begin{equation}\label{m0}
\int_{\partial B(\xi_{i\epsilon},r)}\Big(\frac{\partial u_{\epsilon}}{\partial\nu}v_{l\epsilon}-\frac{\partial v_{l\epsilon}}{\partial\nu}u_{\epsilon}\Big)dS_x=\int_{B(\xi_{i\epsilon},r)}\Big(\big(p\mu_{l\epsilon}-1\big)u_{\epsilon}^{p}+\epsilon\big(q\mu_{l\epsilon}-1\big) u_{\epsilon}^{q}\Big)v_{l\epsilon}dx
\end{equation}
and
\begin{equation}\label{489}
\begin{split}
			\int_{\partial B(\xi_{i\epsilon},r)}\Big(\frac{\partial \omega_{i\epsilon}}{\partial\nu}v_{l\epsilon}-\frac{\partial v_{l\epsilon}}{\partial\nu}\omega_{i\epsilon}\Big)dS_x&=\big(\mu_{l\epsilon}-1\big)\int_{B(\xi_{i\epsilon},r)}\big(pu_{\epsilon}^{p-1}+\epsilon qu_{\epsilon}^{q-1}\big)\omega_{i\epsilon}v_{l\epsilon}dx\\&\hspace{3mm}-\frac{2(p-q)}{p-1}\epsilon \int_{B(\xi_{i\epsilon},r)} u_{\epsilon}^qv_{l\epsilon}dx
		\end{split}
\end{equation}
for any $i=1,\cdots,k$ and any $r>0$ small enough,
where $\nu=\nu(x)$ denotes the unit outward normal to the boundary $\partial B(\xi_{i\epsilon},r)$.
\end{lem}
\begin{proof}
The conclusion can be derived by \eqref{womiga} and simple computations by integrating by parts, and so we omit it.	
	\end{proof}
\begin{lem}\label{THERE}
There exists a nonvector
$$\tilde{\bm{e}}_{l}=\big(\lambda_{1}^{\frac{N-2}{2}}e_{l1},\cdots,\lambda_{m}^{\frac{N-2}{2}}e_{lm}\big)^T\in\mathbb{R}^m$$
such that for each $i=1,\cdots,m$ the function $V_{li\epsilon}\rightharpoonup e_{li}W$ weakly in $\mathcal{D}^{1,2}(\mathbb{R}^N)$ as $\epsilon\rightarrow0$.
\end{lem}
\begin{proof}
The conclusion easily follows from Lemma \ref{Lem6.1} and Lemma \ref{cobvergence-1}.
\end{proof}
Lemma \ref{THERE} and \eqref{V2} then imply the following result.
\begin{lem}\label{gvalues}
Assume that $p=(N+2)/(N-2)$, $1<q<3$ if $N=4$, and $1\leq q<p$ if $N\geq5$. Then
\begin{equation}\label{nvalu-1}
\frac{v_{l\epsilon}(x)}{p\mu_{l\epsilon}\epsilon^{(N-2)\alpha_0}}=\mathcal{A}_{0}\sum_{i=1}^m\lambda_{i\epsilon}^{N-2}G(x,\xi_i)e_{li}+o(1)\hspace{2mm}\mbox{in}\hspace{2mm}C^{1}(\Omega\setminus\{\xi_1,\cdots,\xi_k\}),
\end{equation}
where $\mathcal{A}_0$ is defined in \eqref{pusaik-1-1}.
\end{lem}
\begin{proof}
Recall that \eqref{Vl}, for $r>0$ fixed small and $x\in \Sigma_{r}$, we obtain
\begin{equation*}
\begin{split}
\frac{v_{l\epsilon}(x)}{\mu_{l\epsilon}\epsilon^{(N-2)\alpha_0}}&=\epsilon^{-(N-2)\alpha_0}
\Big(\sum_{i=1}^k\int_{ B(\xi_{i\epsilon},r/2)}+\int_{\Sigma_{r/2}}\Big)\big(pu_{\epsilon}^{p-1}(y)+\epsilon q u_{\epsilon}^{q-1}(y)\big)v_{l\epsilon}(y)G(x,y)dy.
\end{split}
\end{equation*}
In order to conclude the proof of \eqref{nvalu-1}, we establish now the
 following key estimates:
 \begin{equation*}
\begin{split}
\epsilon^{-(N-2)\alpha_0}&\int_{ B(\xi_{i\epsilon},2r)}pu_{\epsilon}^{p-1}(y)v_{l\epsilon}(y)G(x,y)dy
=p\lambda_{i\epsilon}^{N-2}e_{li}\int_{\mathbb{R}^N}W^{p}dyG(x,\xi_i)+o(1),
\end{split}
\end{equation*}
as $\epsilon\rightarrow0$. Moreover, we have
 \begin{equation*}
\begin{split}
&\Big|\epsilon^{-(N-2)\alpha_0}\epsilon\int_{ B(\xi_{i\epsilon},r/2)}qu_{\epsilon}^{q-1}(y)v_{l\epsilon}(y)G(x,y)dy\Big|
 =o(1),
\end{split}
\end{equation*}
by \eqref{eq:bound}, where $2-\frac{(N-2)(q-1)}{2}>0$ and $\mathcal{B}(r)$ is defined as before. On the other hand, $x\in \Sigma_{r}$,
by Lemma \ref{outside-2} and Corollary \ref{AS-eigenvalues-2}, we obtain
\begin{equation*}
\begin{split}
\Big|\epsilon^{-(N-2)\alpha_0}&\int_{\Sigma_{r/2}}\big(pu_{\epsilon}^{p-1}(y)+\epsilon q u_{\epsilon}^{q-1}(y)\big)v_{l\epsilon}(y)G(x,y)dy\Big|\\&
\leq C\big(\epsilon^{\frac{4}{N-6+q(N-2)}}+\epsilon^{\frac{(N-2)(q-1)}{N-6+q(N-2)}+1}\big)=o(1),
\end{split}
\end{equation*}
where $\mathcal{C}_{\epsilon}=\max\{v_{l\epsilon}(x):~x\in\Sigma_{r}\}$. Combining these estimates, we conclude that \eqref{nvalu-1}.
\end{proof}

Now we are ready to prove Theorem \ref{Figalli-1}.	
\begin{proof}[\textbf{Proof of Theorem \ref{Figalli-1}}]
Let us denote by $\mathcal{Z}$ the space spanned by the functions $\bar{u}_{i,\epsilon}$ for $i=1,\cdots,m$.
Then by the Courant-Fischer formula of the eigenvalue $\mu_{i\epsilon}$, we have
		\begin{equation}\label{minmax}
\begin{split}
			\mu_{m\epsilon}&=\min\limits_{\substack{\mathcal{Z}\subset H_{0}^1(\Omega),\\ dim\mathcal{Z}=k}}\max\limits_{f\in\mathcal{Z}\setminus\{0\}}\frac{\displaystyle\int_{\Omega}|\nabla f(x)|^2dx}{\displaystyle\int_{\Omega}\big(pu_{\epsilon}^{p-1}+\epsilon qu_{\epsilon}^{q-1}\big)f^2dx}			\leq\max\limits_{f\in\mathcal{Z}\setminus\{0\}}\frac{\displaystyle\int_{\Omega}|\nabla f(x)|^2dx}{\displaystyle\int_{\Omega}\big(pu_{\epsilon}^{p-1}+\epsilon qu_{\epsilon}^{q-1}\big)f^2dx}:=\max\limits_{f\in\mathcal{Z}\setminus\{0\}}g.
\end{split}
		\end{equation}
We next write $f=\sum_{i=1}^{m}a_i\bar{u}_{i,\epsilon}$ for some $(a_1,\cdots,a_m)\neq0$, then by Lemma \ref{Lem6.1} we get
\begin{equation*}
\begin{split}
			g&
\leq\frac{\displaystyle\int_{B(0,\lambda_{i\epsilon}^{-1}\epsilon^{-\alpha_0})}U_{i\epsilon}^{p+1}dx+ (\frac{\epsilon^{\alpha_0}}{\lambda_{i\epsilon}^{-1}})^{N-(q+1)\tau_0}\epsilon \displaystyle\int_{\Omega\setminus \cup_{i=1}^mB(0,\lambda_{i\epsilon}^{-1}\epsilon^{-\alpha_0})}U_{i\epsilon}^{q+1}dx}{p\displaystyle\int_{B(0,\lambda_{i\epsilon}^{-1}\epsilon^{-\alpha_0})}U_{i\epsilon}^{p+1}dx
+q(\frac{\epsilon^{\alpha_0}}{\lambda_{i\epsilon}^{-1}})^{N-(q+1)\tau_0}\epsilon \displaystyle\int_{\Omega\setminus \cup_{i=1}^mB(0,\lambda_{i\epsilon}^{-1}\epsilon^{-\alpha_0})}U_{i\epsilon}^{q+1}dx}
\rightarrow\frac{1}{p}
\end{split}
		\end{equation*}
as $\epsilon\rightarrow0$, which means that
$\limsup\limits_{\epsilon\rightarrow0}\mu_{m\epsilon}\leq1/p.$
Then, for any $l=1,\cdots,m$, up to a subsequence,  we have $\mu_{l\epsilon}\rightarrow\mu_l\in[0,1/p]$ for $\epsilon$ sufficiently small.
Moreover by Lemma \ref{cobvergence-1}, up to a subsequence, there exists a function $V_0$, such that we have $V_{li\epsilon}\rightarrow V_0$ in $C_{loc}^{1}(\mathbb{R}^N)$, where $V_0$ satisfies
\begin{equation*}
-\Delta V_{0}=\mu_lpW^{p-1}V_{0} \quad \mbox{in}\quad \mathbb{R}^N.
\end{equation*}
It follows from the non-degeneracy of solutions in \cite{BE91} that
\begin{equation}\label{e-00-1}
\mu_l=1/p\quad\mbox{and}\quad v_0=W(x)=(1+|x|^2)^{-\frac{N-2}{2}}.
\end{equation}

Now we will show that \eqref{vie-1}. Recall that $\mathcal{A}_{0}$ is defined as before in \eqref{pusaik-1-1} and Green's function, we denote in what follows
\begin{equation}\label{FK1}	
\mathcal{A}_{0}^{-1}\mathcal{F}_i(x)=-\lambda_{i}^{\frac{N-2}{2}}H(x,\xi_i)+\sum_{j\neq i}\lambda_{j}^{\frac{N-2}{2}}G(x,\xi_j),
\end{equation}
\begin{equation}\label{FK2}		
\mathcal{A}_{0}^{-1}\mathcal{K}_i(x)=-\lambda_{i}^{N-2}e_{li}H(x,\xi_i)+\sum_{j\neq i}\lambda_{j}^{N-2}e_{l,j}G(x,\xi_j),
\end{equation}
which are harmonic near $x_i$. Then, by \eqref{AS-eigenvalues-1} and \eqref{nvalu-1}, a direct computation shows that
\begin{equation}\label{F010}
\frac{u_{\epsilon}(x)}{\epsilon^{(N-2)\alpha_0/2}}=\mathcal{A}_{0}\lambda_{i}^{\frac{N-2}{2}}\frac{1}{(N-2)\omega_N|x-\xi_{i\epsilon}|^{N-2}}+\mathcal{F}_{i}(x)+o(1)
\end{equation}
and
\begin{equation}\label{F012}
\frac{v_{l\epsilon}(x)}{p\lambda_{l\epsilon}\epsilon^{(N-2)\alpha_0}}=\mathcal{A}_{0}\lambda_{i}^{N-2}e_{li}\frac{1}{(N-2)\omega_N|x-\xi_{i\epsilon}|^{N-2}}+\mathcal{K}_{i}(x)+o(1).
\end{equation}
Hence, substituting the two identities \eqref{F010} and \eqref{F012} in the left-hand side of \eqref{m0} and using $\mu_{l\epsilon}\rightarrow 1/p$, we get
\begin{equation}\label{F013}
\begin{split}
&\frac{1}{\epsilon^{[(N-2)/2+N-2]\alpha_0}}\int_{\partial B(\xi_{i\epsilon},r)}\Big(\frac{\partial u_{\epsilon}}{\partial\nu}v_{l\epsilon}-\frac{\partial v_{l\epsilon}}{\partial\nu}u_{\epsilon}\Big)dS_x\\&=\frac{\mathcal{A}_{0}}{\omega_N}\int_{\partial B(\xi_{i\epsilon},r)}\frac{1}{|x-\xi_{i\epsilon}|^{N-1}}\Big[\lambda_{i}^{\frac{N-2}{2}}\mathcal{F}_{i}(x)-\lambda_{i}^{N-2}e_{li}\mathcal{K}_{i}(x)\Big]dS_x+o(1).
\end{split}
\end{equation}
Noticing that $\mathcal{F}_i$ and $\mathcal{K}_i$ satisfy
\begin{equation}\label{F014}
\mathcal{F}_i(\xi_i)=\frac{1}{\omega_N}\int_{\partial B(\xi_{i\epsilon},r)}\frac{\mathcal{F}_i(x)}{|x-\xi_{i\epsilon}|^{N-1}}dS_{x}\hspace{2mm}\mbox{and}\hspace{2mm}\mathcal{K}_i(\xi_i)=\frac{1}{\omega_N}\int_{\partial B(\xi_{i\epsilon},r)}\frac{\mathcal{K}_i(x)}{|x-\xi_{i\epsilon}|^{N-1}}dS_{x}.
\end{equation}
Moreover, since $(\xi_1,\cdots,\xi_m, \lambda_1,\cdots,\lambda_m)$ is a critical point of the functional $\Psi_{\Omega,q}(\xi,\lambda)$,
then we have
\begin{equation}\label{critical}
\begin{split}
\mathcal{A}_0^2\Big[\phi(\xi_i)\lambda_{i}^{N-2}-\sum^k _{j\neq i} G(\xi_i,\xi_j)(\lambda_{i}\lambda_{j})^{\frac{N-2}{2}}\Big]=\mathcal{B}_0(p-q)\lambda_{i}^{\frac{(N-2)(p-q)}{2}}.
\end{split}
\end{equation}
Therefore,
\begin{equation}\label{F015}
\mathcal{A}_0^2\sum^m _{\substack{i,j=1,\\
i\neq j}} G(\xi_i,\xi_j)\lambda_{i}^{N-2}\lambda_{j}^{\frac{N-2}{2}}=\mathcal{A}_0^2\sum_{i=1}^m\phi(\xi_i)\lambda_{i}^{\frac{(3N-6)}{2}}-(p-q)\mathcal{B}_0\sum_{i=1}^m\lambda_{i}^{\frac{(N-2)(p-q+1)}{2}}.
\end{equation}
Combining this identity, \eqref{FK1}-\eqref{FK2} and \eqref{F014}, we can write \eqref{F013} as
\begin{equation}\label{f016}
\begin{split}
&\frac{1}{\epsilon^{[(N-2)/2+N-2]\alpha_0}}\int_{\partial B(\xi_{i\epsilon},r)}\Big(\frac{\partial u_{\epsilon}}{\partial\nu}v_{l\epsilon}-\frac{\partial v_{l\epsilon}}{\partial\nu}u_{\epsilon}\Big)dS_x\\&=
\sum_{i=1}^m\Big(\mathcal{A}_0^2\phi(x_i)\lambda_{i}^{N-2}-\frac{(p-q)}{2}B_0\lambda_{i}^{\frac{(N-2)(p-q)}{2}}\Big)\Big(\lambda_{i}^{\frac{N-2}{2}}e_{li}\Big)\\&\hspace{3mm}-\mathcal{A}_0^2\sum^m _{\substack{i,j=1,\\
i\neq j}} G(\xi_i,\xi_j)\Big(\lambda_{i}\cdot\lambda_{j}\Big)^{\frac{N-2}{2}}\Big(\lambda_{j}^{\frac{N-2}{2}}e_{lj}\Big)+o(1)
\end{split}
\end{equation}
as $\epsilon\rightarrow0$. On the other hand, the right-hand side of \eqref{m0} tells us that
\begin{equation}\label{f017}
\begin{split} &\frac{1}{\epsilon^{(N-2)\alpha_0/2}}\int_{B(\xi_{i\epsilon},r)}\Big(\big(p\mu_{l\epsilon}-1\big)u_{\epsilon}^{p}+\epsilon\big(q\mu_{l\epsilon}-1\big) u_{\epsilon}^{q}\Big)v_{l\epsilon}dx
\\&=\lambda_{i}^{\frac{N-2}{2}}\big(p\mu_{l\epsilon}-1\big)\int_{B(0,\lambda_{i\epsilon}^{-1}\epsilon^{-\alpha_0}r)}U_{i\epsilon}^{p}V_{li\epsilon}dx\\&
\hspace{3mm}+(\epsilon^{\alpha_0})^{N-\frac{(N-2)(q+1)}{2}}\mu_{i}^{N-\frac{(N-2)q}{2}}\big(q\mu_{l\epsilon}-1\big)\epsilon\int_{B(0,\lambda_{i\epsilon}^{-1}\epsilon^{-\alpha_0}r)}U_{i\epsilon}^{q}V_{li\epsilon}dx.
\end{split}
\end{equation}
By Lemma \ref{Lem6.1}, Lemma \ref{THERE} and the dominated convergence theorem,
$$
\lim\limits_{\epsilon\rightarrow0}\int_{B^N(0,\lambda_{i\epsilon}^{-1}\epsilon^{-\alpha_0}r)}U_{i\epsilon}^{p}V_{li\epsilon}dx=\int_{\mathbb{R}^N}W^{p+1}(x)dxe_{li},\quad\lim\limits_{\epsilon\rightarrow0}\int_{B(0,\lambda_{i\epsilon}^{-1}\epsilon^{-\alpha_0}r)}U_{i\epsilon}^{q}V_{li\epsilon}dx
=\int_{\mathbb{R}^N}W^{q+1}(y)dye_{li}.$$
Then we can deduce that the convergence of the second term of the right-hand side of \eqref{f017} to $0$ as $\epsilon\rightarrow0$ since $\frac{2}{N-2}<q<p$ for $N\geq4$.
Therefore, combining this information with \eqref{f016}-\eqref{f017}, we deduce that
\begin{equation}\label{f018}
\begin{split}
&\Big(\phi(x_i)\lambda_{i}^{N-2}-\frac{(p-q)}{2}\frac{\mathcal{B}_0}{\mathcal{A}_0^2}\lambda_{i}^{\frac{(N-2)(p-q)}{2}}\Big)\Big(\lambda_{i}^{\frac{N-2}{2}}e_{li}\Big)\\&-\sum^m _{\substack{i,j=1,\\
i\neq j}} G(\xi_i,\xi_j)\Big(\lambda_{i}\cdot\lambda_{j}\Big)^{\frac{N-2}{2}}\Big(\lambda_{j}^{\frac{N-2}{2}}e_{lj}\Big)=\frac{1}{\mathcal{A}_0^2}\int_{\mathbb{R}^N}W^{p+1}(x)dx\lim\limits_{\epsilon\rightarrow0}\Big(\frac{p\mu_{l\epsilon}-1}{\epsilon^{(N-2)\alpha_{0}}}\Big)\Big(\lambda_i^{\frac{N-2}{2}}e_{li}\Big).
\end{split}
\end{equation}
Observe that we can write:
$$
\nu_{l}^{1}:=\lim\limits_{\epsilon\rightarrow0}\frac{1}{\mathcal{A}_0^2}\int_{\mathbb{R}^N}W^{p+1}(x)dx\left(\frac{p\mu_{l\epsilon}-1}{\epsilon^{(N-2)\alpha_{0}}}\right).
$$
By the definitions of symmetric matrix $A$ in \eqref{pusaik-2} and a nonzero column vector $\tilde{e}_{l}$ in \eqref{eltita}, we infer that \eqref{f018} is equivalent to
$A\tilde{\bm{e}}_{l}=\nu_{l}^{1}\tilde{\bm{e}}_{l}$.
Furthermore, we can compute
$$\mu_{l\epsilon}=\frac{N-2}{N+2}+\frac{\mathcal{A}_0^2(N-2)}{(N+2)\int_{\mathbb{R}^N}W^{p+1}(x)dx}\nu_{l}^{1}\cdot\epsilon^{(N-2)\alpha_0}+o(\epsilon^{(N-2)\alpha_0}).$$
This proves \eqref{vie-1} with $\alpha_N=\mathcal{A}_0^2(N-2)\nu_{l}^{1}/((N+2)\int_{\mathbb{R}^N}W^{p+1}_{0,1}(x)dx)$.

Next, if the eigenfunctions $v_{l_1\epsilon}$, $v_{l_2\epsilon}$, $l_1\neq l_2$ are two eigenfunctions of \eqref{ele-2} corresponding to $\lambda_{l_1\epsilon}$ and $\lambda_{l_2\epsilon}$, orthogonal in the space $H_{0}^1(\Omega)$, then we have $\int_{\Omega}\nabla v_{l_1\epsilon}\cdot\nabla v_{l_2\epsilon}dx=0$. By Lemma \ref{Lem6.1} and Lemma \ref{THERE} imply
\begin{equation*}
\begin{split}
0&=\lim\limits_{\epsilon\rightarrow0}\frac{1}{\epsilon^{(N-2)\alpha_0}}\left(\sum_{i=1}^{m}\int_{B(\xi_{i\epsilon},r)}+\int_{\Omega\setminus\cup_{i=1}^{m}B(\xi_{i\epsilon},r)}\right)\Big(p\lambda_{l_1\epsilon}u_{\epsilon}^{p-1}+\epsilon q\lambda_{l_1,\epsilon}u_{\epsilon}^{q-1}\Big)v_{l_1\epsilon}v_{l_2\epsilon}dx\\&
=\sum_{i=1}^{m}\Big(\lambda_{i}^{\frac{N-2}{2}}e_{l_1\epsilon}\Big)\Big(\lambda_{i}^{\frac{N-2}{2}}e_{l_2,\epsilon}\Big)\int_{\mathbb{R}^N}W^{p+1}(x)dx.
\end{split}
\end{equation*}
Thus $e^T_{l_1\epsilon}\cdot e^T_{l_2\epsilon}=0$ is established. Finally, combined with the result of Lemma \ref{THERE} and so we conclude our proof.
\end{proof}

\section{Asyptotic behaviour of the eigenfunctions, $m+1\leq l\leq(N+1)m$}\label{section4}
In this section, we will prove Theorem \ref{existence-1}. To conclude this,  according to Lemma \ref{identity}, it suffices to prove the following main result.

\begin{thm}\label{dilation}
Assume that
\[
 N=4,\quad 1<q<3,
 \qquad\text{or}\qquad
 N\geq5,\quad 1\leq q<p,
\]
and that the positive multi-bubble hypotheses stated above hold.
For a fixed index \(\ell\), let \(v_{\ell\epsilon}\) satisfy
\[
 \|v_{l\epsilon}\|_{L^\infty(\Omega)}=1,
 \qquad
 m+1\leq l\leq(N+1)m.
\]
Then every local cluster limit in \eqref{identity} has vanishing
dilation coefficients:
\[
 a_{1,1,N+1}=\cdots=a_{l,m,N+1}=0.
\]
Equivalently, every such limit is a linear combination of the
translation Jacobi fields.
\end{thm}

Before we prove Theorem \ref{dilation}, let us introduce some preliminary results.
\subsection{Estimates for the eigenvalues $\mu_{l\epsilon}$ for $m+1\leq l\leq(N+1)m$}\label{section3-1}
The following key estimate estimate of the eigenvalue $\mu_{l\epsilon}$  will be used a couple of times in the paper to conclude the proof of main
results.
\begin{lem}\label{baowenbei}
		For $m+1\leq l\leq(N+1)m$ and $\alpha_0=2/(N-6+q(N-2))$, we have
		\begin{equation}\label{baowen1}
\mu_{l\epsilon}\leq1+O(\epsilon^{N\alpha_0}).
		\end{equation}
Moreover, $\mu_{l\epsilon}\rightarrow1$ as $\epsilon\rightarrow0$ for $m+1\leq l\leq(N+1)m$.
\end{lem}
	\begin{proof}
The proof of this Lemma is very similarly to \cite{GP05,YZ} by using \eqref{baowen1}, Lemma \ref{cobvergence-1}, Lemma \ref{THERE} and proof of Theorem \ref{Figalli-1}. Therefore, we will omit it.
\end{proof}

\subsection{Asymptotic behaviour of the eigenfunctions}
Before going into the details of the proof of  Theorem \ref{dilation} it is useful to introduce some notation.
Set
\begin{equation}\label{qitai}
\bm{c}_l^i=\lambda_{i}^{\frac{N-2}{2}}a_{l,i,N+1}=t_ia_{l,i,N+1},\quad\bm{c}_l=\Big(\lambda_{1}^{\frac{N-2}{2}}a_{l,1,N+1},\cdots,\lambda_{m}^{\frac{N-2}{2}}a_{l,m,N+1}\Big)^T\in\mathbb{R}^m,
\end{equation}
and recalling the definition of matrixs $D_t$ and $D_{q}$ in \eqref{DDT}.
\begin{lem}\label{identity}
For the rescaled eigenfunctions $V_{l,\epsilon}$ defined in \eqref{vie} and $k+1\leq l\leq(N+1)k$, we have
\begin{equation}\label{viep}
V_{li\epsilon}(x)\rightarrow-\sum_{k=1}^{N}a_{l,i,k}\partial_kW+a_{l,i,N+1}Z_0\quad\mbox{in}\quad C_{loc}^{1}(\mathbb{R}^N).
\end{equation}
In addition, there is at least one $i_1\in{1,\cdots,m}$ such that $(a_{l,i_1,1},\cdots,a_{l,i_1,N+1})\neq0$.
\end{lem}
\begin{proof}
We here proceed similarly to the proof of Lemma 4.1 in \cite{YZ} and the result can be achieved by the non-degeneracy of solutions in \cite{BE91}, Lemma \ref{cobvergence-1} and Lemma \ref{baowenbei}. Hence, we omit the details.
	\end{proof}

We have the following lemma.
\begin{lem}\label{lem:exterior-limit}
For any $k+1\leq l\leq(N+1)k$ then there exists some $i_1\in\{1,\cdots,k\}$ such that
\begin{equation}\label{number}
\frac{v_{l\epsilon}(x)}{\epsilon^{(N-2)\alpha_0}}\rightarrow \mathcal{C}_0p\sum_{i=1}^{m}a_{l,i,N+1}\lambda_{i}^{N-2}G(x,\xi_i)\quad\mbox{in}\quad C^{1,\alpha}(\Omega\setminus\{\xi_1,\cdots,\xi_m\}),
\end{equation}
where $\mathcal{C}_0=\int_{\mathbb{R}^N}W^{p-1}(y)Z_0dy=\frac{N-2}{2}\mathcal{A}_0$.
\end{lem}
\begin{proof}
A similar argument as in the proof of Lemma \ref{gvalues}, we can easily deduce estimate \eqref{number} by using Lemma \ref{outside-2}, Lemma \ref{baowenbei} with Lemma \ref{identity} and combined with the oddness of
the integrand. So we omit it here.
 \end{proof}

\begin{lem}\label{lem:green-flux}
For $1\leq i\leq k$, and introducing
\begin{equation}\label{Q1}
\begin{split}
Q_{i,j}(x)&=(x-\xi_i)\cdot\nabla G(x,\xi_j)+\frac{2}{p-1}G(x,\xi_j),\hspace{3mm} H_{i,k}=G(x,\xi_{k})
\end{split}
\end{equation}
for any $1\leq j,m\leq k$.
Then we have
\begin{equation}\label{Q2}
F^{(r)}_{j,k,i}:=\int_{\partial B(\xi_i,r)}\big(\partial_{\nu}Q_{i,j}H_{i,k}-Q_{i,j}\partial_{\nu}H_{i,k}dS\big)=\frac{2}{p-1}
\left\lbrace
\begin{aligned}
&0\quad \quad\quad\quad\quad\quad\hspace{3.5mm}\mbox{if}\quad j\neq i\quad\mbox{and}\quad k\neq i \\
&G(\xi_{i},\xi_{j})\quad \quad\quad\hspace{2.5mm}\mbox{if}\quad j\neq i\quad\mbox{and}\quad k=i,
 \\
&G(\xi_{i},\xi_{k})\quad \quad\hspace{1.5mm}\quad\mbox{if}\quad j=i\quad\mbox{and}\quad k\neq i,
 \\
&-2\phi(\xi_i)\quad \quad \quad\hspace{1.7mm}\mbox{if}\quad j=k=i.
 \end{aligned}
\right.
\end{equation}
\end{lem}
\begin{proof}
See Lemma 4.5 in \cite{CKL-1}.
\end{proof}

Then we have
\begin{lem}\label{m00}
We have
\begin{equation*}
\begin{split}
\frac{1}{\epsilon^{[(N-2)/2+N-2]\alpha_0}}\int_{\partial B(\xi_{i\epsilon},r)}\Big(\frac{\partial \omega_{i\epsilon}}{\partial\nu}v_{l\epsilon}-\frac{\partial v_{l\epsilon}}{\partial\nu}\omega_{i\epsilon}\Big)dS_x\rightarrow
-\mathcal{A}_0\mathcal{C}_0\frac{2p}{p-1}\Big[D_tMD_t
    +\frac{\mathcal{B}_0(p-q)}{\mathcal{A}_{0}^2}
    D_{q}\Big]\bm{c}_{l}
\end{split}
\end{equation*}
as $\epsilon\rightarrow0$.
\end{lem}
\begin{proof}
By the definition of $\omega_{i,\epsilon}$ in \eqref{varep} and \eqref{AS-eigenvalues-1}, we have
\begin{equation}\label{thm1-5}
\frac{\omega_{i\epsilon}(x)}{\epsilon^{(N-2)\alpha_0/2}}\rightarrow\mathcal{A}_{0}\sum_{j=1}^m\lambda_{j,\epsilon}^{\frac{N-2}{2}}\Big((x-\xi_i)\nabla G(x,\xi_j)+\frac{2}{p-1}G(x,\xi_j)\Big)
\end{equation}
in $C^{1}(\Omega\setminus\{x_1,\cdots,x_k\})$ as $\epsilon\rightarrow0$.
Hence combined with \eqref{number}, the left hand side of \eqref{489} becomes
\begin{equation*}
\begin{split}
\frac{1}{\epsilon^{[(N-2)/2+N-2]\alpha_0}}\int_{\partial B(\xi_{i\epsilon},r)}\Big(\frac{\partial \omega_{i\epsilon}}{\partial\nu}v_{l,\epsilon}-\frac{\partial v_{l\epsilon}}{\partial\nu}\omega_{i\epsilon}\Big)dS_x=p\mathcal{A}_0\mathcal{C}_0
\sum_{j,k=1}^m\lambda_{j\epsilon}^{\frac{N-2}{2}}\lambda_{m,\epsilon}^{N-2}a_{l,k,N+1}F^{(r)}_{j,k,i},
\end{split}
\end{equation*}
where $F_{j,m,i}^{(r)}$ can be found in \eqref{Q1}. We further compute the right hand side of the above integrand, by \eqref{Q2}.
Let
\[
    \Phi_i(x):=\frac{\gamma_N }{|x-x_i|^{N-2}},
    \qquad
    \gamma_N:=\frac{1}{(N-2)\omega_N}.
\]
Thus, \(\Phi_i\) is the fundamental solution of \(-\Delta\) with pole at
\(x_i\). Throughout this computation, \(\nu\) denotes the unit normal
pointing outward from the ball \(B(\xi_i,r)\). Consequently,
\[
\begin{aligned}
    \int_{\partial B(\xi_i,r)}\partial_\nu\Phi_i\,dS
    &=
    \gamma_N(2-N)r^{1-N}
    \omega_{N}r^{N-1}
    =-1.
\end{aligned}
\]

For notational convenience, write
\[
    G_j(x):=G(x,\xi_j),
    \qquad
    G_{ij}:=G(\xi_i,\xi_j),
    \qquad
   \phi_i:=\phi(\xi_i).
\]
We first observe that $F_{j,m,i}^{(r)}$ is independent of \(r\),
provided that the annular region under consideration contains no
concentration point other than \(x_i\). Indeed, both \(Q_{i,j}\) and
\(G_m\) are harmonic away from their poles. More precisely,
\[
    \Delta Q_{i,j}
    =
    (x-\xi_i)\cdot\nabla_x(\Delta G_j)
    +\frac{N+2}{2}\Delta G_j
    =0
\]
in every such annulus. Green's second identity therefore gives
\[
    F_{j,k,i}^{(r)}(r_2)
    -
    F_{j,k,i}^{(r)}(r_1)
    =0
\]
whenever
\[
    0<r_1<r_2<
    \frac12\min_{m\ne i}|x_m-\xi_i|.
\]
It is consequently sufficient to compute the limit as \(r\downarrow0\).

Near \(\xi_i\), the Green function and the corresponding scaling field
satisfy
\[
    G_i(x)
    =
    \Phi_i(x)-\phi_i+O(r),
\]
and
\[
\begin{aligned}
    Q_{i,i}(x)
    &=
    (x-\xi_i)\cdot\nabla\Phi_i(x)
    +\frac{N-2}{2}\Phi_i(x)-\frac{N-2}{2}\phi_i+O(r) \\
    &=
    -\frac{N-2}{2}\Phi_i(x)-\frac{N-2}{2}\phi_i+O(r),
\end{aligned}
\]
because
\[
    (x-\xi_i)\cdot\nabla\Phi_i(x)
    =
    -(N-2)\Phi_i(x)
    =
    -2\frac{2}{p-1}\Phi_i(x).
\]
If \(j\ne i\), then \(G_j\) is smooth near \(x_i\), and hence
\[
    G_j(x)=G_{ij}+O(r),
    \qquad
    Q_{ij}(x)=\frac{N-2}{2}G_{ij}+O(r).
\]
Combining the above information, we deduce that
\begin{equation*}
\begin{split}
\sum_{j,k=1}^m\lambda_{j,\epsilon}^{\frac{N-2}{2}}\lambda_{k,\epsilon}^{N-2}&a_{l,k,N+1}F^{(r)}_{j,k,i}=
\frac{2}{p-1}t_i
    \left[
        \sum_{j=1,j\ne i}^{m}
        G_{ij}t_j\bm{c}_l^j
        +
        \bm{c}_l^i
        \sum_{j=1,j\ne i}^{m}
        G_{ij}t_j
        -
        2\phi_it_i\bm{c}_l^i
    \right].
\end{split}
\end{equation*}
This implies
\begin{equation}\label{5-11}
\begin{split}
    \lim_{\epsilon\to0}
    \frac{1}{\epsilon^{[(N-2)/2+N-2]\alpha_0}}&\int_{\partial B^N(\xi_{i\epsilon},r)}\Big(\frac{\partial \omega_{i\epsilon}}{\partial\nu}v_{l\epsilon}-\frac{\partial v_{l\epsilon}}{\partial\nu}\omega_{i\epsilon}\Big)dS_x\\&=\mathcal{A}_0\mathcal{C}_0\frac{2p}{p-1} t_i
    \left[
        \sum_{j=1,j\ne i}^{m}
        G_{ij}t_j\bm{c}_l^j
        +
        \bm{c}_l^i
        \sum_{j=1,j\ne i}^{m}
        G_{ij}t_j
        -
        2\phi_it_i\bm{c}_l^i
    \right].
    \end{split}
\end{equation}

We next continue to estimate the right hand side of \eqref{5-11}.
By the definition of the Green-Robin matrix \(M\), one has
\[
    (Mt)_i
    =
    \phi_i t_i
    -
    \sum_{j=1,j\ne i}^{m}G_{ij}t_j.
\]
Using \eqref{critical}, one discovers
\begin{equation}\label{deference}
    \mathcal{A}_0^2\frac{2}{p-1}t_i(Mt)_i
    =
    \mathcal{B}_0\frac{(N-2)(p-q)}{2}\lambda_i^{\tau_0(p-q)},
\end{equation}
therefore yields
\begin{equation}\label{5-12}
    \sum_{j=1,j\ne i}^{m}G_{ij}t_j
    =
    {\phi}_i t_i
    -\frac{\mathcal{B}_0(p-q)\lambda_i^{\tau_0(p-q)}}{\mathcal{A}_{0}^2\,t_i}.
\end{equation}
Notice also that the symmetry of \(M\) gives
\[
\begin{aligned}
    t^{T}Mt
    &=
    \sum_{i=1}^{m}\phi_i t_i^2
    -
    \sum_{\substack{i,j=1\\i\ne j}}^{m}
    G_{ij}t_it_j
    =
    \sum_{i=1}^{m}\phi_i t_i^2
    -
    2\sum_{1\le i<j\le k}G_{ij}t_it_j.
\end{aligned}
\]
Consequently,
\[
    \nabla_t
    \left(
        \frac{\mathcal{A}_{0}^2}{2}t^{T}Mt
    \right)
    =
    \mathcal{A}_0^2Mt.
\]
Thus, no additional factor of \(2\) is present in the derivative of
the quadratic part of the reduced functional.

By \eqref{5-12}, we compute
\[
\begin{aligned}
    t_i
    \left[
        \sum_{j=1,j\ne i}^{m}
        G_{ij}t_j\bm{c}_l^j
        +
        \bm{c}_l^i
        \sum_{j=1,j\ne i}^{m}
        G_{ij}t_j
        -
        2\phi_it_i\bm{c}_l^i
    \right]&=
    t_i
    \sum_{j=1,j\ne i}^{m}
    G_{ij}t_j\bm{c}_l^j
    -
    \phi_i t_i^2\bm{c}_l^i
    -
    \frac{\mathcal{B}_0(p-q)}{\mathcal{A}_{0}^2}
    \lambda_i^{\tau_{0}(p-q)} \bm{c}_l^i
    \\&=
    -\bigl(D_tMD_t\bm{c}_l\bigr)_i
    -
    \frac{\mathcal{B}_0(p-q)}{\mathcal{A}_{0}^2}
    \bigl(D_{q}\bm{c}_l\bigr)_i.
\end{aligned}
\]
Combing this identity and \eqref{5-11}, we obtain the conclusion.
\end{proof}

We next aim to analyze the right-hand side of
\eqref{489}.
\begin{lem}\label{aepu}
We have
\begin{equation*}
\frac{1}{\epsilon^{(N-2)/2\alpha_0}}\int_{B(\xi_{i\epsilon},r)}\big(pu_{\epsilon}^{p-1}+\epsilon qu_{\epsilon}^{q-1}\big)\omega_{i\epsilon}v_{l\epsilon}dx\rightarrow-pL_0\bm{c}_l^i
\end{equation*}
as $\epsilon\rightarrow0$, where $\bm{c}_l^i$ can be found in \eqref{qitai} and $ L_0
    =\int_{\mathbb R^N}
    W^{p-1}Z_0^2\,dz
    >0$.
\end{lem}
\begin{proof}
Recording $U_{i\epsilon}(x)$ in \eqref{miu-1} and
define
\[
W_{i\epsilon}(x):=
(\lambda_{i,\epsilon}\epsilon^{\alpha_0})^{\tau_0}\omega_{i\epsilon}\left(\lambda_{i\epsilon}\epsilon^{\alpha_0}x+\xi_{i\epsilon}\right).
\]
The definition of \(w_{i\epsilon}\) gives the identity
\[
    W_{i\epsilon}
    =
    x\cdot\nabla U_{i\epsilon}
    +
    \frac{2}{p-1} U_{i\epsilon}.
\]
Then up to the subsequence, we have
\begin{equation}\label{w0-1}
    W_{i\epsilon}
    \longrightarrow
    x\cdot\nabla W+\frac{2}{p-1}W=-\partial_{\lambda}W_{0,\lambda}|_{\lambda=1}\quad\mbox{in}\quad C_{loc}^{1}(\mathbb{R}^N).
\end{equation}

Combining the interior gradient
estimates on the corresponding annular regions, we find
\[
    \left|U_{i\epsilon}(z)\right|
    +
    \left|W_{i\epsilon}(z)\right|
    +
    \left|V_{li\epsilon}(z)\right|
    \leq
    C\langle z\rangle^{-(N-2)},
    \qquad
    z\in\mathcal D_{i\epsilon},
\]
where
\[
    \langle z\rangle
    :=
    \bigl(1+|z|^2\bigr)^{1/2}, \hspace{3mm}
    \mathcal D_{i\epsilon}
    :=
    \frac{B_i-\xi_{i\epsilon}}{\lambda_{i\epsilon}\epsilon^{\alpha_0}}
    =
    B\left(0,\frac{r}{\lambda_{i\epsilon}\epsilon^{\alpha_0}}\right)\quad\mbox{with}\quad B_i=B(\xi_{i\epsilon},r).
\]
This estimate also accounts for the influence of all bubbles other than the
\(i\)-th one. Indeed, in the rescaled coordinates centered at the \(i\)-th
bubble, their combined contribution is \(O((\epsilon^{\alpha_0(N-2)}))\). On the
other hand, the comparability of the concentration scales ensures that
\[
    \langle z\rangle^{-(N-2)}\geq c(\epsilon^{\alpha_0(N-2)}),
    \qquad z\in\mathcal D_{i\epsilon}.
\]
Then as a consequence,
\[
    U_{i\epsilon}^{\,p-1}
    \left|W_{i\epsilon}V_{li\epsilon}\right|
    \leq
    C\langle z\rangle^{-(N-2)(p+1)}
    =
    C\langle z\rangle^{-2N},\quad
   U_{i\epsilon}^{\,q}
    \left|V_{li\epsilon}\right|
    +
U_{i\epsilon}^{\,q-1}
    \left|W_{i\epsilon}V_{li\epsilon}\right|
    \leq
    C\langle z\rangle^{-(N-2)(q+1)}.
\]
By applying the
dominated convergence theorem on the expanding domains
\(\mathcal D_{i\epsilon}\), \eqref{viep} and \eqref{w0-1},
a direct calculation tells us that
\[
\begin{aligned}
    \lim_{\epsilon\to0}
    \frac{1}{\epsilon^{(N-2)/2\alpha_0}} p\int_{B_i}u_{\epsilon}^{p-1}\omega_{i\epsilon}v_{l\epsilon}dx
    &=
    \lim_{\epsilon\to0}p(\epsilon^{\alpha_0})^{-\frac{N-2}{2}}
    (\lambda_{i\epsilon}\epsilon^{\alpha_0})^{\frac{N-2}{2}}
    \int_{\mathcal D_{i\epsilon}}
    U_{i,\epsilon}^{\,p-1}
    W_{i\epsilon}V_{li\epsilon}\,dz \\
    &=
    -pt_i
    \sum_{m=1}^{N}a_{l,i,k}
    \int_{\mathbb R^N}
    W^{p-1}[0,1]\frac{\partial W}{\partial\lambda}\partial_kW\,dz \\
    &\quad
    -pt_ia_{l,i,N+1}
    \int_{\mathbb R^N}
    W^{p-1}(\frac{\partial W}{\partial\lambda})^2\,dz.
\end{aligned}
\]
By the
oddness of the integrand, one has
\[
    \int_{\mathbb R^N}
    W^{p-1}\frac{\partial W}{\partial\lambda}\partial_kW\,dz
    =0,
    \qquad k=1,\ldots,N.
\]
Set $
L_0:=\int_{\mathbb R^N}
W^{p-1}Z_0^2\,dz>0.
$
It follows that
\[
    \lim_{\epsilon\to0}
    \frac{1}{\epsilon^{(N-2)/2\alpha_0}} p\int_{B_i^N}u_{\epsilon}^{p-1}\omega_{i\epsilon}v_{l\epsilon}dx
    =
    -pL_0t_id_{l,i,N+1}
    =
    -pL_0\bm{c}_l^i.
\]
The same change of variables,
together with the identity $N-\frac{N-2}{2}q=\frac{N-2}{2}(p-q+1)$
gives
\[
\begin{aligned}
    \frac{1}{\epsilon^{(N-2)/2\alpha_0}}\int_{B_i}\epsilon qu_{\epsilon}^{q-1}\omega_{i\epsilon}v_{l\epsilon}dx
    &=
    \epsilon q(\epsilon^{\alpha_0})^{-\frac{N-2}{2}}
    (\lambda_{i\epsilon}\epsilon^{\alpha_0})^{\,\frac{N-2}{2}(p-q+1)}
    \int_{\mathcal D_{i\epsilon}}
    U_{i\epsilon}^{\,q-1}
    W_{i\epsilon}V_{li\epsilon}\,dz \\
    &=
    O\bigl(\epsilon^{\alpha_0(N-2)}\bigr)
    =o(1).
\end{aligned}
\]
Here we used the identity
\begin{equation}\label{iden}
   \epsilon\cdot(\epsilon^{\alpha_0})^\frac{(N-2)(p-q)}{2}
    =
    (\epsilon^{\alpha_0})^{N-2}.
\end{equation}
Thus,
we conclude that the proof.
\end{proof}

\begin{lem}\label{limti}
We have
\begin{equation}\label{518-0}
\frac{1}{\epsilon^{[(N-2)/2+N-2]\alpha_0}}\frac{2(p-q)}{p-1}\epsilon \int_{B_i}u_{\epsilon}^{q}v_{l\epsilon}dx\rightarrow(\frac{2(p-q)}{p-1})^2\mathcal{B}_0\lambda_{i}^{\frac{(N-2)(p-q)}{2}}\bm{c}_l^i
\end{equation}
as $\epsilon\rightarrow0$, where $\mathcal{B}_0$ can be found in \eqref{pusaik-1-1}.
\end{lem}
\begin{proof}
Using the change of variables $x=\xi_{i\epsilon}+\lambda_{i\epsilon}\epsilon^{\alpha_0}z$, \eqref{iden},
we obtain
\[
\begin{aligned}
    \frac{1}{\epsilon^{[(N-2)/2+N-2]\alpha_0}}\frac{2(p-q)}{p-1}\epsilon \int_{B_i}u_{\epsilon}^{q}v_{l\epsilon}dx
    &=
     \frac{1}{\epsilon^{[(N-2)/2+N-2]\alpha_0}}\frac{2(p-q)}{p-1}\epsilon
    (\lambda_{i\epsilon}\epsilon^{\alpha_0})^{\,N-\frac{(N-2)q}{2}}
    \int_{\mathcal D_{i\epsilon}}
    U_{i\epsilon}^{\,q}
    V_{li\epsilon}\,dz \\
    &=
    \frac{2(p-q)}{p-1}\lambda_{i\epsilon}^{\frac{(N-2)(p-q+1)}{2}}
    \int_{\mathcal D_{i\epsilon}}
    U_{i\epsilon}^{\,q}
    V_{li\epsilon}\,dz.
\end{aligned}
\]
Furthermore, we find
\begin{equation}\label{UqZ0}
    \int_{\mathbb R^N}W^q\partial_{\lambda}W_{0,\lambda}|_{\lambda=1}\,dz
    =
    \frac{2(p-q)}{p-1}\mathcal{B}_0,
\end{equation}
where $\mathcal{B}_0$ can be found in \eqref{pusaik-1-1}.
Therefore, together with \eqref{UqZ0}, the fact that
\[
\begin{aligned}
    \int_{\mathbb R^N}W^q\partial_kW\,dz
    &=
    \frac{1}{q+1}
    \int_{\mathbb R^N}
    \partial_k\bigl(W^{q+1}\bigr)\,dz=0,
    \qquad k=1,\ldots,N,
\end{aligned}
\]
and the dominated convergence theorem, we obtain
\[
\begin{aligned}
    \int_{\mathcal D_{i\epsilon}}
    U_{i\epsilon}^{\,q}
    V_{li\epsilon}\,dz
    &\longrightarrow
    \sum_{k=1}^{N}a_{l,i,k}
    \int_{\mathbb R^N}
    W^{q}\partial_kW\,dz
    +
    a_{l,i,N+1}
    \int_{\mathbb R^N}
    W^{q}\frac{\partial W}{\partial\lambda}\,dz\\
    &=a_{l,i,N+1}\frac{2(p-q)}{p-1}\mathcal{B}_0.
\end{aligned}
\]
The boundary term at infinity vanishes because $(N-2)(q+1)>N$.
Finally, by $
    \bm{c}_l^i=t_ia_{l,i,N+1}=\lambda_{i,\epsilon}^{\frac{N-2}{2}}a_{l,i,N+1},$
we have
\begin{equation}\label{57}
    \frac{1}{\epsilon^{[(N-2)/2+N-2]\alpha_0}}\frac{2(p-q)}{p-1}\epsilon \int_{B_i}u_{\epsilon}^{q}v_{l\epsilon}dx
    \longrightarrow
    (\frac{2(p-q)}{p-1})^2\mathcal{B}_0\lambda_{i}^{\frac{(N-2)(p-q)}{2}}\bm{c}_l^i
\end{equation}
as $\epsilon\rightarrow0$. Concluding the proof.
\end{proof}

These expansions show that the additional volume term
$\frac{2(p-q)}{p-1}\epsilon \int_{B_i}u_{\epsilon}^{q}v_{l\epsilon}dx$ contributes to the leading-order matrix
and cannot be included in the remainder.
Combining the boundary and volume contributions yields $H_q$.
\begin{lem}\label{ckl}
\label{prop:effective-scale}
In the above notation
\[
 \bm{c}_{l}=(\lambda_1^{\frac{N-2}{2}}a_{l,1,N+1},\ldots,\lambda_k^{\frac{N-2}{2}}a_{l,k,N+1})^T.
\]
If \(\bm{c}_{l}\neq0\), then
\[
 \frac{\mu_{l\epsilon}-1}{\epsilon^{(N-2)\alpha_0}}
 \longrightarrow \nu_l\in\mathbb R,
\]
and
\begin{equation}
 H_q \bm{c}_{l}=pL_0\nu_l  \bm{c}_{l},
 \qquad
 H_q=\mathcal{A}_0^2(\frac{N-2}{2})^2D_tM D_t
      +\mathcal{B}_0(\frac{N-2}{2})^2(p-q)(q+1-p)D_{q},
 \label{eq:H-system}
\end{equation}
where $L_0$ is from Lemma \ref{aepu}.
At the limiting reduced critical point, \(H_q\) is the scale block
\(\nabla_y^2\Psi_k(t(y),x)\) of the leading reduced functional, where
\(y_i=\log\lambda_i\). Moreover, \(H_q\) is positive definite.
Consequently, \(\nu_l>0\).
\end{lem}

\begin{proof}
Let
\[
 \eta_{l\epsilon}
 :=\frac{\mu_{l\epsilon}-1}{\epsilon^{(N-2)\alpha_0}},
 \qquad
\mathcal E_{i\epsilon}
 :=\frac{1}{\epsilon^{[(N-2)/2+N-2]\alpha_0}}\int_{\partial B_i}\Big(\frac{\partial \omega_{i\epsilon}}{\partial\nu}v_{l\epsilon}-\frac{\partial v_{l\epsilon}}{\partial\nu}\omega_{i\epsilon}\Big)dS_x,
\]
\[
\mathcal F_{i\epsilon}
 :=\frac{1}{\epsilon^{(N-2)/2\alpha_0}}\int_{B_i}\big(pu_{\epsilon}^{p-1}+\epsilon qu_{\epsilon}^{q-1}\big)\omega_{i\epsilon}v_{l\epsilon}dx,
 \qquad
\mathcal G_{i\epsilon}
 :=\frac{1}{\epsilon^{[(N-2)/2+N-2]\alpha_0}}\frac{2(p-q)}{p-1}\epsilon \int_{B_i}u_{\epsilon}^{q}v_{l\epsilon}dx.
\]
Recalling \eqref{489}, we obtain
\begin{equation}\label{519}
 \mathcal E_{i\epsilon}
 =\eta_{l\epsilon}\mathcal F_{i\epsilon}
  -\mathcal G_{i,\epsilon}.
\end{equation}

Choose \(i_0\) such that \(\bm{c}_l^{i_0}\neq0\), then by Lemma \ref{aepu}
, for all sufficiently small \(\epsilon\),
\[
 \eta_{l\epsilon}
 =
 \frac{
    \mathcal E_{i_0\epsilon}
   +\mathcal G_{i_0\epsilon}}
 {\mathcal F_{i_0\epsilon}}.
\]
Lemmas \ref{m00} and \ref{limti}
show that the numerator has a finite limit. Therefore,
\[
 \eta_{l\epsilon}\longrightarrow\nu_l
 \quad\text{for some }\nu_l\in\mathbb R.
\]

We may now pass to the limit in \eqref{519}
componentwise. Using Lemma \ref{m00}, Lemma
\ref{aepu} and Lemma \ref{limti}, we obtain
\[
 -\mathcal{A}_0\mathcal{C}_0\frac{2p}{p-1}
 \left(
   D_tM D_t
   +\frac{\mathcal{B}_0(p-q)}{\mathcal{A}_0^2}D_{q}
 \right)\bm{c}_l
 =
 -pL_0\nu_\ell \bm{c}_l-\mathcal{B}_0(\frac{2(p-q)}{p-1})^2B_{pq} \bm{c}_l.
\]
Since \(\mathcal{C}_0=\frac{N-2}{2p}\mathcal{A}_0\), this identity is equivalent to
\eqref{eq:H-system}.

We next identify \(H_q\). Set
$
 y_i:=\log\lambda_i,
 \qquad
 t_i=e^{\frac{N-2}{2}y_i},
$
and by \eqref{pusaik-0}, we write the functional $ \Psi_{\Omega,q}:=\Psi_{\Omega,q}(t(y),x)$, i.e
\[
 \Psi_{\Omega,q}
 :=
 \frac{\mathcal{A}_0^2}{2}\,t(y)^TM(x)t(y)
 -\mathcal{B}_0\sum_{i=1}^m\lambda_{i}^{\tau_0(p-q)}.
\]
Direct differentiation gives
\[
 \partial_{y_i}\Psi_{\Omega,q}
 =
 \mathcal{A}_0^2\frac{N-2}{2}\,t_i(Mt)_i
 -\mathcal{B}_0\frac{(N-2)(p-q)}{2}\lambda_{i}^{\tau_0(p-q)}.
\]
At the limiting reduced critical point, a second differentiation yields
\[
 \nabla_y^2\Psi_{\Omega,q}
 =
 A_0^2(\frac{N-2}{2})^2D_tM D_t
 +\mathcal{B}_0(\frac{N-2}{2})^2(p-q)(q+1-p)D_{q}
 =H_q.
\]
It remains to prove that \(H_q\) is positive definite.
Since \(m_{ij}=-G_{ij}\) for \(i\neq j\), a direct expansion gives
\begin{equation}
 \bm{c}_{l}^TD_tM D_t\bm{c}_l
 =
 \sum_{i<j}G_{ij}t_it_j(\bm{c}_i-\bm{c}_j)^2
 +\sum_{i=1}^k t_i(Mt)_i(\bm{c}_l^i)^2.
 \label{eq:graph-identity}
\end{equation}
Combining \eqref{eq:graph-identity} and the definition of $H_q$, we find
\begin{equation}
\bm{c}_l^TH_q\bm{c}_l
 =
\mathcal{ A}^2(\frac{N-2}{2})^2\sum_{i<j}G_{ij}t_it_j(\bm{c}_l^i-\bm{c}_l^j)^2
 +\mathcal{B}_0\frac{(N-2)(p-q)}{2\alpha_0}
  \sum_{i=1}^k\lambda_i^{\tau_0(p-q)} (\bm{c}_l^i)^2.
 \label{eq:H-positive}
\end{equation}
Since \(G_{ij}\geq0\), it follows that, for every \(\bm{c}_l\neq0\), we have
\begin{equation}
 \bm{c}_l^TH_q\bm{c}_l
 \geq
 \mathcal{B}_0(\frac{N-2}{2})^2(p-q)(q+2-p)
 \min_{1\leq i\leq k}\lambda_i^{\tau_0(p-q)}\,|\bm{c}_l|^2
 >0.
\label{eq:H-coercive}
\end{equation}
Thus \(H_q\) is positive definite. Taking the Euclidean inner product
of \eqref{eq:H-system} with \(\bm{c}_l\) gives
\[
 \nu_\ell
 =\frac{\bm{c}_l^TH_q\bm{c}_l}{pL_0|\bm{c}_l|^2}>0.
\]
This completes the proof of Lemma \ref{ckl}.
\end{proof}

We can now prove Theorem \ref{dilation}.
\begin{proof}[\textbf{Proof of Theorem \ref{dilation}}]
By Lemma \ref{identity} we have to prove that for any $m+1\leq l\leq(N+1)m$ and $i\in\{1,\cdots,m\}$ the number $a_{l,i,N+1}$ in \eqref{viep} is zero. In order to prove this fact, we will use Lemma~\ref{prop:effective-scale}.

Let \(\epsilon_n\downarrow0\) be arbitrary. Passing to a subsequence,
still denoted by $\epsilon$.
Since \(t_i>0\), it is enough to prove that
\[
 \bm{c}_l=(t_1a_{l,1,N+1},\ldots,t_ka_{l,m, N+1})^T
\]
vanishes.
Suppose, to the contrary, that \(\bm{c}_l\neq0\). In virtue of
Lemma~\ref{prop:effective-scale}, we have
\[
 \frac{\mu_{l\epsilon}-1}{\epsilon^{(N-2)\alpha_0}}
 \longrightarrow
 \nu_l
 =\frac{\bm{c}_l^TH_q\bm{c}_l}{pL_0|\bm{c}_l|^2}>0.
\]
On the other hand, estimate of the eigenvalues $\mu_{l,\epsilon}$ in \eqref{baowen1} gives, for $m+1\leq l\leq(N+1)m$,
\[
 \frac{\mu_{l\epsilon}-1}{\epsilon^{(N-2)\alpha_0}}
 \leq
 C\epsilon^{2\alpha_0}
 \longrightarrow0.
\]
Thus we clearly get a contradiction and concluding the proof.
\end{proof}

\section{Asymptotic behaviour of the eigenvalues, $m+1\leq l\leq(N+1)m$}\label{section-6}
In this section, we will prove Theorem \ref{vepsilon}. Motivated by the argument in \cite{CKL-1}, we prove its improvement \eqref{116}, which is stated once more in the following proposition.
Before this, we define the row vectors and matrix as follow:
$$
\mathcal{G}(x)=(G(x,\xi_1),\cdots,G(x,\xi_m)),
$$
\begin{equation}\label{gX}
\mathcal{\tilde{G}}_{(i,k)}(x)=\lambda_i^\frac{N}{2}\partial_{y_k}G(x,\xi_i),\quad (\tilde{\bm{d}}_{l})_{(i,k)}=\lambda_i^{\frac{N-2}{2}}a_{l,i,k},
\end{equation}

\begin{Prop}\label{propos-6-1}
We write
\begin{equation}\label{l2matrix}
N_{q}:=M+
\frac{\mathcal{B}_0(p-q)(q+1-p)}{\mathcal{A}_0^2}\lambda_i^{\frac{(N-2)(p-q-2)}{2}}.
\end{equation}
If $m+1\leq l\leq(N+1)m$, then
\begin{equation}\label{exterior}
\epsilon^{-\alpha_0(N-1)}v_{l\epsilon}\rightarrow \mathcal{A}_0\Big(\mathcal{G}(x)N_{q}^{-1}P+\tilde{\mathcal{G}}(x)\Big)\tilde{\bm{d}}_{l}
\end{equation}
in $C_{loc}^{1}(\Omega\setminus\{\xi_1,\cdots,\xi_m\})$ as $\epsilon\rightarrow0$, where $A$, $P$ and $\tilde{\bm{d}}_{l}$ are defined in \eqref{pusaik-1-1}, \eqref{pusaik-4} and \eqref{pusaik-3}, respectively.
\end{Prop}

We will present the proof by the following some preliminary results.
\begin{lem} It holds that
\begin{equation}\label{tao0}
\epsilon^{\alpha_0(1-\tau_0)}\int_{B_{i\epsilon}}
\big(pu_{\epsilon}^{p-1}+\epsilon qu_{\epsilon}^{q-1}\big)\partial_k u_\epsilon
\,v_{l\epsilon}\,dx
\longrightarrow
-p\mathcal{C}_1\lambda_i^{\tau_0-1}a_{l,i,k},
\end{equation}
where
$
\mathcal{C}_1:=\int_{\mathbb R^N}
W^{p-1}(\partial_1W)^2\,dz>0.
$
\end{lem}

\begin{proof}
We write
\begin{equation}\label{ILK}
I_{lik\epsilon}:=
\int_{B_{i\epsilon}}
\big(pu_{\epsilon}^{p-1}+\epsilon qu_{\epsilon}^{q-1}\big)\partial_k u_\epsilon
\,v_{l\epsilon}\,dx.
\end{equation}
A change of variables gives
\[
\begin{aligned}
(\lambda_{i\epsilon}\epsilon^{\alpha_0})^{1-\tau_0}I_{lik\epsilon}
={}&p\int_{B(0,r/(\lambda_{i\epsilon}\epsilon^{\alpha_0})}
U_{i\varepsilon}^{p-1}
\partial_k U_{i\epsilon}
V_{li\epsilon}\,dz
+\epsilon q\delta_{i\epsilon}^{\tau_0(p-q)}
\int_{B(0,r/(\lambda_{i\epsilon}\epsilon^{\alpha_0})}
U_{i\epsilon}^{q-1}
\partial_k U_{i\epsilon}
V_{li\epsilon}\,dz.
\end{aligned}
\]
Combining some pointwise estimates \eqref{tidu} and \eqref{V5}, we obtain
\[
U_{i\epsilon}(z)\le C(1+|z|)^{-(N-2)},
\quad
|\nabla U_{i\epsilon}(z)|
\le C(1+|z|)^{-(N-1)},
\quad
|V_{li\epsilon}(z)|
\le C(1+|z|)^{-(N-2)}.
\]
Thus dominated convergence and radial symmetry of \(W\) yield
\[
p\int_{B(0,r/(\lambda_{i\epsilon}\epsilon^{\alpha_0}))}
U_{i\epsilon}^{p-1}
\partial_k U_{i\epsilon}
V_{li\epsilon}\,dz
\longrightarrow
-p\mathcal{C}_1d_{l,i,k},\quad
\mathcal{C}_1:=\int_{\mathbb R^N}
W^{p-1}(\partial_1W)^2\,dz>0.
\]
Moreover, the perturbation term satisfies
\[
\begin{aligned}
&\epsilon q(\lambda_{i\epsilon}\epsilon^{\alpha_0})^{\tau_{0}(p-q)}
\left|
\int_{B(0,r/(\lambda_{i\epsilon}\epsilon^{\alpha_0}))}
U_{i\epsilon}^{q-1}
\partial_k U_{i\epsilon}
V_{l i\epsilon}\,dz
\right|\\&
\qquad\le
C\epsilon(\lambda_{i\epsilon}\epsilon^{\alpha_0})^{\tau(p-q)}
\int_{\mathbb R^N}
(1+|z|)^{-(N-2)(q+1)-1}\,dz
=O(\epsilon^{\alpha_0(N-2)})=o(1).
\end{aligned}
\]
Hence the conclusion follows.
\end{proof}

As a consequence of Theorem \ref{dilation} and Lemma \ref{viep}, we deduce the following important eigenvalue estimate.
\begin{cor}
For $m+1\leq l\leq(N+1)m$, one has
\begin{equation}\label{cor1}
|\mu_{l\epsilon}-1|=O(\epsilon^{\alpha_{0}(N-1)}).
\end{equation}
\end{cor}
\begin{proof}
Differentiating the equation for \(u_\epsilon\) gives
\begin{equation}\label{l21:derivative}
 -\Delta(\partial_k u_\epsilon)
 =\big(pu_{\epsilon}^{p-1}+\epsilon qu_{\epsilon}^{q-1}\big)\partial_k u_\epsilon.
\end{equation}
Together with the eigenvalue equation, Green's identity therefore yields
\begin{equation}\label{eq:translationidentity}
F_{lik\epsilon}
=(\mu_{l\epsilon}-1)I_{lik\epsilon},
\end{equation}
where $I_{lik\epsilon}$ is defined in \eqref{ILK},
\[
\begin{aligned}
F_{\ell ik\epsilon}
&:=
\int_{\partial B_{i\epsilon}}
\left(
\partial_\nu(\partial_k u_\epsilon)\,v_{l\epsilon}
-\partial_k u_\epsilon\,\partial_\nu v_{l\epsilon}
\right)\,dS.
\end{aligned}
\]
Here \(\nu\) denotes the outward unit normal to \(B_{i\epsilon}\).
On fixed annuli containing \(\partial B_{i\epsilon}\), by the Lemma \ref{Lem6.1}, Lemma \ref{outside-2}, Lemma \ref{AS-eigenvalues-1}, Corollary \ref{AS-eigenvalues-2} and regularity theory, we deduce
\[
|\nabla u_\epsilon|+|D^2u_\epsilon|
\le C\epsilon^{\alpha_0\tau_0},
\qquad
|v_{l\epsilon}|+|\nabla v_{l\epsilon}|
\le C\epsilon^{\alpha_0(N-2)},
\]
which implies
\[
F_{l ik\epsilon}
=O(\epsilon^{3\alpha_0\tau_0}).
\]
Combining \eqref{tao0}, up to the subsequence, we have
\[
|\mu_{\ell\varepsilon}-1|
=\frac{|F_{\ell ik\varepsilon}|}
{|I_{\ell ik\varepsilon}|}
\le C\epsilon^{\alpha_{0}(3\tau_0-(\tau_0-1))}
=C\epsilon^{\alpha_{0}(N-1)}.
\]
Finally, if the asserted estimate failed for the full family, there would exist a sequence along which
\[
\frac{|\mu_{l\varepsilon}-1|}
{\epsilon^{\alpha_0(N-1)}}\longrightarrow\infty.
\]
Applying the preceding argument to a subsequence gives a contradiction. This completes the proof.

\end{proof}
\begin{lem}\label{positivede}
In the above notation the limiting parameters satisfy
the identity \eqref{deference}.
Then $N_q$ is a real symmetric positive definite matrix.
In particular, $N_q$ is invertible.
\end{lem}
\begin{proof}
Symmetry follows from the symmetry of the Green function.
Since $t_i=\lambda_i^\frac{N-2}{2}>0$, we obtain
\[
\frac{(M t)_i}{t_i}
=
\frac{\mathcal{B}_0(p-q)}{\mathcal{A}_0^2}\lambda_i^{\frac{(N-2)(p-q-2)}{2}}.
\]
For any $z=(z_1,\ldots,z_m)^T\in\mathbb R^m$, expanding the squares
and using $\mathcal m_{ij}=-G(\xi_i,\xi_j)$ for $i\ne j$, we obtain
\[
z^TM z
=
\sum_{i<j}G(\xi_i,\xi_j)t_it_j
\left(\frac{z_i}{t_i}-\frac{z_j}{t_j}\right)^2
+
\sum_{i=1}^m\frac{( M t)_i}{t_i}z_i^2.
\]
Consequently, the definition of $N_q$ yields
\[
z^TN_q z
=
\sum_{i<j}G(\xi_i,\xi_j)t_it_j
\left(\frac{z_i}{t_i}-\frac{z_j}{t_j}\right)^2
+
\frac{\mathcal{B}_0(p-q)(q+2-p)}{\mathcal{A}_0^2}
\sum_{i=1}^m\lambda_i^{\frac{(N-2)(p-q-2)}{2}}z_i^2.
\]
The first term is nonnegative because $G(\xi_i,\xi_j)>0$.
The second is strictly positive whenever $z\ne0$, since
$\mathcal{B}_0>0$, $1<p<q+2$, and $\lambda_i>0$.
Thus $N_q$ is positive definite and hence invertible.
\end{proof}

\begin{Def} We define the $\ast$-norm as
\begin{equation}\label{l21:weight}
 \mathfrak w_\epsilon(x)
 =\sum_{i=1}^m\left(
 \frac{\epsilon^{\alpha_0}}{\epsilon^{\alpha_0}+|x-\xi_{i\epsilon}|}
 \right)^{N-2},
 \qquad
 \|f\|_*:=\sup_{x\in\Omega}\frac{|f(x)|}{\mathfrak w_\epsilon(x)},
 \qquad \mathcal T:=(-\Delta_D)^{-1}.
\end{equation}
Here $-\Delta_D$ is the Dirichlet Laplacian,
\end{Def}

By \eqref{cor1} and \eqref{V5}, we deduce
\begin{equation}\label{l21:spectral}
 \|v_{l\epsilon}\|_*\le C,
 \qquad |\mu_{l\epsilon}-1|\le C(\epsilon^{\alpha_0}).
\end{equation}

Choose $R_0>0$ so that the balls
$B_{i\epsilon}:=B(\xi_{i\epsilon},R_0)$ are pairwise disjoint and
compactly contained in $\Omega$. Let
$\chi_{i\epsilon}\in C_c^\infty(B_{i\epsilon})$ equal one on
$B(\xi_{i\epsilon},R_0/2)$, with its first two derivatives bounded
uniformly in $\epsilon$. We keep the dependence on $\epsilon$ in
the cutoff notation. Set
\begin{equation}\label{l21:psi}
 \psi_{ik\epsilon}
 :=\delta_{i\epsilon}^{a+1}\chi_{i\epsilon}\partial_k u_\epsilon,
 \qquad 1\le i\le m,\quad 1\le k\le N.
\end{equation}
Then we have
$\psi_{ik\epsilon}(\xi_{i\epsilon}+\delta_{i\epsilon}z)
\to\partial_kW(z)$. Moreover,
\eqref{tidu} and scale comparability imply
$\|\psi_{ik\epsilon}\|_*\le C$.

Fix a nonnegative radial function
$\vartheta\in C_c^\infty(\mathbb R^N)$, $\vartheta\not\equiv0$.
Choose $c_{ik\epsilon}$ so that
\begin{equation}\label{l21:zeta}
 \zeta_\epsilon
 :=v_{l\epsilon}-\sum_{i=1}^m\sum_{k=1}^N
             c_{ik\epsilon}\psi_{ik\epsilon}
\end{equation}
satisfies
\begin{equation}\label{l21:orthog}
 \int_{\mathbb R^N}\vartheta(z)W^{p-1}(z)\partial_jW(z)
       \zeta_\epsilon(\xi_{i\epsilon}+\lambda_{i\epsilon}\epsilon^{\alpha_0}z)\,dz=0,
 \qquad j=1,\ldots,N.
\end{equation}
All integrands in \eqref{l21:orthog} are supported inside the rescaled
domain for small $\epsilon$. Since the cutoff supports are disjoint,
the equations for different $i$ decouple. Their coefficient matrices are
\[
 M_{i\epsilon,jk}
 =\int_{\mathbb R^N}\vartheta W^{p-1}\partial_jW\,
             \partial_kU_{i\epsilon}\,dz
 \longrightarrow
 \int_{\mathbb R^N}\vartheta W^{p-1}\partial_jW\partial_kW\,dz.
\]
By radial symmetry and parity, the limiting matrix is
$c_\vartheta I_N$, where
$c_\vartheta=\int\vartheta W^{p-1}(\partial_1W)^2>0$.
The matrices are therefore uniformly invertible. Since the right-hand
sides are bounded by \eqref{l21:spectral}, the coefficients exist uniquely
and satisfy $|c_{ik\epsilon}|\le C$. Consequently,
$\|\zeta_\epsilon\|_*\le C$.

Introduce the critical masses
\begin{equation}\label{l21:masses}
 k_{i\epsilon}:=\mu_\epsilon p
       \int_{B_{i\epsilon}}u_\epsilon^{p-1}v_{l\epsilon}\,dx,
 \qquad k_\epsilon=(k_{1\epsilon},\ldots,k_{m\epsilon})^T,
 \qquad
 \mathfrak h_\epsilon:=\epsilon^{\alpha_0}
             +\epsilon^{-2\alpha_0\tau_0}\|k_\epsilon\|_\infty.
\end{equation}
Thus $\mathfrak h_\epsilon\ge\epsilon^{\alpha_0}>0$. No sign condition
on $k_{i\epsilon}$ and no $O(\epsilon^{\alpha_0(2\tau_0+1)})$ mass estimate
are imposed.

\begin{lem}\label{l21:lemma}
There exists $C>0$, independent of
$\epsilon$, such that
\begin{equation}\label{l21:estimate}
 \|\zeta_\epsilon\|_*\le C\mathfrak h_\epsilon.
\end{equation}
For every fixed $K\subset\Omega\setminus\{\xi_1,\ldots,\xi_m\}$, there exists a constant $C_K$, depends on the fixed exponent $q$, such that
\begin{equation}\label{l21:exterior}
 \|v_{l\epsilon}\|_{C^1(K)}
 \le C_K\epsilon^{2\alpha_0\tau_0}\mathfrak h_\epsilon.
\end{equation}
\end{lem}

\begin{proof}
We first establish the weighted Green estimates used in the proof. For
$R\ge2$, let
\[
 \mathcal O_{R\epsilon}
 :=\Omega\setminus\bigcup_{i=1}^m B(\xi_{i\epsilon},R\epsilon^{\alpha_0}).
\]
We claim that
\begin{align}
 \|\mathcal T(pu_\epsilon^{p-1}\mathfrak w_\epsilon)\|_*
 &\le C,\label{l21:green1}\\
 \|\mathcal T(\mathbf1_{\mathcal O_{R\epsilon}}
           pu_\epsilon^{p-1}\mathfrak w_\epsilon)\|_*
 &\le CR^{-2},\label{l21:green2}\\
 \|\mathcal T(\epsilon q u_\epsilon^{q-1}\mathfrak w_\epsilon)\|_*
 &\le C(\epsilon^{\alpha_0})^{\min\{2\tau_0,2\tau_0q-2\}}
                (1+|\log(\epsilon^{\alpha_0})|)\longrightarrow0.
 \label{l21:green3}
\end{align}
Here and below, $\mathbf1_{O_{R\epsilon}}$ denotes the indicator function of a set $O_{R\epsilon}$.

Write $r_i(y)=|y-\xi_{i\epsilon}|$ and $d(y)=\min_i r_i(y)$.
By \eqref{tidu},
$W_\epsilon(y)\le C(\epsilon^{\alpha_0})^{\tau_0}((\epsilon^{\alpha_0})+d(y))^{-2\tau_0}$
and $\mathfrak w_\epsilon(y)\le
C(\epsilon^{\alpha_0})^{2\tau_0}((\epsilon^{\alpha_0})+d(y))^{-{2\tau_0}}$.
Using $\tau_0(p-1)=2$ and $2\tau_0-\tau_0(p-q)+\tau_0(q-1)=2\tau_0q-2$, we obtain
\begin{align}
 pu_\epsilon^{p-1}\mathfrak w_\epsilon
 &\le C\sum_{i=1}^m
       \frac{\epsilon^{\alpha_0(2\tau_0+2)}}{((\epsilon^{\alpha_0})+r_i)^{2\tau_0+4}},
       \label{l21:sources1}\\
 \epsilon q u_\epsilon^{q-1}\mathfrak w_\epsilon
 &\le C\sum_{i=1}^m
       \frac{(\epsilon^{\alpha_0(2\tau_0+2\tau_0q-2)})}{(\epsilon^{\alpha_0}+r_i)^{2\tau_0q}}.
       \label{l21:sources2}
\end{align}

Set
$X=(x-\xi_{i\epsilon})/(\epsilon^{\alpha_0})$,
$z=(y-\xi_{i\epsilon})/(\epsilon^{\alpha_0})$. By the definition of $\mathcal{T}$, we have
$\mathcal T f(x)=\int_\Omega G(x,y)f(y)\,dy$. Then together with $G(x,y)\le C|x-y|^{-2\tau_0}$ and \eqref{l21:sources1}, we obtain
\begin{equation}\label{Gxy1}
\begin{split}
\mathcal T(pu_\epsilon^{p-1}\mathfrak w_\epsilon)(x)&\leq C\sum_{i=1}^{m}\int_{\Omega}\frac{1}{|x-y|^{2\tau_0}}\frac{\epsilon^{\alpha_0(2\tau_0+2)}}{((\epsilon^{\alpha_0})+|y-\xi_{i\epsilon}|)^{2\tau_0+4}}dy\\&
\leq C\int_{\mathbb R^N}|X-z|^{-2\tau_0}(1+|z|)^{-N-2}\,dz.
\end{split}
\end{equation}
We just need to obtain the estimate for $|X|>2$. If $t=|X|>2$, split into
$|z|<t/2$, $|X-z|<t/2$, and $\mathbb{R}^N\setminus(\{|z|<t/2\}\cup\{|X-z|<t/2\})$. The first part is
$O(t^{-2\tau_0})$; the second is $O(t^{-N})$ because
$\int_{|w|<t/2}|w|^{-s}\,dw=Ct^2$. In the complement, splitting again at $|z|=2t$ gives $O(t^{-N})$. Hence, we conclude that
\[
 \int_{\mathbb R^N}|X-z|^{-2\tau_0}(1+|z|)^{-N-2}\,dz
 \le C(1+|X|)^{-2\tau_0}.
\]
This proves the estimate \eqref{l21:green1}.

Similar to the above argument, we will show that
\begin{equation}\label{l21:newton-tail}
 \int_{|z|>R}|X-z|^{-2\tau_0}(1+|z|)^{-N-2}\,dz
 \le CR^{-2}(1+|X|)^{-2\tau_0}.
\end{equation}
Indeed, if $t=|X|\ge R/2$, the part with $|z|<t/2$ is at most
\[
 Ct^{-2\tau_0}\int_R^\infty r^{N-1}(1+r)^{-N-2}\,dr
 \le CR^{-2}t^{-2\tau_0}.
\]
The part with $|X-z|<t/2$ is $O(t^{-N})$. In the complement with
$|z|\le2t$, both $|z|$ and $|X-z|$ are at least $t/2$, so its contribution
is also $O(t^{-N})$. For $|z|>2t$, use
$|X-z|\ge|z|/2$ and integrate $r^{-2\tau_0-3}$ to obtain $O(t^{-N})$.
Since $t\ge R/2$ and $N=2\tau_0+2$, each such term is bounded by
$CR^{-2}t^{-2\tau_0}$. If $|X|<R/2$, then $|X-z|\ge|z|/2$ for $|z|>R$,
and the integral is at most
$CR^{-2\tau_0-2}\le CR^{-2}(1+|X|)^{-2\tau_0}$.
For $y\in\mathcal O_{R\epsilon}$, every $r_i(y)>R(\epsilon^{\alpha_0})$;
thus \eqref{l21:green2} follows.

To prove \eqref{l21:green3}, we set $\sigma=2\tau_0q>2$.
Choose a fixed $D>4\operatorname{diam}\Omega+4$ and set
$L=D/(\epsilon^{\alpha_0})$. All relevant rescaled points lie in $B_{L/2}$
for small $\epsilon$. In virtue of \eqref{l21:sources2}, and
 $\mathfrak w_{i\epsilon}=(1+|X|)^{-s}$, we deduce
\begin{equation}\label{l21:pert-conv}
\frac{\mathcal T(\epsilon q u_\epsilon^{q-1}\mathfrak w_\epsilon)}{\mathfrak w_{i\epsilon}}\leq C(\epsilon^{2\alpha_0\tau_0})(1+|X|)^{2\tau_0}
 \int_{|z|<L}|X-z|^{-2\tau_0}(1+|z|)^{-\sigma}\,dz.
\end{equation}
If $|X|\le2$, the right-hand side of \eqref{l21:pert-conv} is uniformly bounded,
because $\sigma>2$. For $t=|X|>2$, handling the cases $\{|z|<t/2\}$, $\{|X-z|<t/2\}$ and $\mathbb{R}^N\setminus(\{|z|<t/2\}\cup\{|X-z|<t/2\})$
separately, one can derive
\begin{equation}\label{L1}
\begin{split}
 \int_{\{|z|<t/2\}}|X-z|^{-2\tau_0}(1+|z|)^{-\sigma}\,dz&\leq Ct^{-2\tau_0}\int_0^{t/2}r^{N-1}(1+r)^{-\sigma}\,dr\\&
 \le Ct^{-2\tau_0}\begin{cases}
 1,&\sigma>N,\\
 1+\log t,&\sigma=N,\\
 t^{N-\sigma},&2<\sigma<N.
 \end{cases}
 \end{split}
\end{equation}
For the region $|X-z|<t/2$,  we have the bound $Ct^{2-\sigma}$.
In the complement with $|z|\le2t$ has the same bound; the part with
$|z|>2t$ is bounded by
$C\int_{2t}^{L}r^{1-\sigma}\,dr\le Ct^{2-\sigma}$
when that part is nonempty. Thanks to $t<L/2$, we are able to conclude that
\[
 (1+|X|)^{2\tau_0}\int_{|z|<L}|X-z|^{-2\tau_0}(1+|z|)^{-\sigma}\,dz
 \le C\begin{cases}
 1,&\sigma>N,\\
 1+\log L,&\sigma=N,\\
 L^{N-\sigma},&2<\sigma<N.
 \end{cases}
\]
Finally, multiplying by $\epsilon^{2\alpha_0\tau_0}$ and using
$2\tau_0-N+\sigma=2\tau_0q-2$ gives \eqref{l21:green3}.

We next aim to prove \eqref{l21:estimate}. By \eqref{l21:derivative}, it is easy to check that
\begin{align}
 \big(-\Delta-(pu_{\epsilon}^{p-1}+\epsilon qu_{\epsilon}^{q-1})\big)\zeta_\epsilon&=F_\epsilon,\quad\mbox{with}\hspace{2mm}
 F_\epsilon
 :=(\mu_{l\epsilon}-1)\big(pu_{\epsilon}^{p-1}+\epsilon qu_{\epsilon}^{q-1}\big)v_{l\epsilon}
       -\sum_{i,k}c_{ik\epsilon}E_{ik\epsilon}, \label{l21:remainder-equation}
 \end{align}
 where
\begin{align}
 E_{ik\epsilon}
 &:=\big(-\Delta-(pu_{\epsilon}^{p-1}+\epsilon qu_{\epsilon}^{q-1})\big)\psi_{ik\epsilon}\nonumber\\
 &=-\delta_{i\epsilon}^{a+1}
  \left(2\nabla\chi_{i\epsilon}\cdot\nabla\partial_k u_\epsilon
       +(\Delta\chi_{i\epsilon})\partial_k u_\epsilon\right).
       \label{l21:commutator}
\end{align}
The functions $E_{ik\epsilon}$ are supported on fixed-width annuli
away from the centers. By \eqref{tidu},
$\|E_{ik\epsilon}\|_\infty\le C(\epsilon^{\alpha_0(2\tau_0+1)})$.
Moreover, $\sup_x\int_\Omega G(x,y)\,dy\le C_\Omega$ and
$\mathfrak w_\epsilon(x)\ge c(\epsilon^{2\alpha_0\tau_0})$ in $\Omega$.
It follows that
$\|\mathcal T(|E_{ik\epsilon}|)\|_*\le C(\epsilon^{\alpha_0})$. Combining this bound and
the positivity of the Green kernel, \eqref{cor1}, and
\eqref{l21:green1}--\eqref{l21:green3}, we get
\begin{equation}\label{l21:forcing}
 \|\mathcal T(|F_\epsilon|)\|_*
 \le C\bigl((\epsilon^{\alpha_0})+|\mu_{l\epsilon}-1|\bigr)
 \le C(\epsilon^{\alpha_0}).
\end{equation}

By contradiction, if \eqref{l21:estimate} does not hold true, up to a subsequence, we choose a subsequence $\epsilon\rightarrow0$ such that
\begin{equation}\label{l21:normalization}
 n_\epsilon:=\|\zeta_\epsilon\|_*,
 \qquad \frac{n_\epsilon}{\mathfrak h_\epsilon}\longrightarrow\infty,
 \qquad Z_\epsilon:=\frac{\zeta_\epsilon}{n_\epsilon}.
\end{equation}
Then $\|Z_\epsilon\|_*=1$. Since
$\epsilon^{\alpha_0}\le\mathfrak h_\epsilon$,
\begin{equation}\label{l21:small-ratios}
 \frac{\epsilon^{\alpha_0}}{n_\epsilon}\longrightarrow0,
 \qquad \frac{|\mu_{l\epsilon}-1|}{n_\epsilon}\longrightarrow0,
 \qquad
 \left\|\mathcal T\left(\frac{|F_\epsilon|}{n_\epsilon}\right)
 \right\|_*\longrightarrow0.
\end{equation}
We define $i$-th center by writing
$Z_{i\epsilon}(z)=Z_\epsilon(\xi_{i\epsilon}
+\lambda_{i\epsilon}\epsilon^{\alpha_0}z)$. Then we have
$$
-\Delta Z_{i\epsilon}=(\lambda_{i\epsilon}\epsilon^{\alpha_0})^2(pu_{\epsilon}^{p-1}+\epsilon qu_{\epsilon}^{q-1})(\xi_{i\epsilon}
+\lambda_{i\epsilon}\epsilon^{\alpha_0}z)Z_{i\epsilon}+\frac{(\lambda_{i\epsilon}\epsilon^{\alpha_0})^2}{n_{\epsilon}}F_{\epsilon}(\xi_{i\epsilon}+\lambda_{i\epsilon}\epsilon^{\alpha_0}z).
$$
Notice that
\begin{equation}\label{l21:rescaled-potential}
\begin{split}
 (\lambda_{i\epsilon}\epsilon^{\alpha_0})^2&(pu_{\epsilon}^{p-1}+\epsilon qu_{\epsilon}^{q-1})
       (\xi_{i\epsilon}+\lambda_{i\epsilon}\epsilon^{\alpha_0}z)
 =pU_{i\epsilon}^{p-1}
   +q\epsilon(\lambda_{i\epsilon}\epsilon^{\alpha_0})^{\tau_0(p-q)}
           U_{i\epsilon}^{q-1},\\&
 \epsilon(\lambda_{i\epsilon}\epsilon^{\alpha_0})^{\tau_{0}(p-q)}
 =(\epsilon^{2\alpha_0\tau_0})\lambda_{i\epsilon}^{\tau_0(p-q)}\longrightarrow0.
\end{split}
\end{equation}
Therefore,
by the standard elliptic regularity theorem, \eqref{l21:rescaled-potential} together with $U_{i\epsilon}\rightarrow W$, and \eqref{l21:small-ratios}, we have a subsequence of
$Z_{i\epsilon}$ converges to $Z_i$ in $C^1_{\mathrm{loc}}(\mathbb R^N)$ which satisfies that
\begin{equation}\label{l21:limit-equation}
 -\Delta Z_i=pW^{p-1}Z_i\quad\hbox{in }\mathbb R^N,
 \qquad |Z_i(z)|\le C(1+|z|)^{-(N-2)}.
\end{equation}

Let $\eta_R$ be a smooth cut-off equal on $\mathbb{R}^N$ such that $\eta_R|_{B_R}=1$, $\eta_R|_{B^{c}_{2R}}=0$ and
$|\nabla\eta_R|\le C/R$. Multiply \eqref{l21:limit-equation} by
$\eta_R^2Z_i$ and using Young's inequality gives
\[
 \int\eta_R^2|\nabla Z_i|^2
 \le 2p\int W^{p-1}Z_i^2
       +C\int|\nabla\eta_R|^2Z_i^2.
\]
The first integral is finite by \eqref{l21:limit-equation}; the second
is at most $CR^{N-2-4\tau_0}=CR^{2-N}$.
Thus $\nabla Z_i\in L^2(\mathbb R^N)$. The same decay bound gives
$Z_i\in L^{2^*}(\mathbb R^N)$, where $2^*=2N/(N-2)$.
Finally,
\[
 \|\nabla(Z_i-\eta_R Z_i)\|_2^2
 \le 2\int_{|z|>R}|\nabla Z_i|^2
       +2\int|\nabla\eta_R|^2Z_i^2\longrightarrow0.
\]
After smooth approximation of $\eta_RZ_i$, this proves
$Z_i\in D^{1,2}(\mathbb R^N)$.
Owing to the nondegeneracy of solution \cite[Lemma 2.10]{CKL},
\begin{equation}\label{l21:kernel}
 Z_i=b_iZ_0+\sum_{k=1}^N b_{ik}\partial_kW,
 \qquad Z_0:=-(z\cdot\nabla W+aW).
\end{equation}
Divide \eqref{l21:orthog} by $n_\epsilon$ and pass to the limit as $\epsilon\rightarrow0$, and \eqref{l21:kernel}, we obtain
\begin{equation*}
 0=\int_{\mathbb R^N}\vartheta(z)W^{p-1}(z)\partial_jW(z)
       [b_iZ_0+\sum_{k=1}^N b_{ik}\partial_kW]\,dz=b_{ik}\int_{\mathbb R^N}\vartheta W^{p-1}(\partial_jW)^2dz,
\end{equation*}
where we have used the oddness. Hence $b_{ik}=0$ for every $k$.
It remains to prove $b_i=0$ for every $i$. Moreover, by integration by parts,
\begin{equation}\label{l21:translation-mass}
 \begin{aligned}
 p\int_{B_{i\epsilon}}u_\epsilon^{p-1}\psi_{ik\epsilon}\,dx
 &=(\lambda_{i\epsilon}\epsilon^{\alpha_0})^{\tau_0+1}
      \int_{B_{i\epsilon}}\chi_{i\epsilon}\partial_k(u_\epsilon^p)\,dx\\
 &=-(\lambda_{i\epsilon}\epsilon^{\alpha_0})^{\tau_0+1}
      \int_{B_{i\epsilon}}(\partial_k\chi_{i\epsilon})u_\epsilon^p\,dx
 =O(\epsilon^{\alpha_0(N+1)}).
 \end{aligned}
\end{equation}
Therefore, by \eqref{l21:zeta}, we have
\[
 \mu_{l\epsilon} p\int_{B_{i\epsilon}}
       u_\epsilon^{p-1}\zeta_\epsilon\,dx
 =k_{i\epsilon}+O(\epsilon^{\alpha_0(N+1)}).
\]
Combining the following eatimates
\[
 \frac{|k_{i\epsilon}|}
      {(\lambda_{i\epsilon}\epsilon^{\alpha_0})^{N-2} n_\epsilon}
 \le C\frac{\mathfrak h_\epsilon}{n_\epsilon}\longrightarrow0,
 \qquad
 \frac{\epsilon^{\alpha_0(N+1)}}
      {(\lambda_{i\epsilon}\epsilon^{\alpha_0})^{N-2} n_\epsilon}
 \le C(\epsilon^{\alpha_0})^2
        \frac{\epsilon^{\alpha_0}}{n_\epsilon}\longrightarrow0
\]
we eventually get that
\begin{equation}\label{l21:normalized-mass}
 \mu_{l\epsilon} p
 \int_{B(0,R_0/(\lambda_{i\epsilon}\epsilon^{\alpha_0}))}
       U_{i\epsilon}^{p-1}Z_{i\epsilon}\,dz
 \longrightarrow0.
\end{equation}
We justify passage to the limit on this expanding ball. For
$|z|<R_0/(\lambda_{i\epsilon}\epsilon^{\alpha_0})$, we have
\[
 |U_{i\epsilon}(z)|+|Z_{i\epsilon}(z)|
       \le C\bigl((1+|z|)^{-(N-2)}+\epsilon^{\alpha_0(N-2)}\bigr).
\]
Hence
\[
 |U_{i\epsilon}^{p-1}Z_{i\epsilon}|
 \le C(1+|z|)^{-(N+2)}\in L^1(\mathbb R^N).
\]
By dominated convergence theorem, \eqref{l21:kernel} and
\eqref{l21:normalized-mass}, we obtain
\begin{equation}\label{l21:mass-limit}
 0=p\int_{\mathbb R^N}W^{p-1}Z_i\,dz
   =b_i p\int_{\mathbb R^N}W^{p-1}Z_0\,dz.
\end{equation}
Moreover, by the integration by parts,
\begin{equation}\label{l21:dilation-moment}
 \begin{aligned}
 p\int_{\mathbb R^N}W^{p-1}Z_0\,dz
 &=-\int_{\mathbb R^N}z\cdot\nabla(W^p)\,dz
        -\frac{N+2}{2}\int_{\mathbb R^N}W^p\,dz=\tau_0\mathcal{A}_0>0.
 \end{aligned}
\end{equation}
The boundary term in this calculation is
$R\int_{\partial B_R}W^p\,dS=O(R^{-2})$, so it vanishes.
Thus $b_i=0$ for every $i$.

Local convergence is only controlled on a fixed scaling ball, while the region far away from each bubble remains uncontrolled. Thus we define
\[
 \alpha_\epsilon(R)
 :=\sup_{x\in\bigcup_i B(\xi_{i\epsilon},R(\epsilon^{\alpha_0}))}
                  \frac{|Z_\epsilon(x)|}{\mathfrak w_\epsilon(x)}.
\]
On this region, $\mathfrak w_\epsilon\ge(1+R)^{-(N-2)}$.
Scale comparability and local convergence therefore give
$\alpha_\epsilon(R)\to0$ for fixed $R$.
By Green's representation formula, we have
\[
 Z_\epsilon=\mathcal T(pu_\epsilon^{p-1}Z_\epsilon)
      +\mathcal T(\epsilon q u_\epsilon^{q-1}Z_\epsilon)
      +\mathcal T(F_\epsilon/n_\epsilon).
\]
Using positivity of $G$, $|Z_\epsilon|\le\mathfrak w_\epsilon$,
\eqref{l21:green1}--\eqref{l21:green3}, and \eqref{l21:small-ratios}, we obtain
\[
 1=\|Z_\epsilon\|_*
 \le C\alpha_\epsilon(R)+CR^{-2}+C(\epsilon^{\alpha_0})^{\min\{2\tau_0,2\tau_0q-2\}}
                (1+|\log(\epsilon^{\alpha_0})|)+o(1).
\]
Letting $\epsilon\to0$ with $R$ fixed gives $1\le CR^{-2}$.
Taking $R$ sufficiently large, we clearly get a contradiction, which concludes the proof of
\eqref{l21:estimate}.

Finally, choose $K'$ with
$K\subset\operatorname{int}K'\subset\Omega\setminus\{\xi_1,\ldots,\xi_m\}$.
On $K'$, $\mathfrak w_\epsilon\le C_{K'}\epsilon^{\alpha_0(N-2)}$ and
$\psi_{ik\epsilon}=O(\epsilon^{\alpha_0(N-1)})$.
Thus \eqref{l21:zeta}, \eqref{l21:estimate}, and
$\mathfrak h_\epsilon\ge(\epsilon^{\alpha_0})$ imply
\[
 \|v_\epsilon\|_{L^\infty(K')}
 \le C_{K'}\epsilon^{\alpha_0(N-2)}\mathfrak h_\epsilon.
\]
Also $(pu_{\epsilon}^{p-1}+\epsilon qu_{\epsilon}^{q-1})=O(\epsilon^{2\alpha_0}+\epsilon^{\alpha_0((N-2)q-2)})$ on
$K'$, and $\mu_\epsilon$ is bounded. An interior $W^{2,t}$ estimate
with fixed $t>N$, followed by Sobolev embedding, gives
\[
 \|v_{l\epsilon}\|_{C^1(K)}
 \le C_K\left(\|v_{l\epsilon}\|_{L^\infty(K')}
               +\|\mu_{l\epsilon}(pu_{\epsilon}^{p-1}+\epsilon qu_{\epsilon}^{q-1})v_{l\epsilon}\|_{L^\infty(K')}
        \right)
 \le C_K\epsilon^{\alpha_0(N-2)}\mathfrak h_\epsilon,
\]
as desired.
\end{proof}

\begin{cor}\label{lem:moment}
We have
\begin{equation}\label{eq:moment-match}
\displaystyle
 \epsilon\int_{B_{i\epsilon}}u_\epsilon^qv_{l\epsilon}\,dx
 =\epsilon^{\alpha_0\tau_0}\frac{\mathcal{B}_0(p-q)}{\mathcal{A}_0}\lambda_i^{\frac{(N-2)(p-q-1)}{2}}k_{i\epsilon}
       +o(\epsilon^{3\alpha_0\tau_0}\mathfrak h_\epsilon).
\end{equation}
\end{cor}
\begin{proof}
By Lemma~\ref{l21:lemma}, $\|\zeta_\epsilon/\mathfrak h_\epsilon\|\leq C$. Thus combined with $\mathfrak h_\epsilon\ge\epsilon^{\alpha_0}$, nondegeneracy and orthogonality condition \eqref{l21:orthog}, we have
\[
 \frac{\zeta_\epsilon(\xi_{i\epsilon}+\delta_{i\epsilon}z)}{\mathfrak h_\epsilon}
 \longrightarrow b_i Z_0(z)
\]
By dominated convergence theorem,
\begin{equation}\label{eq:mass-normalized}
 \frac{k_{i\epsilon}}{\epsilon^{\alpha_0(N-2)}\mathfrak h_\epsilon}\longrightarrow
 \lambda_i^{N-2} \tau_0\mathcal{A}_0 b_i.
\end{equation}
Furthermore, we compute
\begin{align}
 \frac{\epsilon\int_{B_{i\epsilon}}u_\epsilon^q\zeta_\epsilon\,dx}
      {\epsilon^{3\alpha_0\tau_0}\mathfrak h_\epsilon}
 &=\lambda_{i\epsilon}^{\tau_0(p+1-q)}
   \int_{B(0,R_0/\delta_{i\epsilon})}U_{i\epsilon}^q
   \frac{\zeta_\epsilon(\xi_{i\epsilon}+\lambda_{i\epsilon}\epsilon^{\alpha_0}z)}{\mathfrak h_\epsilon}\,dz
 \longrightarrow\lambda_i^{\tau_0(p+1-q)}\tau_0(p-q)\mathcal{B}_0 b_i.
 \label{eq:q-normalized}
\end{align}
The removed translations have a smaller contribution, with no approximation to $W_\epsilon$ needed:
\begin{align}
 \epsilon\int_{B_{i\epsilon}}u_\epsilon^q\psi_{ik\epsilon}\,dx
 &=-\frac{\epsilon(\lambda_{i\epsilon}\epsilon^{\alpha_0})^{\tau_0+1}}{q+1}
     \int_{B_{i\epsilon}}(\partial_k\chi_i)u_\epsilon^{q+1}\,dx\nonumber\\
 &=O(\epsilon^{\alpha_0(3\tau_0+1+2\tau_0q-2)})
 =o(\epsilon^{3\alpha_0\tau_0}\mathfrak h_\epsilon).\label{eq:q-translation}
\end{align}
Here the exponent follows from $N-2-\tau_{0}(p-q)+\tau_0(q+2)+1=3\tau_0+1+2\tau_0q-2$, and the last estimate uses $\mathfrak h_\epsilon\ge\epsilon^{\alpha_0}$ and $2\tau_0q>2$.
Finally \eqref{eq:moment-match} follows by \eqref{eq:mass-normalized}, \eqref{eq:q-normalized} and \eqref{eq:q-translation}.
\end{proof}

\begin{cor}\label{lem:q-potential}
For every exterior compact set $K$, it holds that
\begin{equation}\label{eq:q-potential}
 \|\mathcal{T}(\mu_{l\epsilon}\epsilon q u_\epsilon^{q-1}v_{l\epsilon})\|_{C^1(K)}
 =o(\epsilon^{2\alpha_0\tau_0}\mathfrak h_\epsilon).
\end{equation}
\end{cor}
\begin{proof}
Fix an open set $V$ such that
\[
v_{l\epsilon}=\zeta_\epsilon
+\sum_{i,k}c_{ik\epsilon}\psi_{ik\epsilon},
\]
we estimate the remainder and translation terms separately.
Using \eqref{l21:estimate} and \eqref{l21:green3}, we obtain
\[
\left|\mathcal T(b_\epsilon\zeta_\epsilon)\right|
\le C(\epsilon^{\alpha_0})^{\min\{2\tau_0,2\tau_0q-2\}}
                (1+|\log(\epsilon^{\alpha_0})|)\mathfrak h_\epsilon
\mathfrak w_\epsilon,\quad\mbox{where}\quad b_{\epsilon}:=\mu_{l\epsilon}\epsilon q u_\epsilon^{q-1}.
\]
On $V$, we have $\mathfrak w_\epsilon=O(\epsilon^{\alpha_0(N-2)})$,
$b_\epsilon=O(\epsilon^{\alpha_0(2\tau_0q-2)})$, and
$|\zeta_\epsilon|\le C\mathfrak h_\epsilon
\mathfrak w_\epsilon$. Interior elliptic estimates therefore yield
\begin{equation}\label{cor23simple:remainder}
\begin{split}
\left\|\mathcal T(b_\epsilon\zeta_\epsilon)\right\|_{C^1(K)}
&\le C_K\epsilon^{\alpha_0(N-2)}\mathfrak h_\epsilon
\left( C(\epsilon^{\alpha_0})^{\min\{2\tau_0,2\tau_0q-2\}}
                (1+|\log(\epsilon^{\alpha_0})|)+\epsilon^{\alpha_0(2\tau_0q-2)}\right)
\\&=o\left(\epsilon^{\alpha_0(N-2)}\mathfrak h_\epsilon\right).
\end{split}
\end{equation}
For the translation terms, the solution envelope and rescaling give
\begin{equation}\label{cor23simple:mass}
\begin{aligned}
\epsilon\int_\Omega u_\epsilon^q\,dy
&\le C\epsilon^{3\alpha_0\tau_0}
\int_0^{C/(\epsilon^{\alpha_0})}
t^{N-1}(1+t)^{-(N-2)q}\,dt\\
&\le C(\epsilon^{\alpha_0})^{\tau_0+\min\{N-2,2\tau_0q-2\}}
\left(1+|\log(\epsilon^{\alpha_0})|\right)
=o(\epsilon^{\alpha_0\tau_0}).
\end{aligned}
\end{equation}
Set $P_{ik\epsilon}:=\mathcal T(b_\epsilon\psi_{ik\epsilon})$.
Since $\psi_{ik\epsilon}=(\lambda_{i\epsilon}\epsilon^{\alpha_0})^{\tau_0+1}
\chi_{i\epsilon}\partial_k u_\epsilon$, integration by parts yields
\begin{equation}\label{cor23simple:ibp}
\begin{aligned}
P_{ik\epsilon}(x)
=&-\mu_{l\epsilon}\epsilon(\lambda_{i\epsilon}\epsilon^{\alpha_0})^{\tau_0+1}
\int_\Omega
\bigl[\chi_{i\epsilon}(y)\partial_{y_k}G(x,y)
+G(x,y)\partial_k\chi_{i\epsilon}(y)\bigr]
u_\epsilon^q(y)\,dy.
\end{aligned}
\end{equation}
This identity remains valid when $x$ belongs to the cutoff support:
$|\nabla_yG(x,y)|\le C|x-y|^{1-N}$ is locally integrable, and the
boundary term obtained by removing a ball of radius $\eta$ about $x$
is $O_\epsilon(\eta)$ and tends to zero as $\eta\to0$.
Choose $d>0$, independent of small $\epsilon$, such that
$B(\xi_{i\epsilon},d)$ is separated from $V$.
On this ball, the kernels in \eqref{cor23simple:ibp} are uniformly
bounded for $x\in V$. On the remaining cutoff support, the exterior
estimate implies
\[
\epsilon u_\epsilon^q
=O((\epsilon^{\alpha_0})^{2\tau_0-\tau_0(p-q)+\tau_0q})
=O(\epsilon^{\alpha_0(\tau_0+2\tau_0q-2)}).
\]
Moreover,
\[
\sup_{x\in V}
\int_{\operatorname{supp}\chi_{i\epsilon}}
\bigl(G(x,y)+|\nabla_yG(x,y)|\bigr)\,dy
\le C_V.
\]
Splitting the integral in \eqref{cor23simple:ibp} over these two
regions, and using \eqref{cor23simple:mass} and
$\lambda_{i\epsilon}\epsilon^{\alpha_0}\approx\epsilon^{\alpha_0}$, we obtain
\begin{equation}\label{cor23simple:potential}
\|P_{ik\epsilon}\|_{L^\infty(V)}
\le C_V\epsilon^{\alpha_0(\tau_0+1)}
\left(M_\epsilon+\epsilon^{\alpha_0(\tau_0+2\tau_0q-2)}\right)
=o(\epsilon^{\alpha_0(N-1)}).
\end{equation}
On the other hand, the exterior bounds for $u_\epsilon$ and its
gradient give
\[
\|b_\epsilon\psi_{ik\epsilon}\|_{L^\infty(V)}
\le C_V(\epsilon^{\alpha_0(N-1+2\tau_0q-2)}).
\]
Applying interior $W^{2,t}$ estimates with $t>N$ to
$-\Delta P_{ik\epsilon}=b_\epsilon\psi_{ik\epsilon}$,
and then Sobolev embedding, yields
\begin{equation}\label{cor23simple:translation}
\|P_{ik\epsilon}\|_{C^1(K)}
=o(\epsilon^{\alpha_0(N-1)}).
\end{equation}
Finally, the coefficients $c_{ik\epsilon}$ are bounded and
$\mathfrak h_\epsilon\ge\epsilon^{\alpha_0}$.
Summing \eqref{cor23simple:translation} over the finitely many
translation terms and using \eqref{cor23simple:remainder} proves
\eqref{eq:q-potential}.
\end{proof}

Fix the small radius $R_0$, and $i$, and let $B_i=B(\xi_{i\epsilon},R_0)$ be the ball in the sequel, we introduce
\begin{equation}\label{eq:kappas}
 \begin{aligned}
 \kappa_{i0}&=\int_{B_{i}}(u_\epsilon^p+\epsilon u_\epsilon^q)\,dy,\quad\kappa_{i1}=\int_{B_{i}}u_\epsilon^{p-1}v_{l\epsilon}\,dy,\\
 \kappa_{i2}&=\int_{B_{i}}(y-\xi_{i\epsilon})u_\epsilon^{p-1}v_{l\epsilon}\,dy
       =(\kappa_{i21},\ldots,\kappa_{i2N}).
 \end{aligned}
\end{equation}
Recalling \eqref{l21:masses}, we have
\begin{equation}\label{eq:mass-conversion}
 k_{i\epsilon}=\mu_{l\epsilon} p\kappa_{i1},\qquad
 \mathfrak h_\epsilon=\epsilon^{\alpha_0}+
 \frac{\mu_{l\epsilon} p\max_i|\kappa_{i1}|}{\epsilon^{\alpha_0(N-2)}}.
\end{equation}
Thus $\mathfrak h_\epsilon$ is a temporary scalar error size, it is different from the harmonic function $h_{i\epsilon}$ introduced in \eqref{eq:regular-parts}.

\begin{lem}\label{lem:green}
On every compact subset of $\Omega\setminus\{\xi_1,\ldots,\xi_m\}$, we have
\begin{align}
 u_\epsilon(x)&=\sum_{i=1}^m\kappa_{i0}G(x,\xi_{i\epsilon})
                      +o(\epsilon^{\alpha_0\tau_0}),\label{eq:W-kappa}\\
 \frac{v_{l\epsilon}(x)}{p\mu_{l\epsilon}}
 &=\sum_{i=1}^m\left[\kappa_{i1}G(x,\xi_{i\epsilon})
           +\kappa_{i2}\cdot\nabla_yG(x,\xi_{i\epsilon})\right]
                      +o(\epsilon^{\alpha_0(N-2)}\mathfrak h_\epsilon).
 \label{eq:v-kappa}
\end{align}
Moreover, the coefficients satisfy
\begin{equation}\label{eq:kappa-orders}
 \kappa_{i0}=\mathcal{A}_0(\epsilon^{\alpha_0\tau_0})\lambda_i^{\tau_0}+o(\epsilon^{\alpha_0\tau_0}),\qquad
 \kappa_{i1}=o(\epsilon^{\alpha_0(N-2)}),\qquad
 \kappa_{i2}=O(\epsilon^{\alpha_0(N-1)}).
\end{equation}
\end{lem}
\begin{proof}
Proof of \eqref{eq:W-kappa}, $\kappa_{i1}$ and $\kappa_{i2}$ is similar to Lemma 5.3 in \cite{CKL-1}. So we omit it.
By the change of variables $y=\xi_{i\epsilon}+\lambda_{i\epsilon}\epsilon^{\alpha_0}z$. Since $N-\tau_0p=\tau_0$,
\[
 \epsilon^{-\alpha_0\tau_0}\int_{B_{i\epsilon}^N}u_\epsilon^p\,dy
 =\lambda_{i\epsilon}^{\tau_0}\int_{B(0,R_0/(\lambda_{i\epsilon}\epsilon^{\alpha_0}))}U_{i\epsilon}^p\,dz
 \longrightarrow\lambda_i^{\tau_0}\mathcal{A}_0.
\]
In virtue of \eqref{cor23simple:mass} and $\epsilon\int_{B_{i\epsilon}^{N}}u_\epsilon^q=o(\epsilon^{\alpha_0\tau_0})$, we obtain the estimate of $\kappa_{i0}$.

We next prove \eqref{eq:v-kappa}. Let $K$ be an exterior compact set and first choose the integration balls disjoint from a fixed neighborhood of $K$. By the Green's representation, we have
\[
 \frac{v_{l\epsilon}(x)}{p\mu_{l\epsilon}}
 =\int_\Omega G(x,y)u_\epsilon^{p-1}v_\epsilon(y)\,dy
   +\frac\epsilon p\int_\Omega G(x,y)q u_\epsilon^{q-1}v_\epsilon(y)\,dy.
\]
According to \eqref{eq:q-potential}, we have that the second term is bounded by $o_{C^1(K)}(\epsilon^{\alpha_0(N-2)}\mathfrak h_\epsilon)$. Outside the fixed balls, the profile and remainder estimates give
$
 |u_\epsilon^{p-1}v_{l\epsilon}|\le C\epsilon^{\alpha_0N}\mathfrak h_\epsilon.
$
The contribution of this bounded source and its first $x$ derivatives is $O(\epsilon^{\alpha_0N}\mathfrak h_\epsilon)$ on $K$, by the Green bound and local elliptic estimates.
Inside a ball, Taylor's formula in the source variable gives
\[
 G(x,y)=G(x,\xi_{i\epsilon})
 +(y-\xi_{i\epsilon})\cdot\nabla_yG(x,\xi_{i\epsilon})+E_i(x,y),
\]
where, uniformly for $x\in K$,
$|E_i|+|\nabla_xE_i|\le C_K|y-\xi_{i\epsilon}|^2$. Indeed, the second $y$ derivatives of $G$ and their first $x$ derivatives are bounded on these separated sets. The resulting remainder is bounded by
\begin{align}
 \int_{B_{i\epsilon}}|y-\xi_{i\epsilon}|^2u_\epsilon^{p-1}|v_{l\epsilon}|\,dy
 &\le C(\lambda_{i\epsilon}\epsilon^{\alpha_0})^{N}
 \int_0^{R_0/(\lambda_{i\epsilon}\epsilon^{\alpha_0})}
 \frac{t^{N+1}}{(1+t)^{N+2}}\,dt\nonumber\\
 &\le C\epsilon^{\alpha_0N}(1+|\log(\epsilon^{\alpha_0})|)
 =o(\epsilon^{\alpha_0(N-1)}).\label{eq:taylor-error}
\end{align}
Since $\mathfrak h_\epsilon\ge\epsilon^{\alpha_0}$, this is also $o(\epsilon^{\alpha_0(N-2)}\mathfrak h_\epsilon)$. Hence we prove \eqref{eq:v-kappa}.

\end{proof}

Write $\Gamma_i(x)=\gamma_N|x-\xi_{i\epsilon}|^{-(N-2)}$. Introducing
\begin{equation}\label{eq:regular-parts}
 \begin{aligned}
 g_{i\epsilon}(x)&=-\kappa_{i0}H(x,\xi_{i\epsilon})
             +\sum_{j\ne i}\kappa_{j0}G(x,\xi_{j\epsilon}),\\
 h_{i\epsilon}(x)&=-\kappa_{i1}H(x,\xi_{i\epsilon})
              -\kappa_{i2}\cdot\nabla_yH(x,\xi_{i\epsilon})\\
 &\quad+\sum_{j\ne i}\left[\kappa_{j1}G(x,\xi_{j\epsilon})
                   +\kappa_{j2}\cdot\nabla_yG(x,\xi_{j\epsilon})\right].
 \end{aligned}
\end{equation}
Both functions are harmonic near $\xi_{i\epsilon}$. Their local expansions are
\begin{align}
 u_\epsilon(x)&=\kappa_{i0}\Gamma_i(x)+g_{i\epsilon}(x)+o_{C^1}(\epsilon^{\alpha_0\tau_0}),\label{eq:local-W}\\
 \frac{v_{l\epsilon(x)}}{p\mu_{l\epsilon}}
 &=\kappa_{i1}\Gamma_i(x)
   +(N-2)\gamma_N\frac{\kappa_{i2}\cdot(x-\xi_{i\epsilon})}{|x-\xi_{i\epsilon}|^N}
   +h_{i\epsilon}(x)+o(\epsilon^{\alpha_0(N-2)}\mathfrak h_\epsilon).
 \label{eq:local-v}
\end{align}
These estimates are on any fixed annulus about the center. We will apply them on a fixed sphere. Only the harmonic comparison functions, not their PDE remainders, will subsequently be evaluated at a radius tending to zero.
Introducing the nonzero column vectors  $\kappa_1=(\kappa_{11},\ldots,\kappa_{m1})^T$ and $\kappa_0=(\kappa_{10},\ldots,\kappa_{m0})^T$. Let
\begin{equation}\label{eq:R-kappa}
 \mathcal R_{i\epsilon}(\kappa_2)
 =-\frac12\kappa_{i2}\cdot\nabla\tau(\xi_{i\epsilon})
       +\sum_{j\ne i}\kappa_{j2}\cdot
                  \nabla_yG(\xi_{i\epsilon},\xi_{j\epsilon}),
\end{equation}
and $\mathcal R_\epsilon=(\mathcal R_{1\epsilon},\ldots,\mathcal R_{m\epsilon})^T$.

\begin{lem}\label{lem:coefficient}
Recall the definition of $N_q$ in \eqref{Nq}, we have
\begin{equation}\label{eq:kappa-sharp}
 \kappa_1=O(\epsilon^{\alpha_0(N-1)}),\qquad
 \mathfrak h_\epsilon=O(\epsilon^{\alpha_0}),\qquad
 \kappa_1=N_q^{-1}\mathcal R_\epsilon(\kappa_2)+o(\epsilon^{\alpha_0(N-1)}).
\end{equation}
In particular, if $(\eta^{ij})=N_q^{-1}$, then it holds for $m+1\leq l\leq(n+1)m$ that
\begin{align}
 \epsilon^{-\alpha_0(N-1)}\kappa_{i1}
 =\sum_{j=1}^m\eta^{ij}\bigg[
 -\frac12\epsilon^{-\alpha_0(N-1)}\kappa_{j2}\cdot\nabla\phi(\xi_j)
 +\sum_{l\ne j}\epsilon^{-\alpha_0(N-1)}\kappa_{l2}\cdot
                      \nabla_yG(\xi_j,\xi_l)\bigg]+o(1).
 \label{eq:CKL59}
\end{align}
\end{lem}
\begin{proof}
Fix $R_0$, we define
\[
 \bar{X}=x-\xi_{i\epsilon},\qquad
 \mu=\mu_{\ell\epsilon},\qquad
 \overline v_\epsilon=\frac{v_{\ell\epsilon}}{\mu_{l\epsilon} p},\qquad
 f_\epsilon(t)=t^p+\epsilon t^q,
\]
so that $-\Delta u_\epsilon=f_\epsilon(u_\epsilon)$ and
$-\Delta\overline v_\epsilon
=\mu f_\epsilon'(u_\epsilon)\overline v_\epsilon$.
We use the symmetric boundary form
\begin{equation}\label{massrev:boundary-form}
 \begin{aligned}
 \mathcal P_i(\phi,\psi)
 :={}&R_0\int_{\partial B_i}
       \bigl(\nabla\phi\cdot\nabla\psi
             -2\partial_\nu\phi\,\partial_\nu\psi\bigr)\,dS-\tau_0\int_{\partial B_i}
       \bigl(\partial_\nu\phi\,\psi
                      +\phi\,\partial_\nu\psi\bigr)\,dS.
 \end{aligned}
\end{equation}
Applying Lemma 5.4 in \cite{CKL-1} and Green's identity, we obtain
\begin{equation}\label{massrev:volume}
 \begin{aligned}
 \mathcal P_i(u_\epsilon,\overline v_\epsilon)
 =&-\tau_0(p-q)J_{i\epsilon}
      +(\mu-1)D_{i\epsilon}
      +R_0\int_{\partial B_i}
                   f_\epsilon(u_\epsilon)\overline v_\epsilon\,dS,\\
 J_{i\epsilon}:=&\epsilon\int_{B_i}
                   u_\epsilon^q\overline v_\epsilon\,dx,\qquad
 D_{i\epsilon}:=\int_{B_i}f_\epsilon'(u_\epsilon)
          (\bar{X}\cdot\nabla u_\epsilon+\tau_0 u_\epsilon)
                            \overline v_\epsilon\,dx.
 \end{aligned}
\end{equation}
Indeed, the nonlinear terms combine into
$[\tau_0 u_\epsilon f_\epsilon'(u_\epsilon)
 +(\tau_0-N)f_\epsilon(u_\epsilon)]\overline v_\epsilon
=-\tau_0(p-q)\epsilon u_\epsilon^q\overline v_\epsilon$.

The exterior estimates on the fixed sphere imply
\[
 R_0\int_{\partial B_i}
           f_\epsilon(u_\epsilon)\overline v_\epsilon\,dS
 =O\!\left((\epsilon^{\alpha_0})^{3\tau_0+\min\{2,\kappa\}}
                          \mathfrak h_\epsilon\right)
 =o(\epsilon^{3\alpha_0\tau_0}\mathfrak h_\epsilon).
\]
For the critical part of $D_{i\epsilon}$, by Theorem \ref{existence-1} and oddness, we have
\[
 \epsilon^{-\alpha_0\tau_0}
 \int_{B_i}p u_\epsilon^{p-1}
       (\bar{X}\cdot\nabla u_\epsilon+\frac{N-2}{2}u_\epsilon)
                         \overline v_\epsilon\,dx
 \longrightarrow
 -\lambda_i^{\frac{N-2}{2}}\sum_{k=1}^N d_{l,i,k}
 \int_{\mathbb R^N}W^{p-1}(z\cdot\nabla W+\frac{N-2}{2}W)\partial_kW\,dz
 =0.
\]
The profile bounds supply
the integrable majorant $C(1+|z|)^{-2N}$. The perturbative part,
after multiplication by $(\epsilon^{-\alpha_0\tau_0})$, is
$O(\epsilon^{\alpha_0(N-2)})$, with majorant
$C(1+|z|)^{-(N-2)(q+1)}\in L^1(\mathbb R^N)$.
Thus $D_{i\epsilon}=o(\epsilon^{\alpha_0\tau_0})$.
By the estimate \eqref{cor1}
and $\mathfrak h_\epsilon\ge\epsilon^{\alpha_0}$, we have
\[
 (\mu-1)D_{i\epsilon}
 =o(\epsilon^{\alpha_0}(3\tau_0+1))
 =o(\epsilon^{3\alpha_0\tau_0}\mathfrak h_\epsilon).
\]
Thus, we can compute \eqref{massrev:volume} as
\begin{equation}\label{massrev:reduced-volume}
 \mathcal P_i(u_\epsilon,\overline v_\epsilon)
 =-\tau_0(p-q)J_{i\epsilon}
                   +o(\epsilon^{3\alpha_0\tau_0}\mathfrak h_\epsilon).
\end{equation}

We next evaluate this boundary form using
\eqref{eq:local-W}--\eqref{eq:local-v}. Taking
\[
 I_i^0=\kappa_{i0}\Gamma_i+g_{i\epsilon},\qquad
 v_i^0=\kappa_{i1}\Gamma_i
       +\kappa_{i2}\cdot\nabla_{\xi_{i\epsilon}}\Gamma_i
       +h_{i\epsilon}.
\]
On $\partial B_i^N$, these functions have respective $C^1$ bounds
$O(\epsilon^{\alpha_0\tau_0})$ and
$O(\epsilon^{\alpha_0(N-2)})\mathfrak h_\epsilon)$. Therefore,
\begin{equation}\label{massrev:harmonic-replacement}
 \mathcal P_i(u_\epsilon,\overline v_\epsilon)
 =\mathcal P_i(I_i^0,v_i^0)
                   +o(\epsilon^{3\alpha_0\tau_0}\mathfrak h_\epsilon).
\end{equation}
For harmonic regular parts $g,h$, the following identity holds:
\begin{equation}\label{massrev:harmonic-identity}
 \mathcal P_i(C\Gamma_i+g,
                  K\Gamma_i+b\cdot\nabla_{\xi_{i\epsilon}}\Gamma_i+h)
 =a\bigl(Kg(\xi_{i\epsilon})+Ch(\xi_{i\epsilon})\bigr)
   +(a+1)b\cdot\nabla g(\xi_{i\epsilon}).
\end{equation}
To verify this, we define $\mathcal P_r(\phi,\psi)$ to be the right-hand side of \eqref{massrev:volume}, where $R_0$ is replaced by $r$.
Then we have the bilinear Pohozaev identity (5.8) of Lemma 5.4 in \cite{CKL-1} on annuli
\begin{equation}\label{massrev:annular-identity}
\begin{aligned}
\mathcal P_R(\phi,\psi)-\mathcal P_r(\phi,\psi)
=-\int_{B(\xi,R)\setminus\overline{B(\xi,r)}}
\Big[
 ((x-\xi)\cdot\nabla\phi+a\phi)\Delta\psi
+((x-\xi)\cdot\nabla\psi+a\psi)\Delta\phi
\Big]\,dx.
\end{aligned}
\end{equation}
Using this identity \eqref{massrev:annular-identity}, we can show that
the left-hand side is independent of the sphere radius.
Evaluate it as that radius tends to zero, using the Taylor expansions
of $g,h$, the normalization $(N-2)\gamma_N|\mathbb S^{N-1}|=1$, and
$\int_{\mathbb S^{N-1}}\nu_j\nu_k\,dS
=|\mathbb S^{N-1}|\delta_{jk}/N$.
The monopole--regular pairings contribute $aKg(\xi_{i\epsilon})$
and $aCh(\xi_{i\epsilon})$, while the dipole--regular pairing
contributes $(a+1)b\cdot\nabla g(\xi_{i\epsilon})$.
The monopole--monopole terms cancel, the monopole--dipole terms
vanish by oddness, and the regular--regular contribution vanishes
by the bilinear identity.
This auxiliary radius limit concerns only the harmonic comparison
functions; the boundary replacement in
\eqref{massrev:harmonic-replacement} is made at the fixed radius $R_0$.

Thanks to \eqref{eq:regular-parts},
\begin{equation*}
 \nabla g_{i\epsilon}(\xi_{i\epsilon})=-\kappa_{i0}\nabla_{x}H(\xi_{i\epsilon},\xi_{i\epsilon})
             +\sum_{j\ne i}\kappa_{j0}\nabla_{x}G(\xi_{i\epsilon},\xi_{j\epsilon}),
\end{equation*}
 combined with $\nabla\phi=2\nabla_xH(x,x)$ and \eqref{eq:kappa-orders},
we obtain
\[
 \epsilon^{-\alpha_0\tau_0}\nabla g_{i\epsilon}(\xi_{i\epsilon})
 \longrightarrow \mathcal{A}_0\left[-\frac{t_i}{2}\nabla\phi(\xi_i)
       +\sum_{j\ne i}t_j\nabla_xG(\xi_i,\xi_j)\right]=0.
\]
Since $\kappa_{i2}=O(\epsilon^{\alpha_0(N-1)})$, the dipole term in
\eqref{massrev:harmonic-identity} is
$o(\epsilon^{\alpha_0(3\tau_0+1)})
=o(\epsilon^{3\alpha_0\tau_0}\mathfrak h_\epsilon)$.
Combining \eqref{massrev:reduced-volume}, \eqref{massrev:harmonic-replacement} and \eqref{massrev:harmonic-identity},
we get
\begin{equation}\label{massrev:regular-relation}
 \kappa_{i1}g_{i\epsilon}(\xi_{i\epsilon})
 +\kappa_{i0}h_{i\epsilon}(\xi_{i\epsilon})
 =-(p-q)J_{i\epsilon}
                   +o(\epsilon^{3\alpha_0\tau_0}\mathfrak h_\epsilon).
\end{equation}
Recalling Corollary~\ref{lem:moment},  \eqref{eq:moment-match} divided by $\mu p$, gives
\begin{equation}\label{massrev:signed-moment}
 J_{i\epsilon}
 =\epsilon^{\alpha_0\tau_0}\frac{(p-q)\mathcal{B}_0}{\mathcal{A}_0}
                  \lambda_i^{\tau_0(p-q-1)}\kappa_{i1}
                   +o(\epsilon^{3\alpha_0\tau_0}\mathfrak h_\epsilon).
\end{equation}
Let $M_\epsilon$ denote the Green-Robin matrix at the
current centers, i.e
$$
(M_{\epsilon})_{ii}=\phi(\xi_{i\epsilon}),\hspace{3mm} (M_{\epsilon})_{ij}=-G(\xi_{i\epsilon},\xi_{j\epsilon})\hspace{2mm}\mbox{if}\hspace{2mm}i\neq j.
$$
Then the regular parts satisfy
\[
 g_{i\epsilon}(\xi_{i\epsilon})
     =-(M_\epsilon\kappa_0)_i,
 \qquad
 h_{i\epsilon}(\xi_{i\epsilon})
     =-( M_\epsilon\kappa_1)_i
                         +\mathcal R_{i\epsilon}(\kappa_2).
\]
Substituting these identities and \eqref{massrev:signed-moment}
into \eqref{massrev:regular-relation}, we deduce that
\begin{equation*}
 (M_{\epsilon}\kappa_1)_i
 +\frac{(M_{\epsilon}\kappa_0)_i}{\kappa_{i0}}
       \kappa_{i1}
 =\mathcal R_{i\epsilon}(\kappa_2)+(p-q)\frac{J_{i\epsilon}}{\kappa_{i0}}
                       +o(\epsilon^{\alpha_0(N-2)}\mathfrak h_\epsilon).
\end{equation*}
Combining this bound and the following two facts
$$
\frac{(M_{\epsilon}\kappa_0)_i}{\kappa_{i0}}\rightarrow\frac{(Mt)_i}{t_{i}}
$$
and
$$
(p-q)\frac{J_{i\epsilon}}{\kappa_{i0}}
 =\frac{\mathcal{B}_0(p-q)^2}{\mathcal{A}_0^2}
                  \lambda_i^{\tau_{0}(p-q)-(N-2)}\kappa_{i1}
                   +o(\epsilon^{\alpha_0(N-2)}\mathfrak h_\epsilon),
$$
so that we eventually get that
\begin{equation}\label{massrev:prematrix}
 (M\kappa_1)_i
 +\left[\frac{(Mt)_i}{t_i}
       -\frac{\mathcal{B}_0(p-q)^2}{\mathcal{A}_0^2}\lambda_i^{\tau_0(p-q)-(N-2)}\right]\kappa_{i1}
 =\mathcal R_{i\epsilon}(\kappa_2)
                       +o(\epsilon^{\alpha_0(N-2)}\mathfrak h_\epsilon).
\end{equation}
Here converging coefficients were replaced by their limits using
$\|\kappa_1\|_\infty=O(\epsilon^{\alpha_0(N-2)}\mathfrak h_\epsilon)$,
which follows directly from \eqref{eq:mass-conversion}.
The identity \eqref{deference} gives
\[
 \frac{(Mt)_i}{t_i}
 =\frac{\mathcal{B}_0(p-q)}{\mathcal{A}_0^2}\lambda_i^{\tau_0(p-q)-(N-2)}.
\]
Thus \eqref{massrev:prematrix} together with \eqref{l2matrix} imply that
\begin{equation}\label{expanded:masssystem}
N_q\kappa_1
 =\mathcal R_\epsilon(\kappa_2)
                       +o(\epsilon^{\alpha_0(N-2)}\mathfrak h_\epsilon).
\end{equation}
Finally, $N_q$ is positive definite by
Lemma \ref{positivede}, and
$\mathcal R_\epsilon(\kappa_2)=O(\epsilon^{\alpha_0(N-1)})$
by \eqref{eq:kappa-orders}. Hence \eqref{eq:mass-conversion} and \eqref{expanded:masssystem}
imply that
\[
 \|\kappa_1\|_\infty
 \le C\epsilon^{\alpha_0}(N-1)
       +o(1)\left(\epsilon^{\alpha_0(N-1)}+\|\kappa_1\|_\infty\right).
\]
Absorbing the last term gives
$\kappa_1=O(\epsilon^{\alpha_0(N-1)})$, and then
$\mathfrak h_\epsilon=O(\epsilon^{\alpha_0})$.
\eqref{eq:kappa-sharp} follows by \eqref{expanded:masssystem}. Dividing by $\epsilon^{\alpha_0(N-1)}$ and
using convergence of the centers and boundedness of
$\epsilon^{-\alpha_0(N-1)}\kappa_2$ yields \eqref{eq:CKL59}.
\end{proof}

We can now prove Proposition \ref{propos-6-1}.

\begin{proof}[Proof of Proposition \ref{propos-6-1}]
Fix $l$ and a subsequence on which the local translation limit
\eqref{vlie} holds. Recording the previous some information
\eqref{eq:kappas} and the rescaled profiles from \eqref{tidu} and
\eqref{tidu}; in particular, $\kappa_{i2k}$ is the $k$-th component
of the first-moment vector $\kappa_{i2}$. The change of variables in
$y=\xi_{i\epsilon}+\lambda_{i\epsilon}\epsilon^{\alpha_0}z$ in $\kappa_{i2}$, we have
\begin{equation}\label{eq:closing-first-rescaling}
 \epsilon^{-\alpha_0(N-1)}\kappa_{i2k}
 =\lambda_{i\epsilon}^{N-1}
 \int_{B(0,R_0/(\lambda_{i\epsilon}\epsilon^{\alpha_0}))}
 z_kU_{i\epsilon}^{p-1}V_{li\epsilon}dz.
\end{equation}
Furthermore, we bound
\[
 \bigl|z_kU_{i,\epsilon}^{p-1}
 V_{\ell i\epsilon}\bigr|
 \le C|z|(1+|z|)^{-N-2}\in L^1(\mathbb{R}^N).
\]
Using \eqref{vlie}, Lemma \ref{Lem6.1} and
dominated convergence, we obtain
\[
 \epsilon^{-\alpha_0(N-1)}\kappa_{i2k}
 \longrightarrow
 -\lambda_i^{N-1}\sum_{j=1}^N d_{\ell,i,j}
 \int_{\mathbb{R}^N}z_kW^{p-1}\partial_jWdz.
\]
To evaluate this integral, integration by parts on $B_R=B(0,R)$ gives
\[
 p\int_{B_R}z_kW^{p-1}\partial_jWdz
 =\int_{\partial B_R}z_kW^p\nu_j\,dS
  -\delta_{kj}\int_{B_R}W^pdz.
\]
Since $|z_k|\le R$, $|\nu_j|\le1$, and $W^p=O(R^{-N-2})$
on $\partial B_R$, we deduce
\[
 \left|\int_{\partial B_R}z_kW^p\nu_j\,dS\right|
 \le CR\,R^{-N-2}|\partial B_R|
 \le CR^{-2}\longrightarrow0\quad\mbox{as}\quad R\rightarrow\infty.
\]
Letting $R\to\infty$ and using the definition of $\mathcal{A}_0$ in
\eqref{pusaik-1-1}, we obtain
\begin{align}
 \int_{\mathbb{R}^N}z_kU^{p-1}\partial_jWdz
 &=-\frac{\mathcal{A}_0}{p}\delta_{kj},\label{eq:closing-ibp}\\
 \epsilon^{-\alpha_0(N-1)}\kappa_{i2k}
 &\longrightarrow \frac{\mathcal{A}_0}{p}\lambda_i^{N-1}a_{l,i,k}.
 \label{eq:closing-first-limit}
\end{align}

We next evaluate the interaction term in the mass relation.
By \eqref{eq:R-kappa} and
$\nabla_yH(\xi,\xi)=\tfrac12\nabla\tau(\xi)$, we obtain
\[
\begin{aligned}
 \epsilon^{-\alpha_0(N-1)}\mathcal R_{i\epsilon}(\kappa_2)
 &=-\frac12\sum_{k=1}^N
 \bigl(\epsilon^{-\alpha_0(N-1)}\kappa_{i2k}\bigr)
 \partial_k\phi(\xi_{i\epsilon})
+\sum_{j\ne i}^{m}\sum_{k=1}^N
 \bigl(\epsilon^{-\alpha_0(N-1)}\kappa_{j2k}\bigr)
 \partial_{y_k}G(\xi_{i\epsilon},\xi_{j\epsilon}).
\end{aligned}
\]
By $\xi_{i\epsilon}\rightarrow \xi_{i}$ as $\epsilon\rightarrow0$ and
smoothness of $\tau$ and of $G$ off the diagonal therefore gives
\[
 \begin{aligned}
 \partial_k\tau(\xi_{i\epsilon})
 &\longrightarrow\partial_k\tau(\xi_i),\quad
 \partial_{y_k}G(\xi_{i\epsilon},\xi_{j\epsilon})
 \longrightarrow\partial_{y_k}G(\xi_i,\xi_j),\qquad j\ne i.
 \end{aligned}
\]
We may thus apply \eqref{eq:closing-first-limit}
to each term and obtain
\begin{equation}\label{eq:closing-R-limit}
 \begin{aligned}
 \epsilon^{-\alpha_0(N-1)}\mathcal R_{i\epsilon}(\kappa_2)
 \longrightarrow&\frac{\mathcal{A}_0}{p}\Biggl[
 -\frac12\lambda_i^{N-1}
 \sum_{k=1}^N a_{l,i,k}\partial_k\phi(\xi_i)
 +\sum_{j\ne i}\lambda_j^{N-1}
 \sum_{k=1}^N a_{\ell,j,k}\partial_{y_k}G(\xi_i,\xi_j)
 \Biggr]\\
 &=\frac{\mathcal{A}_0}{p}(P\tilde{\bm{d}}_l)_i.
 \end{aligned}
\end{equation}
For the above last equality, we use \eqref{gX} and \eqref{pusaik-4} to write
\begin{equation}\label{Pl}
\begin{split}
 (P\tilde{\bm{d}}_l)_i
 &=\sum_{j=1}^m\sum_{k=1}^N
 P_{i,(j,k)}\lambda_j^\frac{N-2}{2} a_{l,j,k}
 =-\frac12\lambda_i^{N-1}
 \sum_{k=1}^N a_{l,i,k}\partial_k\tau(\xi_i)
 +\sum_{j\ne i}^{m}\lambda_j^{N-1}
 \sum_{k=1}^N a_{l,j,k}\partial_{y_k}G(\xi_i,\xi_j).
 \end{split}
\end{equation}
The terms with $j=i$ and $j\ne i$ give the Robin and Green terms,
respectively.
Recall from \eqref{eq:R-kappa} that
$\mathcal R_\epsilon(\kappa_2)
=(\mathcal R_{1\epsilon}(\kappa_2),\ldots,
\mathcal R_{m\epsilon}(\kappa_2))^T\in\mathbb{R}^m$,
thus $\mathcal R_{i\epsilon}(\kappa_2)$ is its $i$-th scalar component.
Since \eqref{eq:closing-R-limit} holds for every $i=1,\ldots,m$,
\begin{equation}\label{roue}
 \epsilon^{-\alpha_0(N-1)}\mathcal R_\epsilon(\kappa_2)
 \longrightarrow\frac{\mathcal{A}_0}{p}P\tilde{\bm{d}}_l
 \qquad\text{in }\mathbb{R}^m.
\end{equation}
Using the mass relation from \eqref{eq:kappa-sharp} by $\epsilon^{\alpha_0(N-1)}$ and \eqref{roue}, we conclude that
\begin{equation}\label{eq:closing-zeroth-limit}
 \epsilon^{-\alpha_0(N-1)}\kappa_1
 \longrightarrow \frac{\mathcal{A}_0}{p}N_q^{-1}P\tilde{\bm{d}}_l.
\end{equation}
Finally, let $K\subset\Omega\setminus\{\xi_1,\ldots,\xi_m\}$.
The remainder in \eqref{eq:v-kappa} is
$o((\epsilon^{\alpha_0})^s\mathfrak h_\epsilon)$.
Since $\mathfrak h_\epsilon/(\epsilon^{\alpha_0})=O(1)$ by Lemma~\ref{lem:coefficient},
\[
 \epsilon^{-\alpha_0(N-1)}
 o(\epsilon^{\alpha_0(N-2)}\mathfrak h_\epsilon)
 =\frac{\mathfrak h_\epsilon}{\epsilon^{\alpha_0}}
 o(1)=o(1).
\]
Combining \eqref{eq:v-kappa}, we obtain
\[
\begin{aligned}
 \epsilon^{-\alpha_0(N-1)}v_\epsilon(x)
 =\mu_\epsilon p\sum_{i=1}^m
 \Bigl[&\epsilon^{-\alpha_0(N-1)}\kappa_{i1}G(x,\xi_{i\epsilon})
 +\epsilon^{-\alpha_0(N-1)}\kappa_{i2}\cdot
 \nabla_yG(x,\xi_{i\epsilon})\Bigr]+o(1).
\end{aligned}
\]
For small $\epsilon$, the centers remain a positive distance from
$K$. Smoothness of $G$ off the diagonal implies that both
$G(\cdot,\xi_{i\epsilon})$ and
$\nabla_yG(\cdot,\xi_{i\epsilon})$ converge to their limiting
kernels in $C^1(K)$. Applying \eqref{eq:closing-first-limit} and
\eqref{eq:closing-zeroth-limit}, and combined with $\mu_\epsilon\to1$, we obtain
\begin{equation}\label{eq:closing-exterior-limit}
\begin{aligned}
 \epsilon^{-\alpha_0(N-1)}v_{\ell\epsilon}(x)
 &\longrightarrow
 \mathcal{A}_0\sum_{i=1}^mG(x,\xi_i)(N_q^{-1}P\tilde{\bm{d}}_l)_i
+\mathcal{A}_0\sum_{i=1}^m\lambda_i^{N-1}
       \sum_{k=1}^N a_{l,i,k}\partial_{y_k}G(x,\xi_i)\\
 &=\mathcal{A}_0\bigl[\mathcal{G}(x)N_q^{-1}P
           +\widetilde{\mathcal{G}}(x)\bigr]\tilde{\bm{d}}_l
 \qquad\text{in }C^1(K).
\end{aligned}
\end{equation}
Here the last equality have used \eqref{gX}.
Since $K$ was arbitrary, this proves the asserted
$C^1_{\mathrm{loc}}$ convergence along the chosen subsequence.
\end{proof}

\section{Proof of Theorem \ref{remainder terms}}\label{section-7}
In this section, we are devote to prove Theorem \ref{remainder terms}.
Recalling the symmetric $Nm\times Nm$ matrix $Q$ by $Q_{(i,k),(j,h)}$ in \eqref{pusaik-5}for $i,j\in\{1,\cdots,m\}$ and $k,h\in\{1,\cdots,N\}$.
We can now prove Theorem \ref{remainder terms}.
\begin{proof}[Proof of Theorem \ref{remainder terms}]
Taking any sequence $\epsilon_n\to0$, up to subsequence, we still denoted by $\epsilon$.
We first compute the limit of the boundary integral in \eqref{eq:translationidentity}. Set
\[
\mathcal W(x)=\sum_j\lambda_j^{\frac{N-2}{2}}G(x,\xi_j),\qquad
\mathcal V_l(x)=\bigl[\mathcal{G}(x)N_q^{-1}P+\tilde{\mathcal{G}}(x)\bigr]\tilde{\bm{d}}_l,
\]
and define
\[
\mathscr B_i(f,g)=\int_{\partial B(\xi_i,r_0)}
(\partial_\nu f\,g-f\,\partial_\nu g)\,dS.
\]
The exterior convergences \eqref{AS-eigenvalues-1} and \eqref{exterior} hold uniformly near the spheres $\partial B_{i\epsilon}$. Since their centers converge to $\xi_i$, we have
\[
\epsilon^{-\frac{3N-4}{2}\alpha_0}F_{lik\epsilon}
\longrightarrow \mathcal{A}_0^2\mathscr B_i(\partial_k\mathcal W,\mathcal V_l).
\]
We shall evaluate this limit using the Green identities in \cite[Lemma~6.2]{CKL-1}.
Write $\mathcal{\mathcal{G}}_l(x)=G(x,\xi_l)$ and let $\widetilde{\mathcal{G}}_{l,h}(x)=\lambda_l^{N/2}\partial_{y_h}G(x,\xi_l)$, the $(l,h)$th entry of $\tilde{\mathcal{G}}$. For $l\ne i$, those identities and Green symmetry give
\begin{align*}
\lambda_i\mathscr B_i(\partial_k\mathcal W,\mathcal{G}_l)
&=\lambda_i^{N/2}\partial_{x_k}G(\xi_i,\xi_l)=P_{l,(i,k)},\\
\lambda_i\mathscr B_i(\partial_k\mathcal W,\widetilde{\mathcal{G}}_{l,h})
&=(\lambda_i\lambda_l)^{N/2}\partial_{x_k}\partial_{y_h}G(\xi_i,\xi_l)
=Q_{(i,k),(l,h)}.
\end{align*}
For $l=i$, the first identity contains a regular part from each center. Then we have
\begin{align*}
\lambda_i\mathscr B_i(\partial_k\mathcal W,\mathcal{G}_i)
&=\lambda_i\left(\sum_{j\ne i}\lambda_j^{\frac{N-2}{2}}\partial_{x_k}G(\xi_i,\xi_j)
-\lambda_i^{\frac{N-2}{2}}\partial_k\phi(\xi_i)\right)\\
&=-\tfrac12\lambda_i^{N/2}\partial_k\phi(\xi_i)
=P_{i,(i,k)}.
\end{align*}
The second identity gives
\[
\lambda_i\mathscr B_i(\partial_k\mathcal W,\widetilde{\mathcal{G}}_{i,h})
=\lambda_i^{(N+2)/2}\left(
\sum_{j\ne i}\lambda_j^{\frac{N-2}{2}}\partial_{x_kx_h}G(\xi_i,\xi_j)
-\frac{\lambda_i^{\frac{N-2}{2}}}{2}\partial_{kh}\phi(\xi_i)\right)
=Q_{(i,k),(i,h)}.
\]
Thus, with $\mathscr B_i$ acting on row vectors componentwise,
\begin{equation}\label{eq:rowidentities}
\begin{aligned}
\lambda_i\mathscr B_i(\partial_k\mathcal W,\mathcal{G})=(P^T)_{(i,k)},\quad
\lambda_i\mathscr B_i(\partial_k\mathcal W,\tilde{\mathcal{G}})=Q_{(i,k)}.
\end{aligned}
\end{equation}
By the definition of $\mathcal V_\ell$ and bilinearity,
\[
\lambda_i\mathscr B_i(\partial_k\mathcal W,\mathcal V_l)
=\bigl[(P^T)_{(i,k),\cdot}N_q^{-1}P+Q_{(i,k),\cdot}\bigr]\tilde{\bm{d}}_l
=(B_q\tilde{\bm{d}}_l)_{(i,k)}.
\]
Consequently,
\begin{equation}\label{eq:translationflux}
\lambda_i\epsilon^{-\frac{3N-4}{2}\alpha_0}F_{lik\epsilon}
\longrightarrow \mathcal{A}_0^2(B_q\tilde{\bm{d}}_l)_{(i,k)}.
\end{equation}

We next estimate the eigenvalue $\mu_{l\epsilon}$ for $m+1\leq l\leq (N+1)m$. Let
\[
\hbar_{l\epsilon}=\frac{\mu_{l\epsilon}-1}{\epsilon^{\alpha_0N}}.
\]
Multiplying \eqref{eq:translationidentity} by $\lambda_i\epsilon^{-\frac{3N-4}{2}\alpha_0}$, we obtain
\[
\lambda_i\epsilon^{-\frac{3N-4}{2}\alpha_0}F_{l ik\epsilon}
=\hbar_{l\epsilon}\,\lambda_i\epsilon^{\frac{\alpha_0(4-N)}{2}}I_{l ik\epsilon}.
\]
Combining \eqref{tao0} and \eqref{eq:translationflux}, we have
\[
\mathcal{A}_0^2(B_q\tilde{\bm{d}}_l)_{(i,k)}+o(1)
=\hbar_{l\epsilon}\bigl[-p\mathcal{C}_1(\tilde{\bm{d}}_l)_{(i,k)}+o(1)\bigr].
\]
Choose $(i_0,k_0)$ with $(\tilde{\bm{d}}_l)_{(i_0,k_0)}\ne0$, we deduce that
\[
\hbar_{l\epsilon}
=\frac{\mathcal{A}_0^2(B_q\tilde{\bm{d}}_l)_{(i_0,k_0)}+o(1)}
{-p\mathcal{C}_1(\tilde{\bm{d}}_l)_{(i_0,k_0)}+o(1)}
\longrightarrow\hbar_l\in\mathbb{R}.
\]
In particular, $\hbar_{l\epsilon}$ is bounded. We may therefore pass to the limit in every component, the terms $\eta_{\ell\epsilon}o(1)$ tend to zero. This leads to
\begin{equation}\label{eq:translationmatrix}
B_q\tilde{\bm{d}}_l=-\frac{p\mathcal{C}_1}{\mathcal{A}_0^2}\hbar_l\tilde{\bm{d}}_l.
\end{equation}
Up to now, we have established that each shift is associated with an eigenvalue of $B$. The multiplicity and order of this eigenvalue remain to be determined.
For this end, we exploit the $v_\epsilon$-orthogonality of the selected eigenfunctions and carry this orthogonality over to their local limits. In $B_{i\epsilon}$, using the change of variables $x=\xi_{i\epsilon}+\lambda_{i\epsilon}\epsilon^{\alpha_0}z$ and \eqref{vlie}, we deduce
\begin{align*}
\lim_{\epsilon\to0}\epsilon^{-\alpha_0(N-2)}\int_{B_{i\epsilon}}
pu_\epsilon^{p-1}v_{l\epsilon}v_{l^{\prime}\epsilon}
&=p\lambda_i^{N-2}\sum_{k,h=1}^N a_{l,i,k}a_{l^{\prime},i,h}
\int_{\mathbb{R}^N}W^{p-1}\partial_kW\,\partial_hW\\
&=p\mathcal{C}_1\lambda_i^{N-2}\sum_{k=1}^N a_{l,i,k}a_{l^{\prime},i,k}.
\end{align*}
Indeed, \eqref{tidu} and Lemma \ref{3-AS-eigenvalues-1} give the integrable bound $C(1+|z|)^{-2N}$ for the rescaled integrand and radial symmetry gives
$\int W^{p-1}\partial_kW\,\partial_hW=\mathcal{C}_1\delta_{k}^{h}$.
The same bounds control the remaining terms:
\begin{align*}
\epsilon^{-\alpha_0(N-2)}\int_{\Omega\setminus\bigcup_i B_{i\epsilon}}
pu_\epsilon^{p-1}|v_{l\epsilon}v_{l^{\prime}\epsilon}|
&=O(\epsilon^{\alpha_0N}),\\
\epsilon^{-\alpha_0(N-2)}\epsilon q\int_\Omega u_\epsilon^{q-1}|v_{l\epsilon}v_{l^{\prime}\epsilon}|
&\le C\epsilon\epsilon^{\alpha_0\tau_0(p-q)}=C\epsilon^{\alpha_0(N-2)}.
\end{align*}
The second estimate uses $2\tau_0(q+1)>N$ from \eqref{beta1}. Both terms tend to zero. Summing over the bubbles and using the definition \eqref{gX}, we obtain, for $\ell\ne\ell'$,
\begin{equation}\label{eq:translationorthogonality}
0=\lim_{\epsilon\to0}\epsilon^{-\alpha_0(N-2)}
\int_\Omega u_\epsilon v_{l\epsilon}v_{l^{\prime}\epsilon}
=p\mathcal{C}_1\,\tilde{\bm{d}}_l\cdot\tilde{\bm{d}}_{l^{\prime}}.
\end{equation}
The $Nm$ vectors $\tilde{\bm{d}}_l$ are nonzero and mutually orthogonal in $\mathbb{R}^{Nm}$,
whence they constitute a basis. By virtue of \eqref{eq:translationmatrix}, the quantities $-p\mathcal{C}_1\hbar_l/\mathcal{A}_0^2$ are precisely all eigenvalues of $B_q$, counted with multiplicity. Since the $\mu_{l\epsilon}$ are increasing, their normalized limits satisfy
$\hbar_{m+1}\le\cdots\le\hbar_{(N+1)m}$. Multiplying by $-p\mathcal{C}_1/\mathcal{A}_0^2$ reverses this ordering, so that
\[
-\frac{p\mathcal{C}_1}{\mathcal{A}_0^2}\hbar_{m+j}=\nu_{j}^2,
\qquad B_q\tilde{\bm{d}}_{m+j}=\nu_{j}^2\tilde{\bm{d}}_{m+j},
\qquad j=1,\ldots,Nm.
\]
These limits are identical for any initial sequence. To verify convergence for the full family, assume that some normalized shift remains at a positive distance from its claimed limit along a certain subsequence. Applying the preceding argument to this sequence yields a subsequence converging to the claimed limit, which is a contradiction.
This establishes identity \eqref{eq:translationresult} together with the assertions concerning the coefficient vectors.
\end{proof}

\section{Proof of Theorem \ref{thmprtb}}\label{section-8}

This section is devoted to proving Theorem \ref{thmprtb}, and its proof relies on the following two lemmas.
\begin{lem}\label{lem:dilationcluster}
For $(N+1)m+1\le l\le(N+2)m$, we have
\begin{equation}\label{eq:dilationupper}
\mu_{l\epsilon}\to1,\qquad
\mu_{l\epsilon}\le1+C\epsilon^{\alpha_0(N-2)}.
\end{equation}
Up to subsequences,
\begin{equation}\label{eq:dilationlocal}
V_{\ell i\epsilon}\to d_{\ell,i}(\partial_{\lambda}W)
\quad\text{in }C^1_{\mathrm{loc}}(\mathbb{R}^N),
\qquad \tilde{\bm{f}}_l=\Big(\lambda_{1}^{\frac{N-2}{2}}a_{l,1},\cdots,\lambda_{m}^{\frac{N-2}{2}}a_{l,m}\Big)^T\in\mathbb{R}^m\ne0.
\end{equation}
\end{lem}

\begin{proof}
We first define Rayleigh
$$
R_{\epsilon}(f):=\frac{\int_{\Omega}|\nabla f|^2dx}{\int_\Omega \big(pu_{\epsilon}^{p-1}+\epsilon qu_{\epsilon}^{q-1}\big)f^2}.
$$
Write $L_\epsilon=-\Delta-(pu_{\epsilon}^{p-1}+\epsilon qu_{\epsilon}^{q-1})$ and
$
T_{ik\epsilon}=\lambda_{i\epsilon}\epsilon^{\alpha_0}\partial_k u_\epsilon$
, and recalling $w_{i\epsilon}=(x-\xi_{i\epsilon})\cdot\nabla u_\epsilon+\frac{N-2}{2}u_\epsilon$, we have
\begin{equation}\label{eq:fieldidentities}
\begin{gathered}
L_\epsilon T_{ik\epsilon}=0,\qquad L_\epsilon w_{i\epsilon}=\epsilon\tau_0(p-q) u_\epsilon^q,\quad
L_\epsilon u_\epsilon=-(p-1)u_\epsilon^p-\epsilon(q-1)u_\epsilon^q.
\end{gathered}
\end{equation}
Choose smooth cutoff functions $\chi_{i\epsilon}\subset B_{i\epsilon}$ and $\chi_{i\epsilon}\equiv1$ on $B(\xi_{i\epsilon},r_0/2)$, and all derivatives are uniformly bounded.
The space
\[
\mathcal V_\epsilon=span\{\chi_{i\epsilon}u_\epsilon,
\chi_{i\epsilon}T_{ik\epsilon},\chi_{i\epsilon}w_{i\epsilon}:i=1,\ldots,m,\ k=1,\ldots,N\}
\]
is of dimension $(N+2)m$. Indeed, each weighted Gram block converges to the positive definite Gram matrix associated with $W,\partial_1W,\ldots,\partial_NW,-(z\cdot\nabla W+\frac{N-2}{2}W)$ in the space $L^2(pW^{p-1}dz)$. For
\[
f=\sum_i\chi_{i\epsilon}(\alpha_i u_\epsilon+K_i),
\qquad K_i=\sum_k b_{ik}T_{ik\epsilon}+c_iw_{i\epsilon},
\]
we have
\begin{equation}\label{eq:gram}
D_\epsilon(f,f):=\int_\Omega \big(pu_{\epsilon}^{p-1}+\epsilon qu_{\epsilon}^{q-1}\big)f^2
\ge c\sum_i(\alpha_i^2+|b_i|^2+c_i^2).
\end{equation}
Define $\mathcal B_\epsilon(f,g)=\int\nabla f\cdot\nabla g-\int(pu_{\epsilon}^{p-1}+\epsilon qu_{\epsilon}^{q-1})fg$. Integration by parts gives
\begin{equation}\label{eq:cutoffidentity}
\begin{split}
\mathcal B_\epsilon(\chi f,\chi g)
={}&\int\chi^2fL_\epsilon g
+\int\chi\nabla\chi\cdot(g\nabla f-f\nabla g)
+\int|\nabla\chi|^2fg.
\end{split}
\end{equation}
On the cutoff annuli, \eqref{AS-eigenvalues-1} gives
$u_\epsilon,w_{i\epsilon},\nabla u_\epsilon,\nabla w_{i\epsilon}=O(\epsilon^{\alpha_0\tau_0})$ and
$T_{ik\epsilon},\nabla T_{ik\epsilon}=O(\epsilon^{\alpha_0(\tau_0+1)})$. Also,
\[
\epsilon\int_{B_{i\epsilon}}u_\epsilon^{q+1}=O(\epsilon^{\alpha_0(N-2)}),\qquad
\epsilon\int_{B_{i\epsilon}}u_\epsilon^q|K_i|
\le C\epsilon^{\alpha_0(N-2)}(|b_i|+|c_i|),
\]
by \eqref{tidu} and $(N-2)(q+1)>N$. Using \eqref{eq:fieldidentities}, \eqref{eq:gram} and \eqref{eq:cutoffidentity}, we obtain
\begin{align*}
\mathcal B_\epsilon(\chi_i u_\epsilon,\chi_i u_\epsilon)&\le-c_0,\\
|\mathcal B_\epsilon(\chi_i u_\epsilon,\chi_i K_i)|
&\le C\epsilon^{\alpha_0(N-2)}(|b_i|+|c_i|),\\
|\mathcal B_\epsilon(\chi_i K_i,\chi_i K_i)|
&\le C\epsilon^{\alpha_0(N-2)}(|b_i|^2+c_i^2),
\end{align*}
where $\chi_i=\chi_{i\epsilon}$. In particular, the amplitude--dilation cross term is controlled by the exact equation for $w_{i\epsilon}$:
\[
\mathcal B_\epsilon(\chi_i u_\epsilon,\chi_iw_{i\epsilon})
=\epsilon\tau_0(p-q)\int\chi_i^2u_\epsilon^{q+1}+O(\epsilon^{\alpha_0(N-2)})=O(\epsilon^{\alpha_0(N-2)}).
\]
Young's inequality and disjointness of the supports yield
\[
\mathcal B_\epsilon(f,f)
\le-\frac{c_0}{2}\sum_i\alpha_i^2
+C\epsilon^{\alpha_0(N-2)}\sum_i(|b_i|^2+c_i^2).
\]
Together with \eqref{eq:gram}, the min--max principle gives
$\mu_{(N+2)m,\epsilon}\leq\sup\limits_{0\neq f\in\mathcal{V}_\epsilon}R_{\epsilon}(f)\le1+C\epsilon^{\alpha_0(N-2)}$.
The lower limit follows from the ordering and
$\mu_{(N+1)m,\epsilon}\to1$.

We notice that the local limit equation $-\Delta V=pW^{p-1}V$ and the bubble nondegeneracy in \cite[Lemma~2.10]{CKL} gives
\[
V_{\ell i}=-\sum_k b_{l,i,k}\partial_kU+a_{l,i}(\partial_{\lambda}W).
\]
The $L^\infty$ normalization and \eqref{V5} imply that not all these coefficients vanish.
For every translation eigenfunction, weighted orthogonality, \eqref{eq:translationorthogonality}, and
$\int W^{p-1}Z_0\partial_kW=0$ give
\[
(\lambda_i^{\frac{N-2}{2}}b_{l,i,k})_{i,k}\cdot\tilde{\bm{d}}_j=0,
\qquad m+1\le j\le(N+1)m.
\]
The vectors $\tilde{\bm{d}}_j$ form a basis by Theorem~\ref{remainder terms}, hence all $a_{l,i,k}$ vanish. This proves \eqref{eq:dilationlocal}.
\end{proof}

\begin{lem}\label{thm:dilation}
Let $0<\nu_{(N+1)m+1}^{3}\le\cdots\le\nu_{(N+2)m}^3$ be the eigenvalues of $H_q$, counted with multiplicity. Then
\begin{equation}\label{eq:dilationresult}
\mu_{l\epsilon}
=1+\frac{\nu_l^3}{pL_0}\epsilon^{\alpha_0(N-2)}+o(\epsilon^{\alpha_0(N-2)}),
\qquad (N+1)m+1\leq l\leq(N+2)m.
\end{equation}
The nonzero vectors $\tilde{\bm{f}}_{(N+1)m+j}$ in \eqref{eq:dilationlocal} are corresponding eigenvectors of $H_q$ and form an orthogonal basis of $\mathbb{R}^m$ along every simultaneous convergent subsequence. In particular, all eigenvalues in this cluster exceed $1$ for sufficiently small $\epsilon$.
\end{lem}

\begin{proof}
Recalling some identities in Lemma \ref{m00}, Lemma \ref{aepu} and Lemma \ref{limti} by using Lemma \ref{lem:dilationcluster}.
Choose $i$ with $\tilde{\bm{f}}_i\ne0$, by \eqref{519},
\[
\eta_{l\epsilon}:=\frac{\mu_{l\epsilon}-1}{\epsilon^{\alpha_0(N-2)}}
=\frac{
    \mathcal E_{i_0\epsilon}
   +\mathcal G_{i_0\epsilon}}
 {\mathcal F_{i_0\epsilon}}
\to\eta_l.
\]
Taking the limit, we conclude that
\begin{equation}\label{eq:scalematrix}
H_q\tilde{\bm{f}}_l=pL_0\eta_\ell\tilde{\bm{f}}_l.
\end{equation}
Moreover,
\[
0=\lim_{\epsilon\to0}\epsilon^{-\alpha_0(N-2)}
\int_\Omega \big(pu_{\epsilon}^{p-1}+\epsilon qu_{\epsilon}^{q-1}\big)v_{l\epsilon}v_{l'\epsilon}
=pC_4\tilde{\bm{f}}_l\cdot\tilde{\bm{f}}_{l'}\quad\mbox{for}\quad l\ne l'.
\]
The $m$ nonzero orthogonal vectors exhaust the spectrum of $H_q$. Its positivity follows from the previously proved identity \eqref{eq:H-positive}.
Ordering and the same subsequence argument as in Theorem \ref{remainder terms} yield \eqref{eq:dilationresult} for the full family.
\end{proof}

We final explain the relationship the reduced functional with two matrixs.
\begin{Rem}\label{r9-3}
The matrices governing the translation and dilation spectra in
Theorems~\ref{remainder terms} and~\ref{thmprtb} admit a common interpretation in terms of the
reduced functional. In the logarithmic scale variables
$h_i=\log\lambda_i$, write
$t=(e^{ah_1},\ldots,e^{ah_m})^T$ and set
\[
    \Phi_{\Omega,q}(h,\xi)
    =
    \frac{\mathcal{A}_0^2}{2}\,t^T M(\xi)t
    -
    \mathcal{B}_0\sum_{i=1}^{m}e^{\tau_0(p-q)h_i}.
\]
All derivatives below are evaluated at the limiting configuration
satisfying the identities~\eqref{deference}. Direct differentiation gives
$D_{hh}^2\Phi_{\Omega,q}=H_q,
$
which is positive definite by~\eqref{eq:H-positive}. Thus, the matrix governing the
dilation spectrum is precisely the Hessian of $\Phi_q$ with respect
to the logarithmic scale variables.
To identify the corresponding structure of the translation matrix, put
\[
    \mathcal C=D_{h\xi}^2\Phi_{\Omega,q},
    \qquad
    \mathcal R=D_{\xi\xi}^2\Phi_{\Omega,q},
    \qquad
    D_\lambda
    =
    \operatorname{diag}(\lambda_1I_N,\ldots,\lambda_mI_N).
\]
The definitions of $P$ and $Q$, together with~(1.14), yield
\[
    \mathcal C D_\lambda=-\mathcal{A}_0^2aD_tP,
    \qquad
    D_\lambda\mathcal R D_\lambda=-\mathcal{A}_0^2Q.
\]
Using $H_q=\mathcal{A}_0^2\tau_0^2D_tN_qD_t$, we therefore obtain
\begin{equation}\label{eq:bn-spectral-schur}
    D_\lambda
    \left(
        \mathcal R
        -
        \mathcal C^TH_q^{-1}\mathcal C
    \right)
    D_\lambda
    =
    -\mathcal{A}_0^2B_q.
\end{equation}
Hence, up to the diagonal change of variables and the factor $-\mathcal{A}_0^2$,
the translation matrix is the Schur complement of the scale block
$H_q$ in the full Hessian of $\Phi_{\Omega,q}$. In particular, the
term $P^TN_q^{-1}P$ arises from eliminating the scale
variables and accounts for their coupling with the concentration
points. This is an algebraic identification of the matrices already obtained
from the local identities; it does not rely on a $C^2$ expansion of
an $\epsilon$-dependent reduced functional.
\end{Rem}

\appendix

\section{A sharp decay estimate of the eigenfunctions} \label{sharp}
In this appendix, we will prove Lemma \ref{3-AS-eigenvalues-1}. Before the proof, we recall some notations in \eqref{XX}
and define
\[
U_\epsilon(x)=\epsilon^{\alpha_0\tau_0} u_\epsilon(\epsilon^{\alpha_0}x),\qquad
r_i=|x-X_{i\epsilon}|,\qquad d_\epsilon(x)=\min_i r_i,
\qquad \langle z\rangle=(1+|z|^2)^{1/2}.
\]


We can now prove Lemma \ref{3-AS-eigenvalues-1}.
\begin{proof}[\textbf{Proof of Lemma \ref{3-AS-eigenvalues-1}}]
For fixed $R_{\ast}, c_{\ast}$, we obtain
\begin{equation}\label{WW1}
\epsilon^{\frac{\alpha_0(N-2)}{2}}u_{\epsilon}(\xi_{1\epsilon}+\epsilon^{\alpha_0}z)\geq c_{\ast}, \quad \mbox{for}\quad |z|\leq R_{\ast}.
\end{equation}
Choose an $l$-dimensional subspace $E\subset C_c^\infty(B(0,R_*))$ and set $\phi_\epsilon(x)=\psi((x-\xi_{1\epsilon})/(\epsilon^{\alpha_0}))$ for $\psi\in E$.
Its support is contained in $\Omega$ for small $\epsilon$. Since $\tau_0(p-1)=2$, \eqref{WW1} gives
\[
\int_\Omega|\nabla\phi_\epsilon|^2=(\epsilon^{\alpha_0})^s\int|\nabla\psi|^2,
\qquad
\int_\Omega\bigl(pu_\epsilon^{p-1}+\epsilon q u_\epsilon^{q-1}\bigr)\phi_\epsilon^2
\ge p c_*^{p-1}(\epsilon^{\alpha_0})^s\int\psi^2.
\]
Norm equivalence on $E$ and the min--max characterization therefore imply
\begin{equation}\label{eq:mu}
0<\mu_{l\epsilon}\le C_l.
\end{equation}
By \eqref{XX}, we notice that $V_{l\epsilon}=0$ on $\partial \Omega_\epsilon$, $\|V_{l\epsilon}\|_\infty=1$, and
\begin{equation}\label{eq:scaled}
-\Delta V_{l\epsilon}=c_\epsilon V_{l\epsilon},
\qquad
c_\epsilon=\mu_{\ell\epsilon}
\bigl(pU_\epsilon^{p-1}+q\epsilon\cdot\epsilon^{\alpha_0\tau_0(p-q)}U_\epsilon^{q-1}\bigr)\ge0.
\end{equation}
By \eqref{tidu}, rescaled functions satidfy
\[
U_\epsilon(x)\le C\sum_i\langle x-X_{i\epsilon}\rangle^{-s}.
\]
Hence, on $d_\epsilon\ge1$, using $\epsilon\rho_{\epsilon}^{\beta}\leq C\rho_{\epsilon}^{s}$, \eqref{eq:mu}, and $(N-2)(p-1)=4$, we obtain
\[
c_\epsilon(x)\le C\bigl(d_\epsilon^{-4}
+\epsilon^{\alpha_0(N-2)}d_\epsilon^{-(N-2)(q-1)}\bigr).
\]
Since $d_\epsilon(x)\le\operatorname{diam}(\Omega)/(\epsilon^{\alpha_0})$, it follows from \eqref{tidu} that
\begin{equation}\label{wise}
c_\epsilon(x)\le C d_\epsilon(x)^{-\sigma},
\qquad \sigma:=\sigma_q=\min\{4,(N-2)q\}>2,
\qquad d_\epsilon(x)\ge1.
\end{equation}
For $R\ge1$, we denote
\[
\mathcal E_{R}=\Omega_\epsilon\setminus
\bigcup_{i=1}^m\overline{B(X_{i\epsilon},R)}.
\]
Then we have
\begin{equation}\label{eq:critical-tail}
\begin{aligned}
\int_{\mathcal E_{R}}c_\epsilon^{\frac{N}{2}}\,dX
&\le C\sum_i\int_{|x-X_{i\epsilon}|>R}
|x-X_{i\epsilon}|^{-\sigma N/2}\,dx\\
&\le C R^{-\frac{N(\sigma-2)}{2}}.
\end{aligned}
\end{equation}
Here $\sigma>2$ makes the radial integrals convergent, and the constant is independent of $R$ and $\epsilon$.
Choosing $R$ sufficiently, the right-hand side tends to zero.
The iteration below uses the pointwise bound \eqref{wise} directly.

Set
\begin{equation}\label{eq:parameters}
\zeta=\frac{\sigma-2}{3},\qquad T=(N-2)-\zeta>0,
\qquad J=\left\lceil\frac{T}{\zeta}\right\rceil,
\qquad t_j=\min\{j\zeta,T\}\quad\mbox{for}\hspace{2mm}0\le j\le J,
\end{equation}
where $\lceil x\rceil$ denotes the smallest integer greater than or equal to $x$.
Therefore, $t_0=0$, $t_J=T$, and
\[
0<t_j-t_{j-1}\le\zeta<\sigma-2,
\qquad 0<t_j<N-2\quad\mbox{for}\hspace{2mm}1\le j\le J.
\]
Fix $R\ge2$ so large that $R^{-\zeta}\le1/2$. For $\epsilon$ small enough, \eqref{en} ensures that the closed balls $\overline{B(X_{i\epsilon},R)}$ are disjoint and contained in $\Omega_\epsilon$.
We next aim to prove that
\begin{equation}\label{eq:induction}
|V_{l\epsilon}(x)|\le C_j\sum_i r_i^{-t_j}
\qquad\text{on }\mathcal E_{R}.
\end{equation}
The iteration starts from the normalization \(\|V_{l\epsilon}\|_\infty=1\), which yields \eqref{eq:induction} for \(j=0\) with \(C_0=1\). The nearest-center inequalities give
\[
d_\epsilon^{-\sigma}\sum_i r_i^{-t_{j-1}}
\le m d_\epsilon^{-\sigma-t_{j-1}}
\le m\sum_i r_i^{-\sigma-t_{j-1}}.
\]
Consequently, choosing $\eta_j=t_j-t_{j-1}$, then \eqref{wise} tell us that
\begin{equation}\label{eq:source-step}
c_\epsilon|V_{l\epsilon}|
\le K_0 C_{j-1}\sum_i r_i^{-t_{j-1}-\sigma}
\le K_0 C_{j-1}R^{-(\sigma-2-\eta_j)}
\sum_i r_i^{-t_j-2},
\end{equation}
where $K_0$ is a constant, independent of $j$ and $\epsilon$. Define the function
\[
\Psi_j(x)=C_j\sum_i r_i^{-t_j}.
\]
Since $0<t_j<N-2$, we have
\[
-\Delta\Psi_j=C_j t_j(N-2-t_j)\sum_i r_i^{-t_j-2}.
\]
Fix $R$, choose $C_j>0$, independently of $\epsilon$, sufficiently large such that
\begin{equation}\label{eq:constants}
C_j\ge R^{t_j},
\qquad
C_j t_j(N-2-t_j)
\ge K_0 C_{j-1}R^{-(\sigma-2-\eta_j)}.
\end{equation}
By the choice of \(C_j\), we have
$
-\Delta\Psi_j\ge c_\epsilon|V_{l\epsilon}|
\qquad\text{in }\mathcal E_{R}.
$
Moreover, \(\Psi_j\ge1\ge|V_{l\epsilon}|\) on each inner boundary component, while \(V_{l\epsilon}=0\) and \(\Psi_j>0\) on \(\partial \Omega_\varepsilon\). Hence,
$$
\begin{cases}
-\Delta(\Psi_j\pm V_{l\epsilon})
=-\Delta\Psi_j\pm c_\epsilon V_{l\epsilon}\ge0
&\text{in }\mathcal E_{R},\\
\Psi_j\pm V_{l\epsilon}\ge0
&\text{on }\partial\mathcal E_{R}.
\end{cases}
$$
The maximum principle therefore gives
$$
|V_{l\epsilon}|\le\Psi_j
=C_j\sum_{i=1}^{m}r_i^{-t_j},
$$
which completes the induction and proves \eqref{eq:induction}. Since \(J\) is finite and independent of \(\varepsilon\), and \(t_J=N-2-\zeta\), we conclude that
\begin{equation}\label{eq:almost}
|V_{l\epsilon}(x)|
\le C_J\sum_{i=1}^{m}r_i^{-(N-2-\zeta)}
\qquad\text{in }\mathcal E_{R},
\end{equation}
where \(C_J\) is independent of \(\epsilon\).
Combining \eqref{wise} and \eqref{eq:almost} as in \eqref{eq:source-step}, and using $\sigma-2=3\zeta$, yields
\begin{equation}\label{eq:source-final}
c_\epsilon|V_{l\epsilon}|\le C_*\sum_i r_i^{-N-2\zeta}
\qquad\text{on }\mathcal E_{R}.
\end{equation}

We next define the corrected barrier function
\begin{equation}\label{eq:barrier}
\Phi_\epsilon(x)=K_1\sum_i\bigl(r_i^{-(N-2)}-r_i^{-(N-2)-\zeta}\bigr).
\end{equation}
where $K_1$ is a constant and independent of $j$ and $\epsilon$.
Our fixed choice of $R$ gives
\begin{equation}\label{eq:positive-barrier}
\frac{K_1}{2}\sum_i r_i^{-(N-2)}\le\Phi_\epsilon(x)
\le K_1\sum_i r_i^{-(N-2)}
\qquad\text{on }\mathcal E_{R}.
\end{equation}
We note that each function \(r_i^{-s}\) is harmonic away from \(X_{i\varepsilon}\). Direct computation therefore yields
\begin{equation}\label{eq:laplace-barrier}
-\Delta\Phi_\varepsilon
=
K\zeta(N-2+\zeta)\sum_{i=1}^{m}r_i^{-N-\zeta}
\qquad\text{in }\mathcal E_{R}.
\end{equation}
Since $r_i\ge R$, the right-hand side of \eqref{eq:source-final} is controlled by
$C_*R^{-\zeta}\sum_i r_i^{-N-\zeta}$. Choose
\[
K_1\ge\max\left\{2R^{N-2},
\frac{C_*R^{-\zeta}}{\zeta(N-2+\zeta)}\right\}.
\]
Then $-\Delta\Phi_\epsilon\ge c_\epsilon|V_{l\epsilon}|$, and \eqref{eq:positive-barrier} ensures $\Phi_\epsilon\ge1\ge|V_{l\epsilon}|$ on each inner sphere. On $\partial \Omega_\epsilon$, $V_{l\epsilon}=0$ and $\Phi_\epsilon>0$.
Applying the comparison principle to $\Phi_\epsilon\pm V_{l\epsilon}$ gives
\[
|V_{l\epsilon}(x)|\le\Phi_\epsilon(x)\le K_1\sum_i r_i^{-(N-2)}
\qquad\text{on }\mathcal E_{R}.
\]
Since \(|V_{l\epsilon}|\le 1\) and
$$
\sum_{i=1}^{m}\langle x-X_{i\epsilon}\rangle^{-(N-2)}
\ge (1+R^2)^{-\frac{N-2}{2}}
$$
inside the core balls, the exterior estimate extends to
\begin{equation}\label{eq:global-scaled}
|V_{l\epsilon}(X)|
\le C_l\sum_{i=1}^{m}\langle x-X_{i\epsilon}\rangle^{-(N-2)},
\qquad x\in \Omega_\varepsilon,
\end{equation}
after enlarging \(C_l\) if necessary.
Returning to $x=\epsilon^{\alpha_0}x$ first gives \eqref{V5}
and using the uniform comparability in \eqref{en}, we obtain \eqref{V5}.

Fix $0<r_0<c_0/3$,
by (8), we obtain
\begin{equation}\label{eq:local}
|V_{li\epsilon}(z)|
\le C_\ell(1+|z|^2)^{-(N-2)/2},
\qquad |z|\le\frac{r_0}{\lambda_{i\epsilon}\epsilon^{\alpha_0}}.
\end{equation}
The restriction to this neighborhood excludes the other concentration points, near which a decay bound centered solely at the $i$-th bubble need not hold.
Thus, for every fixed $r>0$, we obtain \eqref{V8}.
\end{proof}

We next state the optimality of the common decay exponent.
\begin{Rem}\label{rem:sharpness}
Under assumptions Lemma \ref{3-AS-eigenvalues-1}, the exponent
$a=:(N-2)/2$ in \eqref{V5} is optimal among estimates
valid for every fixed $l\in\mathbb N$, with constants allowed
to depend on $l$ but independent of $\epsilon$.

Indeed, consider the positive first eigenfunction normalized by
$\|v_{1\epsilon}\|_{\infty}=1$.
Estimate \eqref{V5} places a maximum point within
$O(\epsilon^{\alpha_0})$ of some concentration center.
The local Harnack inequality, followed by harmonic comparison
on a surrounding annulus, yields points $x_\epsilon$ staying
a fixed positive distance from all concentration centers such that
\[
v_{1\epsilon}(x_\epsilon)
\ge c(\epsilon^{\alpha_0(N-2)}).
\]
If $a$ in \eqref{V5} could be replaced by $a+\eta$,
for some $\eta>0$, then scale comparability would give
\[
c(\epsilon^{\alpha_0(N-2)})
\le v_{1\epsilon}(x_\epsilon)
\le C(\epsilon^{\alpha_0(N-2+2\eta)}),
\]
a contradiction as $\epsilon\to0$.
This establishes optimality of the common exponent, which does not
assert optimality for each individual eigenfunction branch.
\end{Rem}

	\vspace{3mm}
\noindent \textbf{Acknowledgments:}
Minbo Yang was partially supported by the National Natural Science Foundation of China (No.~12471114) and the Natural Science Foundation of Zhejiang Province (No.~LZ26A010002). Shunneng Zhao was partially supported by the National Natural Science Foundation of China (No.~12401146) and the Natural Science Foundation of Zhejiang Province (No.~LMS25A010007). The authors are grateful to Professor Wenming Zou of Tsinghua University for his valuable suggestions.

	\vspace{3mm}
\noindent \textbf{AI assistance statement:} The authors partly used AI tools to assist with exploratory investigation; all mathematical validation, final proof decisions, and final wording remain the sole responsibility of the human authors.

\end{document}